\documentclass[11pt,reqno]{amsart}

\usepackage{amsmath, amsfonts, amsthm, amssymb, stmaryrd, color, cite}

\usepackage{hyperref}
\newtheorem{theorem}{Theorem}[section]
\newtheorem{lemma}{Lemma}[section]
\newtheorem{proposition}{Proposition}[section]

\theoremstyle{definition}
\newtheorem{definition}{Definition}[section]

\theoremstyle{remark}
\newtheorem{remark}{Remark}[section]

\numberwithin{equation}{section}
\allowdisplaybreaks

\def\f{\frac}

\def\hf1{^\f{1}{1-\xi^2}}

\def\be{\begin{equation}}
\def\en{\end{equation}}
\def\bs{\begin{split}}
\def\es{\end{split}}
\def\ba{\begin{align}}
\def\ea{\end{align}}

\title[Martingale weak solutions to stochastic CH-NS equations]
{Stochastic Cahn-Hilliard-Navier-Stokes equations with the dynamic boundary: Martingale weak solution, Markov selection}
\author[H. Gao]{Hongjun Gao}
\address{School of Mathematics, Southeast University, Nanjing, 211189, China.}
\email{hjgao@seu.edu.cn}

\author[Z. Qiu]{Zhaoyang Qiu}
\address{School of Applied Mathematics, Nanjing University of Finance and Economics, Nanjing, 210046, China.}
\email{zhqmath@163.com}

\author[H. Wang]{Huaqiao Wang}
\address{College of Mathematics and Statistics, Chongqing University, Chongqing, 401331, China.}
\email{wanghuaqiao@cqu.edu.cn}

\keywords{Stochastic Cahn-Hilliard-Navier-Stokes equations, martingale weak solution, the dynamic boundary, martingale problem, Markov selection.}
\subjclass[2010]{35Q35, 76D05, 35R60, 60F10}

\date{\today}

\begin{document}
\begin{abstract}
The existence of global martingale weak solution for the 2D and 3D stochastic Cahn-Hilliard-Navier-Stokes equations driven by multiplicative noise in a smooth bounded domain is established. In particular, the system is supplied with the dynamic boundary condition which accounts for the interaction between the fluid components and the rigid walls. The proof is completed by a three-level approximate scheme combining a fixed point argument and the stochastic compactness argument, overcoming challenges from strong nonlinearity, dynamic boundary and random effect. Then, we prove the existence of an almost surely Markov selection to the associated martingale problem following the abstract framework established by F. Flandoli and M. Romito \cite{F2}.
\end{abstract}

\maketitle
\section{Introduction}

Cahn-Hilliard-Navier-Stokes (CH-NS) equations describe the motion of two immiscible mixture fluid, which consist of the Navier-Stokes equations governing the fluid velocity and a convective Cahn-Hilliard equation for the relative density of atoms of one of fluids, for further backgrounds see \cite{GG,Gurtin, Heida, Blesgen} and references therein. The motion of fluid could be easily affected by external random factor which accounts for numerical and physical uncertainties in applications ranging from climatology to turbulence theory. Therefore, the uncertainty and randomness must be taken into account for better understanding the dynamical phenomenon. Here we consider the stochastic CH-NS system with multiplicative noise in a smooth bounded domain $\mathcal{D}$:
\begin{eqnarray}\label{Equ1.1}
\left\{\begin{array}{ll}
\!\!du-({\rm div}(2\nu_0D(u))-(u\cdot \nabla)u-\nabla p+\mu\nabla \phi)dt=h(u,\nabla\phi)d\mathcal{W},\\
\!\!d\phi+(u\cdot\nabla)\phi dt-\Delta \mu dt=0,\\
\!\!\mu=-\nu_1\Delta \phi+ \alpha f(\phi),\\
\!\!\nabla\cdot u=0,\\
\end{array}\right.
\end{eqnarray}
where $u, p$ and $\phi$ denote velocity, pressure and phase parameter, respectively. $D(u)=\frac{\nabla u+\nabla u^{\bot}}{2}$ is the viscous stress tensor, and $\nu_0$ is the viscosity of the fluid. Here we assume that $\nu_0\equiv1$. $\mathcal{W}$ is a cylindrical Wiener process introduced later. $\mu$ stands for the chemical potential of binary mixtures which is derived from the variational derivative of the quantity (free energy functional):
\begin{eqnarray*}
E^f(\phi)=\int_{\mathcal{D}}\frac{\nu_{1}}{2}|\nabla\phi|^{2}+\alpha F(\phi)dx,
\end{eqnarray*}
where $F(\xi)=\int_{0}^{\xi}f(r)dr$, constant $\nu_{1}$ is related to the thickness of the interface between two fluids and parameter $\alpha=\frac{1}{\nu_1}$ describes the interaction between two phases, we also assume that $\alpha=\nu_1\equiv 1$. In physical background, a representative example of $F$ is logarithmic type, that is,
$$F(s)=c_0[(1+s)\ln (1+s)+(1-s)\ln(1-s)]-c_1s^2,\; c_1>c_0>0,$$
where $s\in (-1,1)$. Usually, we use a polynomial approximation of type $F(r)=r^4-r^2$ taking place the type of logarithmic.

For the case of the Neumann boundary condition of $\phi$, abundant results of system (\ref{Equ1.1}) were published in both deterministic and stochastic areas, see \cite{Mejdo, Mejdo2,Mejdo3, FY,Feireisl,FGG,FGK,FRS,GG,You,gio,yang} and references therein. Nevertheless, the Neumann boundary condition is incompatible with the moving contact line which describes the intersection of the incompressible two-phase flow interface with the solid wall. Therefore, a new boundary condition was proposed by \cite{qian} (see also \cite{qian2}) accounting for the physical phenomenon. Specifically,  the velocity equations are equipped with the following generalized Navier boundary conditions which was introduced by Maxwell \cite{max} in the kinetic theory of gases:
\begin{align}
&u\cdot n=0,\label{1.2}\\
&2(D(u)\cdot n)_{\tau}+ u_{\tau}=\mathcal{K}(\phi)\nabla_\tau \phi,\label{13}
\end{align}
where $n$ is the unit outward normal vector to the boundary $\partial\mathcal{D}$, while the Cahn-Hilliard equation is equipped with a dynamic boundary and the Neumann boundary conditions:
\begin{align}
&\partial_n\mu=0, \label{1.4}\\
&\partial_t\phi+u_\tau\cdot \nabla_\tau \phi=-\mathcal{K}(\phi),\label{1.5}
\end{align}
on $\partial\mathcal{D}\times (0,\infty)$ and
$$\mathcal{K}(\phi)=-\Delta_{\tau}\phi+\partial_n\phi+\phi+g(\phi).$$
Such a boundary condition (\ref{1.5}) could be considered as a parabolic equation on boundary $\partial\mathcal{D}$. The notation $\nabla_\tau$ represents the tangential gradient operator along tangential direction $\tau=(\tau_1,\cdots, \tau_{N-1})$ on the boundary $\partial\mathcal{D}$ and we denote by $\Delta_\tau$ the Laplace-Beltrami operator on $\partial\mathcal{D}$, and $u_\tau$ the tangential component of $u$ on $\partial\mathcal{D}$. $g$ corresponds to the interfacial energy at the wall interface. For more descriptions of this kind of boundary, we refer to \cite{qian,qian2}.

Note that, from the physical point of view it is significant to find the boundary conditions (\ref{1.2}) and (\ref{1.4}) guarantee the conservation form of $\phi$, that is $$\langle\phi(t)\rangle=\langle\phi(0)\rangle,~ {\rm for~ all}~ t\geq 0,$$ where $\langle\cdot\rangle$ denotes the mean value over $\mathcal{D}$, thus, $\langle\phi(t)\rangle=|\mathcal{D}|^{-1}\int_{\mathcal{D}}\phi(t,x) dx$.

We supply the following initial conditions:
\begin{align}\label{1.6}
u(0,x)=u_0, ~\phi(0,x)=\phi_0.
\end{align}

When $h\equiv 0$, system \eqref{Equ1.1}-\eqref{1.6} reduces to the deterministic one. Gal, Grasselli and Miranville \cite{GMA} proved the existence of global energy weak solution, and also established the convergence of such solution to a single equilibrium. Moreover, Cherfils and Petcu \cite{LM} obtained the existence, uniqueness, and regularity of the solutions, in which the Navier boundary condition was replaced by non-slip boundary condition. For more researches, see \cite{wu,gal2} and references therein.

According to what we know, the relative stochastic issue to system (\ref{Equ1.1})-(\ref{1.6}) has not been considered yet. We are the first to consider the stochastic version with dynamic boundary, and devoted to proving the global existence of martingale weak solution to 2D and 3D system (\ref{Equ1.1})-(\ref{1.6}) and an almost surely Markov selection to associated martingale problem. Note that even for 2D case, we can not establish the existence of global pathwise solution which is not a drawback of our approach but an intrinsic difficulty, mainly arising from following reasons: firstly, the boundary term $\mathcal{K}(\phi)\nabla_\tau \phi$ and the boundary advection term involve strong coupling of the system which fails the proof of uniqueness. Consequently, we can not use the Gy\"{o}ngy-Krylov theorem to recover the convergence in original fixed probability space; Secondly, it is not allowed to apply the Galerkin scheme to establish the existence of approximate solutions due to the fact that in process of obtaining the a priori estimates, the test functions are not compatible any more, see Remark \ref{rem3.3}. As a result, it is impossible to show the Galerkin-type approximations converge in mean square for passing the limit.

In view of these difficulties, we will adopt a mixed method which combines a fixed point theorem, the stochastic compactness argument and an extension technique to establish the existence of global martingale weak solution. The specific strategy is as follows.

$\bullet$ By the linearization and mollification technique, we transform the original system into a linear approximate modified system driven by additive noise with dynamic boundary. In this step, we add cut-off function in nonlinear terms to control strong nonlinearity construction.

$\bullet$ With linear approximate modified system established, we first show that the global well-posedness for the auxiliary system holds.

$\bullet$ Then, we could define a mapping $\mathcal{M}$ such that the well-posedness of approximate solutions of original system is equivalent to searching for a fixed point of this mapping.

$\bullet$ Applying the truncated skill and stopping time technique to show the contraction property of mapping $\mathcal{M}$. In this procedure, we will design a two-level stopping time technique and an extra level approximate scheme to surmount  difficulties caused by random effect and cut-off function.

$\bullet$ Finally, we establish the uniform a priori estimates of the regularized solution by the elliptic regularity theory, the stochastic compactness, and identify the limit.

How to establish the approximate solutions is more complex and challenging compared with the deterministic case \cite{GMA}. In order to close the contraction estimate, here we have to consider an extra level approximation and develop the cut-off skill. However, the cut-off operators bring further difficulty in establishing the contraction estimate, we will introduce stopping time technique to overcome it. Moreover, we have to choose the same cut-off operators in front of all nonlinear terms, making sure that the a priori estimates could be achieved.

A remarkable feature of dynamical system with uniqueness is the memoryless property, thus, the Markov property. As mentioned, even for 2D case, uniqueness for the system is unknown, therefore it is natural to consider if we could choose one law satisfying the Markov property. Existence of Markov selection is originally introduced by  Krylov \cite{kry} for the stochastic differential equations in finite-dimensional setting. Then,  Flandoli and Romito \cite{F2} generalized the result to infinite-dimensional setting, establishing the Markov selection for the 3D incompressible Navier–Stokes equations with additive noise. Furthermore, Breit, Feireisl and  Hofmanov\'{a} \cite{brei} proved the existence of the Markov selection for the 3D stochastic compressible Navier–Stokes equations with a general multiplicative stochastic forcing. Dong and Zhai \cite{dong} considered the 3D stochastic incompressible Navier–Stokes equations with jump noise, obtaining the existence of martingale solution and an almost surely Markov selection. For more researches of Markov selection to related stochastic dynamical systems, we refer to \cite{zhang, dong2}.

A special point of the coupled system is that we only have regularity $\mu\in L^2_tH^1, \mathcal{K}\in L^2_tL^2(\Gamma)$ a.s. due to the involved structure. Therefore, quantities $\mu, \mathcal{K}$ are not processes in the classical sense and we have to understand them as $\mathcal{F}_t$-adapted random distribution. Following the idea of \cite{brei}, we overcome the difficulty by introducing new variables $\mathcal{V}_1, \mathcal{V}_2$ satisfying equations $d\mathcal{V}_1=\mu dt$ and $d\mathcal{V}_2=\mathcal{K}dt$. Note that $\mathcal{V}_1$ and $\mathcal{V}_2$ are two continuous stochastic processes with trajectories in $W^{1,2}_tH^1$ and $W^{1,2}_tL^2(\Gamma)$, a.s. respectively and $\mu, \mathcal{K}$ are measurable functions of $\mathcal{V}_1, \mathcal{V}_2$. As a result, we could  establish the existence of an almost surely Markov selection for $(u, \phi, \psi, \mathcal{V}_1, \mathcal{V}_2)$.

The rest of this paper is structured as follows. In section 2, we will give some deterministic and stochastic preliminaries needed throughout the paper. Our focus in section 3 will be to prove the existence of global martingale weak solution by four parts. We prove the existence of Markov selection to associated martingale problem in section 4. Appendix is included afterwards to state some known results that are frequently used in this paper. Finally, two auxiliary results are given in section 6.

\section{Preliminaries}

Throughout this paper, we use short notations $L_t^pX:=L^p(0,T; X)$, $L_\omega^pX:=L^p(\Omega;X)$ and $W^{\alpha,p}_tX:=W^{\alpha,p}(0,T;X)$, $C^\alpha_t X:=C^\alpha([0,T];X)$  for $\alpha\in (0,1],~ p\geq 1$, where $X$ is a Banach space and $\Omega$ is a sample space. Here, we use the fractional order Sobolev space with respect to $t$, since the evolution equations forced by noise only have $\alpha$-order H\"{o}lder continuous of $t$ for $\alpha\in (0,\frac{1}{2})$, see \cite{ANR,FD} for more details of these spaces. Denote
$$(L_t^pX)_w:={\rm the~ space~} L^p(0,T; X)~{\rm with~ weak~ topology}.$$
To simplify the notation, let $\Gamma: =\partial\mathcal{D}$, denote $H^s(\mathcal{D})$ and $H^s(\Gamma)$ norm by $\|\cdot\|_{H^s}$ and $\|\cdot\|_{H^s(\Gamma)}$ for any $s>0$, when $s=0$, the notation $L_x^2:=L^2(\mathcal{D})$. Denote by $X'$ the dual of Banach space $X$ and by $\mathbb{E}$ the mathematical expectation. Denote by ${\rm Tr}_{\mathcal{D}}$ the trace operator from $H^s(\mathcal{D})$ into $H^{s-\frac{1}{2}}(\Gamma)$.

Recall that the Laplace-Beltrami operator $A_{\tau}:=-\Delta_{\tau}$ is a nonnegative self-adjoint operator in $L^2(\Gamma)$. Therefore, we could define $$H^s(\Gamma)=\left\{\psi\in L^2(\Gamma):\psi\in D(A_{\tau}^\frac{s}{2})\right\},$$
with norm
$$\|\psi\|_{H^s(\Gamma)}=\|A_{\tau}^\frac{s}{2} \psi\|_{L^2(\Gamma)}+\|\psi\|_{L^2(\Gamma)}.$$

Introducing the Banach space $V^s$ for $s\geq 0$ corresponding to the parabolic equation on boundary $\Gamma$
$$V^s=\left\{(\phi, \psi)\in H^s\times H^{s-\frac{1}{2}}(\Gamma):\psi={\rm Tr}_{\mathcal{D}}(\phi)\in H^s(\Gamma)\right\},$$
with norm
$$\|(\phi,\psi)\|_{V^s}^2=\|\phi\|_{H^s}^2+\|\psi\|_{H^s(\Gamma)}^2,$$
when $s=0$, we write
$$\|(\phi,\psi)\|_{V^0}^2=\|\phi\|_{L^2_x}^2+\|\psi\|_{L^2(\Gamma)}^2.$$

Also we introduce the operator $A_1\phi=-\Delta \phi$ for any
$$\phi\in D(A_1)=\left\{\phi\in H^2(\mathcal{D}):\partial_n\phi|_{\Gamma}=0\right\}.$$
The Poincar\'{e} inequality implies $H^2$-norm is equivalent to norm $\|A_1(\cdot)\|_{L^2_x}+|\langle\cdot\rangle|$. Then, we introduce the operator
$$A_2\phi=-\Delta \phi,~ {\rm for}~\phi \in D(A_2)=D(A_1)\cap L_0^2,$$
where $L_0^2$ is the space of $L^2$-functions with $0$-mean valued. Note that $A_2^{-1}$ is a compact operator in $L_0^2$, and if $\phi\in D(A_2)$, the operator $A_2$ is equivalent to $A_1$. For any $\phi\in D(A_1)$, we have $\phi-\langle\phi\rangle\in D(A_2)$, and then $H^s$-norm is equivalent to $\|A_2^\frac{s}{2}(\phi-\langle\phi\rangle)\|_{L_x^2}+\|\langle\phi\rangle\|_{L_x^2}$. The Poincar\'{e} inequality again implies $H^s$-norm is equivalent to the norm $\|A^\frac{s}{2}_2(\cdot)\|_{L^2_x}$ for all $\phi\in D(A_2^\frac{s}{2})$.

 If $\mathrm{L}_x^p, \mathrm{H}^{s,p}$ represent the solution spaces of the fluid velocity, we define them by the closure of
\begin{equation*}
\mathcal{V}=\left\{u\in \mathcal{C}_{c}^{\infty}(\mathcal{D}):~{\rm div} u=0~ {\rm in} ~\mathcal{D}, u\cdot n=0,{\rm on }~\Gamma\right\},
\end{equation*}
with respect to $L^p(\mathcal{D}), H^{s,p}(\mathcal{D})$-norms. Considering the boundary condition, we introduce the bilinear form
$$a_0(u,v):=2(D(u), D(v))+(u_\tau,v_\tau)_{\Gamma},~ {\rm for}~u, v\in \mathrm{H}^1.$$
Note that, the bilinear form $a_0$ is coercive, continuous, symmetric and $\sqrt{a_0(u,u)}$ is equivalent to $\mathrm{H}^1$-norm. Corresponding to $a_0$, we could define the Stokes operator by $A_0u=-P{\rm div} (2D(u))$ such that $(A_0u, v)=a_0(u,v)$ for all
$$u\in D(A_0):=\left\{u\in H^2\cap\mathrm{H}^1: 2(D(u)\cdot n)_\tau+ u_\tau=0~ {\rm on} ~\Gamma\right\},$$
where $P$ is the Helmholtz-Leray projection. By the definition, we could have that the operator $A_0$ is positive and self-adjoint on $\mathrm{L}_x^2$, and $A_0^{-1}$ is compact.

To simplify the notation, we next define $(b_{0}, b_{1}, b_{2}): (L_x^p\times H^{1,q})^3 \rightarrow (L_x^{r'})^3$ with $1-\frac{1}{r}=\frac{1}{p}+\frac{1}{q}=\frac{1}{r'}$ as the bilinear operators such that
\begin{eqnarray*}
&&\langle b_{0}(u,v),\eta\rangle=\int_{\mathcal{D}}(u\cdot \nabla v)\cdot \eta dx=B_{0}(u,v,\eta),\\
&&\langle b_{1}(\mu,\phi),\eta\rangle=\int_{\mathcal{D}}\mu(\nabla \phi\cdot \eta) dx=B_{1}(\mu,\phi, \eta),\\
&&\langle b_{2}(u,\phi),\rho\rangle=\int_{\mathcal{D}}(u\cdot \nabla \phi)\cdot \rho dx=B_{2}(u,\phi,\rho),
\end{eqnarray*}
and also define the bilinear operator by $b_\Gamma: L^p(\Gamma)\times H^{1,q}(\Gamma) \rightarrow L^{r'}(\Gamma)$ such that
$$\langle b_{\Gamma}(u,\psi),\eta\rangle=\int_{\Gamma}(u_{\tau}\cdot \nabla_{\tau}\psi)\cdot \eta dS= B_{\Gamma}(u, \psi, \eta).$$
Recall that $B_{0}(u,v,v)=0$ for $u,v\in \mathrm{H}^1$.

Let $\mathcal{S}=(\Omega,\mathcal{F},\{\mathcal{F}_{t}\}_{t\geq0},\mathbb{P}, \mathcal{W})$ be a fixed stochastic basis and $(\Omega,\mathcal{F},\mathbb{P})$ be a complete probability space. $\mathcal{W}$ is a cylindrical Wiener process defined on the Hilbert space $\mathcal{H}$, which is adapted to the complete, right continuous filtration $\{\mathcal{F}_{t}\}_{t\geq 0}$. Namely, $\mathcal{W}=\sum_{k\geq 1}e_k\beta_{k}$ with $\{e_k\}_{k\geq 1}$ being a complete orthonormal basis of $\mathcal{H}$ and $\{\beta_{k}\}_{k\geq 1}$ being a sequence of independent standard one-dimensional Brownian motions. In addition, $L_{2}(\mathcal{H},X)$ denotes the collection of Hilbert-Schmidt operators, the set of all linear operators $G$ from $\mathcal{H}$ to $X$, with norm $\|G\|_{L_{2}(\mathcal{H},X)}^2=\sum_{k\geq 1}\|Ge_k\|_{X}^2$.

Considering an auxiliary space $\mathcal{H}_0\supset \mathcal{H}$,  we define $\mathcal{H}_0$ by
\begin{eqnarray*}
\mathcal{H}_0=\left\{h=\sum_{k\geq 1}\alpha_k e_k: \sum_{k\geq 1}\alpha_k^2k^{-2}<\infty\right\},
\end{eqnarray*}
 with norm $\|h\|_{\mathcal{H}_0}^2=\sum_{k\geq 1}\alpha_k^2k^{-2}$. Observe that the mapping $\Phi:\mathcal{H}\rightarrow\mathcal{H}_0$ is Hilbert-Schmidt. We also have that $\mathcal{W}\in C([0,\infty),\mathcal{H}_0)$ for almost all $\omega$, see \cite{Zabczyk}.

For an $X$-valued predictable process $\varpi\in L^{2}(\Omega;L^{2}_{loc}([0,\infty),L_{2}(\mathcal{H},X)))$  by taking $\varpi_{k}=\varpi e_{k}$, the following Burkholder-Davis-Gundy inequality holds:
\begin{eqnarray*}
\mathbb{E}\left(\sup_{t\in [0,T]}\left\|\int_{0}^{t}\varpi d\mathcal{W}\right\|_{X}^{p}\right)\leq c_{p}\mathbb{E}\left(\int_{0}^{T}\|\varpi\|_{L_{2}(\mathcal{H},X)}^{2}dt\right)^{\frac{p}{2}}
=c_{p}\mathbb{E}\left(\int_{0}^{T}\sum_{k\geq 1}\|\varpi_k\|_{X}^{2}dt\right)^{\frac{p}{2}},
\end{eqnarray*}
for any $p\geq1$.

In the next sections, since 2D case is easier, we only focus on 3D case for the details of proof.

\section{The existence of martingale weak solution }
Our goal in this section is to prove the existence of global martingale weak solution to system \eqref{Equ1.1}-\eqref{1.6}. First, we formulate the definition.
\begin{definition}\label{def2.1} Let $\mathcal{P}_0$ be a probability measure on space $\mathrm{L}_x^2\times V^1$. The $(\mathcal{S}, u,\phi, \psi)$ is a global martingale weak solution to system (\ref{Equ1.1})-(\ref{1.6}) if the followings are satisfied:\\
1. $\mathcal{S}=(\Omega, \mathcal{F}, \{\mathcal{F}_t\}_{t\geq 0}, \mathbb{P}, \mathcal{W})$ is a stochastic basis, where $\mathcal{W}$ is a Wiener process;\\
2. $(u,\phi, \psi)$ are $\{\mathcal{F}_t\}_{t\geq 0}$ progressive measurable processes such that
\begin{align*}
u\in L_w^p(L^\infty_t\mathrm{L}_x^2\cap L^2_t\mathrm{H}^1), ~(\phi, \psi)\in L_w^p(L^\infty_tV^1\cap L^2_tV^2),
\end{align*}
where $\psi={\rm Tr}_{\mathcal{D}}(\phi)$, $\mathbb{P}${\rm-a.s.};\\
3. $\mathcal{P}_0=\mathbb{P}\circ \{(u_0, \phi_0, \psi_0)\}^{-1}$;\\
4. for any $v_0\in \mathrm{H}^1, v_1\in H^1, v_2\in L^2(\Gamma)$, $t\in [0,T]$, $(u,\phi, \psi)$ satisfy $\mathbb{P}${\rm-a.s.}
\begin{align}\label{2.1}
\left\{\begin{array}{ll}
\!\!(u,v_0)+\int_{0}^{t}a_0(u,v_0)+B_0(u,u,v_0)-B_1(\mu, \phi, v_0)ds=\int_{0}^{t}(\mathcal{K}(\psi)\nabla_{\tau}\psi, v_{0,\tau})_\Gamma ds\\ \qquad\qquad\qquad\qquad\qquad\qquad\qquad+\int_{0}^{t}(h(u, \nabla\phi),v_0)d\mathcal{W}+(u_0, v_0),\\
\!\!(\phi, v_1)+\int_{0}^{t} (\nabla \mu, \nabla v_1)+ B_2(u,\phi,v_1)ds=(\phi_0, v_1),\\
\!\!(\psi, v_2)+\int_{0}^{t} (\mathcal{K}(\psi), v_2)+ B_\Gamma(u,\psi,v_2)ds=(\psi_0, v_2),
\end{array}\right.
\end{align}
with
\begin{align*}
\left\{\begin{array}{ll}
\!\!\mu=-\Delta\phi+f(\phi),~\mathbb{P}\mbox{-a.s.}~ {\rm in} ~\mathcal{D}\times (0,T),\\
\!\!\mathcal{K}(\psi)=A_{\tau}\psi+\partial_n\phi+\psi+g(\psi),~ \mathbb{P}\mbox{-a.s.}~{\rm on} ~\Gamma\times (0,T),
\end{array}\right.
\end{align*}
and $\mu, \mathcal{K}(\psi)$ are random distributions (for more details of definition see \cite[Section 3]{brei}) adapted to $\mathcal{F}_t$ with regularity
$$\mu\in L_w^p(L^2_tH^1),~ \mathcal{K}(\psi)\in L_w^p(L^2_tL^2(\Gamma));$$
5. the energy inequality
\begin{align*}
&\frac{1}{2}\left[\mathbb{E}\left[1_{\mathcal{A}}E^2\right]\right]_{s}^t+\mathbb{E}\left[1_{\mathcal{A}}\int_{s}^{t}E\left(\|\nabla u,\nabla\mu\|_{L_x^2}^2+\|(\phi, \psi)\|^2_{V^2}+\|\mathcal{K}(\psi)\|^2_{L^2(\Gamma)}\right)dr\right]\nonumber\\
&\leq \mathbb{E}\left[1_{\mathcal{A}}\int_{s}^{t}E\|h(u,\nabla\phi)\|^2_{L_2(\mathcal{H};L_x^2)}dr\right]
+\mathbb{E}\left[1_{\mathcal{A}}\int_{s}^{t}(h(u,\nabla\phi),u)^2dr\right]
\end{align*}
holds for any $t\geq 0$ and a.e. $s$, $0\leq s\leq t$ and $\mathcal{A}\in \mathcal{F}_s$, where
$$E:=\frac{1}{2}\left(\|u\|_{\mathrm{L}_x^2}^2+\|\nabla \phi\|_{L_x^2}^2+\|\psi\|_{L^2(\Gamma)}^2+\|\nabla_{\tau}\psi\|_{L^2(\Gamma)}^2\right)+\int_{\mathcal{D}}F(\phi)dx+\int_{\Gamma}G(\psi)dS,$$
and $1_{\cdot}$ is the indicator function.
\end{definition}

We next introduce assumptions imposed on the noise intensity operator $h$: there exists constant $C$ such that\\
${\bf A}_1$. $\|h(u, \nabla\phi)\|^2_{L_2(\mathcal{H}; L_x^2)}\leq C(1+\|(u,\nabla\phi)\|_{L_x^2}^2)$, for $u \in \mathrm{L}_x^2, \phi\in H^1$,\\
${\bf A}_2$. $\|h(u_1, \nabla\phi_1)-h(u_2, \nabla\phi_2)\|^2_{L_2(\mathcal{H}; L_x^2)}\leq C\|(u_1-u_2, \nabla(\phi_1-\phi_2))\|^2_{L_x^2}$ for $u\in \mathrm{L}_x^2,\phi\in H^1$.\\
Furthermore, assumptions on $f,g$ are as follows: there exist positive constants $C_1, C_2$ such that\\
${\bf A}_3$. $|f'(r)|\leq C_1(1+|r|^2)$ for $f\in C^1(\mathbb{R})$, $f(0)=0$ and $F''(s)\geq -C_2, F(s)\geq -C_2$, \\
${\bf A}_4$. $|g'(r)|\leq C_1(1+|r|^m)$ for $g\in C^1(\mathbb{R})$, any $m\geq 1$, $g(0)=0$ and $G''(s)\geq -C_2, G(s)\geq -C_2$, where $G(s)=\int_{0}^{s}g(r)dr$.

Here the norm $\|(u,v)\|_{L_x^2}:=\|u\|_{\mathrm{L}_x^2}+\|v\|_{L_x^2}$, for $u\in \mathrm{L}_x^2, v\in L_x^2$.
\begin{remark}
These assumptions imposed on $f$ include the meaningful physical case $f(r)=r^3-r$.
\end{remark}

Now, we state our existence result.
\begin{theorem}\label{thm2.1} Suppose that initial data $(u_0, \psi_0,\psi_0)\in L^p_\omega(\mathrm{L}^2_x\times V^1)$ are $\mathcal{F}_0$-measurable random variables for any $p>2$, and functions $f,g\in C^1(\mathbb{R})$ satisfy assumptions ${\bf A}_3$ and ${\bf A}_4$, operator $h$ satisfies assumptions ${\bf A}_1$ and ${\bf A}_2$. Then, there exists a global martingale weak solution to system (\ref{Equ1.1})-(\ref{1.6}) in the sense of Definition \ref{def2.1}.
\end{theorem}

The proof of Theorem \ref{thm2.1} is divided into four parts. The existence and uniqueness of approximate solutions are established by exploiting the fixed point argument in first part. In part 2, we are devoted to obtaining the a priori estimates by selecting suitable test functions. We pass the limit with respect to $\varepsilon,\delta$ in the regularized system by the stochastic compactness argument in Parts 3, 4.

{\bf Part 1. The existence and uniqueness of approximate solutions.}
Define by space $X=X_1\times X_2$, where
\begin{align*}
&X_1=\left\{u:u\in L^p_\omega (L_t^\infty \mathrm{L}_x^2\cap L^2_t\mathrm{H}^1\cap W_t^{1,2}(\mathrm{H}^1)'), u|_{t=0}=u_0\right\},\\
&X_2=\left\{(\phi,\psi):(\phi,\psi)\in L^p_\omega (L_t^\infty V^1\cap L^2_tV^2\cap W_t^{1,2}V^0), (\phi,\psi)|_{t=0}=(\phi_0, \psi_0)\right\},
\end{align*}
for any $p>2$.

Let $\mathcal{J}_\varepsilon, J_\varepsilon$ be the mollifiers on the corresponding spaces $H^s(\mathcal{D}), H^s(\Gamma)$ for $s\geq 0$, respectively, for more properties and the construction of such operators see Lemma \ref{lem4.4} and Remark \ref{rem4.1}. We consider the mollification system:
\begin{eqnarray}\label{equ3.1}
\left\{\begin{array}{ll}
\!\!(du_{\varepsilon, \delta}, \xi)+a(u_{\varepsilon, \delta},\xi)dt+B_0(u_{\varepsilon, \delta},\mathcal{J}_\varepsilon u_{\varepsilon, \delta}, \xi)dt-B_1(\mu_{\varepsilon, \delta}, \mathcal{J}_{\varepsilon}\phi_{\varepsilon, \delta}, \xi)dt\\
\!\!\qquad\qquad-(\mathcal{K}(\psi_{\varepsilon,\delta})\nabla_{\tau}(J_\varepsilon \psi_{\varepsilon,\delta}), \xi_\tau)_{\Gamma}dt=(h(u_{\varepsilon, \delta},\nabla\phi_{\varepsilon, \delta}), \xi)d\mathcal{W},\\
\!\!(\partial_t\phi_{\varepsilon, \delta}, \xi_1)+(\nabla\mu_{\varepsilon, \delta}, \nabla\xi_1)+B_2(u_{\varepsilon, \delta}, \mathcal{J}_{\varepsilon}\phi_{\varepsilon, \delta}, \xi_1)=0,\\
\!\!(\partial_t\psi_{\varepsilon, \delta}, \xi_2)_{\Gamma}+B_\Gamma(u_{\varepsilon, \delta},J_{\varepsilon}(\psi_{\varepsilon,\delta}), \xi_2)+(\mathcal{K}(\psi_{\varepsilon,\delta}), \xi_2)_{\Gamma}=0,
\end{array}\right.
\end{eqnarray}
for $\xi\in \mathrm{H}^1, \xi_1\in H^1, \xi_2\in  L^2(\Gamma)$, where
\begin{align}
&\mu_{\varepsilon, \delta}=-\Delta \phi_{\varepsilon, \delta}+\delta\partial_t\phi_{\varepsilon, \delta}+f(\phi_{\varepsilon, \delta}),\label{3.3}\\&
\mathcal{K}(\psi_{\varepsilon,\delta})=A_{\tau}\psi_{\varepsilon,\delta}+\partial_n\phi_{\varepsilon, \delta}+\psi_{\varepsilon,\delta}+g(\psi_{\varepsilon,\delta}).\label{3.4}
\end{align}
In order to overcome the difficulty caused by random effect, we introduce $C^{\infty}$-smooth cut-off function $\overline{\Psi}_{R}:[0,\infty)\rightarrow [0,1]$ with
\begin{eqnarray*}
 \overline{\Psi}_{R}(x) =\left\{\begin{array}{ll}
                  1,& \mbox{if} \ 0<x<R,  \\
                  0,& \mbox{if} \ x\geq 2R.  \\
                \end{array}\right.
\end{eqnarray*}
Then, define $\Psi_R=\overline{\Psi}_{R}(\|u\|_{\mathrm{L}_x^2})\cdot\overline{\Psi}_{R}(\|(\phi, \psi)\|_{ V^1})$, add the cut-off function in front of nonlinear terms of system (\ref{equ3.1}), we obtain
\begin{eqnarray}\label{equ3.1*}
\left\{\begin{array}{ll}
\!\!(du_{\varepsilon, \delta}, \xi)+a(u_{\varepsilon, \delta},\xi)dt+\Psi_RB_0(u_{\varepsilon, \delta},\mathcal{J}_\varepsilon u_{\varepsilon, \delta}, \xi)dt-\Psi_RB_1(\mu_{\varepsilon, \delta}, \mathcal{J}_{\varepsilon}\phi_{\varepsilon, \delta}, \xi)dt\\
\!\!\qquad\qquad-\Psi_R(\mathcal{K}(\psi_{\varepsilon,\delta})\nabla_{\tau}(J_\varepsilon \psi_{\varepsilon,\delta}), \xi_\tau)_{\Gamma}dt=(h(u_{\varepsilon, \delta},\nabla\phi_{\varepsilon, \delta}), \xi)d\mathcal{W},\\
\!\!(\partial_t\phi_{\varepsilon, \delta}, \xi_1)+(\nabla\mu_{\varepsilon, \delta}, \nabla\xi_1)+\Psi_RB_2(u_{\varepsilon, \delta}, \mathcal{J}_{\varepsilon}\phi_{\varepsilon, \delta}, \xi_1)=0,\\
\!\!(\partial_t\psi_{\varepsilon, \delta}, \xi_2)_{\Gamma}+\Psi_RB_\Gamma(u_{\varepsilon, \delta},J_{\varepsilon}(\psi_{\varepsilon,\delta}), \xi_2)+(\mathcal{K}(\psi_{\varepsilon,\delta}), \xi_2)_{\Gamma}=0.
\end{array}\right.
\end{eqnarray}
Considering its linearization
\begin{eqnarray}\label{equ3.2}
\left\{\begin{array}{ll}
\!\!(du_{\varepsilon, \delta}, \xi)+a(u_{\varepsilon, \delta},\xi)dt+\Psi_RB_0( v_{\varepsilon, \delta},\mathcal{J}_\varepsilon v_{\varepsilon, \delta}, \xi)dt-\Psi_RB_1(\tilde{\mu}_{\varepsilon, \delta}, \mathcal{J}_{\varepsilon}\varphi_{\varepsilon, \delta}, \xi)dt\\
\!\!\qquad\qquad-\Psi_R(\mathcal{K}(\eta_{\varepsilon,\delta})\nabla_{\tau}(J_\varepsilon \eta_{\varepsilon,\delta}), \xi_\tau)_{\Gamma}dt=(h(v_{\varepsilon, \delta},\nabla\varphi_{\varepsilon, \delta}), \xi)d\mathcal{W},\\
\!\!(\partial_t\phi_{\varepsilon, \delta}, \xi_1)+(\nabla\bar{\mu}_{\varepsilon, \delta}, \nabla\xi_1)+\Psi_RB_2(v_{\varepsilon, \delta}, \mathcal{J}_{\varepsilon}\varphi_{\varepsilon, \delta}, \xi_1)=0,\\
\!\!(\partial_t\psi_{\varepsilon, \delta}, \xi_2)_{\Gamma}+\Psi_RB_\Gamma(v_{\varepsilon, \delta},J_{\varepsilon}(\eta_{\varepsilon,\delta}), \xi_2)+(\widetilde{\mathcal{K}}(\psi_{\varepsilon,\delta}), \xi_2)_{\Gamma}=-(g(\eta_{\varepsilon,\delta}),\xi_2)_{\Gamma},
\end{array}\right.
\end{eqnarray}
where processes $(v_{\varepsilon, \delta}, \varphi_{\varepsilon, \delta}, \eta_{\varepsilon,\delta})\in X$ and
\begin{align}
&\tilde{\mu}_{\varepsilon, \delta}=-\Delta \varphi_{\varepsilon, \delta}+\delta\partial_t\varphi_{\varepsilon, \delta}+f(\varphi_{\varepsilon, \delta}),\label{3.5}\\&
\mathcal{K}(\eta_{\varepsilon,\delta})=A_{\tau}\eta_{\varepsilon,\delta}+\partial_n\varphi_{\varepsilon, \delta}+\eta_{\varepsilon,\delta}+g(\eta_{\varepsilon,\delta}),\label{3.7**}\\
&\widetilde{\mathcal{K}}(\psi_{\varepsilon,\delta})=A_{\tau}\psi_{\varepsilon,\delta}+\partial_n\phi_{\varepsilon, \delta}+\psi_{\varepsilon,\delta},\\
&\bar{\mu}_{\varepsilon, \delta}=-\Delta \phi_{\varepsilon, \delta}+\delta\partial_t\phi_{\varepsilon, \delta}+f(\varphi_{\varepsilon, \delta}).\label{3.6*}
\end{align}

Assumptions ${\bf A}_3$ and ${\bf A}_4$ are insufficient for closing the estimate, therefore we use the functions $f_\varepsilon, g_\varepsilon\in C^2$ substituting for $f, g$ with the properties:\\
(i) there exist two constants $c_{f,\varepsilon}, c_{g,\varepsilon}$ such that
$$|f_\varepsilon'(s)|\leq c_{f,\varepsilon}, ~|g_\varepsilon'(s)|\leq c_{g,\varepsilon},$$
for all $s\in \mathbb{R}$,\\
(ii) on any compact set of $\mathbb{R}$, the convergences $f_\varepsilon\rightarrow f, g_\varepsilon\rightarrow g$ are uniform, see \cite{GMA,gal2} for the construction of $f_\varepsilon, g_\varepsilon$.

Define the mapping
$$\mathcal{M}:X\longmapsto X,~ (v_{\varepsilon, \delta}, \varphi_{\varepsilon, \delta}, \eta_{\varepsilon,\delta})\longmapsto (u_{\varepsilon, \delta}, \phi_{\varepsilon, \delta}, \psi_{\varepsilon, \delta}).$$
We show that the mapping is well-defined. %Observe that, taking $\xi_1\equiv 1$, we have $\langle \partial_t\phi_{\varepsilon,\delta}\rangle=0, {\rm a.e.}~ (0,T)\times \Omega$ which allows us to apply the Poincar\'{e} inequality to get the desired estimates.
By Property (i) and the mean-value theorem, we have
\begin{align}\label{3.7*}
\|f_\varepsilon(\varphi_{\varepsilon, \delta})\|_{L_t^\infty L_x^2}&\leq \|f_\varepsilon'(\theta \varphi_{\varepsilon,\delta})\varphi_{\varepsilon, \delta}\|_{L_t^\infty L_x^2}\leq c_{f,\varepsilon}\|\varphi_{\varepsilon,\delta}\|_{L_t^\infty L_x^2},
\end{align}
where $\theta\in (0,1)$ is a constant. Similarly, we have
\begin{align}\label{3.8**}
\begin{split}
\|g_\varepsilon(\eta_{\varepsilon,\delta})\|_{L_t^\infty L^2(\Gamma)}& \leq c_{g,\varepsilon}\|\eta_{\varepsilon,\delta}\|_{L_t^\infty L^2(\Gamma)}.
\end{split}
\end{align}
We use \eqref{3.5}, \eqref{3.7*}  and the fact $(v_{\varepsilon, \delta}, \varphi_{\varepsilon, \delta}, \eta_{\varepsilon,\delta})\in X$ to find
$$\tilde{\mu}_{\varepsilon, \delta}\in L_\omega^pL^2_tL^2_x.$$
And from $\eqref{3.7**}$, $\eqref{3.8**}$ and the embedding $H^{m,p}(\mathcal{D})\hookrightarrow H^{m-\frac{1}{p},p}(\Gamma)$, we get
$$\mathcal{K}(\eta_{\varepsilon,\delta})\in L_\omega^pL_t^2L^2(\Gamma).$$
Moreover, we have
\begin{align*}
&\mathcal{J}_\varepsilon v_{\varepsilon, \delta}\in L_\omega^pL_t^\infty \mathrm{H}^2,~\mathcal{J}_{\varepsilon}\varphi_{\varepsilon, \delta}\in L_\omega^pL_t^\infty H^3, ~J_\varepsilon \eta_{\varepsilon,\delta}\in L_\omega^pL_t^\infty H^3(\Gamma),
\end{align*}
which combines with assumption ${\bf A}_1$, H\"{o}lder's inequality leading to
\begin{align*}
&h(v_{\varepsilon, \delta},\nabla\varphi_{\varepsilon, \delta})\in L_\omega^pL_t^\infty L^2_x,\\
&f_1:=\mathcal{K}(\eta_{\varepsilon,\delta})\nabla_{\tau}(J_\varepsilon \eta_{\varepsilon,\delta})\in L_\omega^pL_t^2L^2(\Gamma),\\
&f_2:=-b_0( v_{\varepsilon, \delta}, \mathcal{J}_\varepsilon v_{\varepsilon, \delta})+\tilde{\mu}_{\varepsilon, \delta}\cdot\nabla \mathcal{J}_{\varepsilon}\varphi_{\varepsilon, \delta}\in L_\omega^pL_t^\infty H^2,\\
&g_1:=-b_1(v_{\varepsilon, \delta}, \mathcal{J}_{\varepsilon}\varphi_{\varepsilon, \delta})\in L_\omega^pL_t^2L_x^2,\\
&g_2:=-b_\Gamma(v_{\varepsilon, \delta},J_{\varepsilon}(\eta_{\varepsilon,\delta}))-g_\varepsilon(\eta_{\varepsilon,\delta})\in L_\omega^pL_t^2L^2(\Gamma).
\end{align*}
Thus, we verified the required conditions of the auxiliary Propositions \ref{pro5.1} and \ref{pro5.2} which are invoked to obtain the existence and uniqueness of solution $(u_{\varepsilon, \delta}, \phi_{\varepsilon, \delta}, \psi_{\varepsilon, \delta})\in X$ to the linearization system (\ref{equ3.2}).

The fixed point theorem tells us once we show mapping $\mathcal{M}$ is contraction from $X$ into itself, we could obtain the existence and uniqueness of approximate solutions to mollification system (\ref{equ3.1*}). To this end, we first show:

{\bf Step 1}. Proving that $\mathcal{M}$ is a mapping from $B_{K_1,K_2}$ into itself, where
$$B_{K_1,K_2}:=\left\{(v,\varphi, \eta)\in X: \|v\|_{X_1}\leq K_1, \|(\varphi, \eta)\|_{X_2}\leq K_2\right\},$$
for two positive constants $K_1, K_2$ choosing later.

Using It\^{o}'s formula to the function $\frac{1}{2}\|u_{\varepsilon, \delta}\|_{\mathrm{L}_x^2}^2$, we have
\begin{align}\label{3.7}
&\frac{1}{2}d\|u_{\varepsilon,\delta}\|_{\mathrm{L}_x^2}^2+a_0(u_{\varepsilon,\delta}, u_{\varepsilon,\delta})dt\nonumber\\
&=\big(-\Psi_RB_0( v_{{\varepsilon,\delta}},\mathcal{J}_{\varepsilon,\delta} v_{{\varepsilon,\delta}}, u_{\varepsilon,\delta})+\Psi_RB_1(\tilde{\mu}_{\varepsilon,\delta}, \mathcal{J}_{{\varepsilon}}\varphi_{\varepsilon,\delta}, u_{\varepsilon,\delta})\nonumber\\
&\quad+\Psi_R(\mathcal{K}(\eta_{\varepsilon,\delta})\nabla_\tau(J_{\varepsilon} \eta_{{\varepsilon,\delta}}), u_{{\varepsilon,\delta},\tau})_\Gamma\big)dt\nonumber\\
&\quad+\frac{1}{2}\|Ph(v_{{\varepsilon,\delta}},\nabla\varphi_{{\varepsilon,\delta}})\|_{L_2(\mathcal{H};\mathrm{L}_x^2)}^2dt
+(h(v_{{\varepsilon,\delta}},\nabla\varphi_{{\varepsilon,\delta}}), u_{\varepsilon,\delta})d\mathcal{W}\nonumber\\
&=:\sum_{i=1}^{4}I_{i}dt+(h(v_{{\varepsilon,\delta}},\nabla\varphi_{{\varepsilon,\delta}}), u_{\varepsilon,\delta})d\mathcal{W}.
\end{align}
Integrating of $t$, and then taking supremum, by H\"{o}lder's inequality, Sobolev's embedding $H^2\hookrightarrow L^\infty$, Lemma \ref{lem4.2} and Lemma \ref{lem4.4}, we obtain
\begin{align}
\int_{0}^{T}|I_1|dt&\leq  C\int_{0}^{T}\Psi_R\|u_{\varepsilon,\delta}\|_{\mathrm{L}_x^2}\|v_{\varepsilon,\delta}\|_{\mathrm{L}_x^2}\|\mathcal{J}_{\varepsilon} v_{{\varepsilon,\delta}}\|_{\mathrm{H}^3}dt\nonumber\\
&\leq \int_{0}^{T}\frac{3}{16}\|u_{\varepsilon,\delta}\|_{\mathrm{L}_x^2}^2dt+\frac{C}{\varepsilon^6}\int_{0}^{T}\Psi_R^2\|v_{\varepsilon, \delta}\|_{\mathrm{L}_x^2}^4dt,\nonumber\\
&\leq \int_{0}^{T}\frac{3}{16}\|u_{\varepsilon,\delta}\|_{\mathrm{L}_x^2}^2dt+\frac{CR^4T}{\varepsilon^6},\\
\int_{0}^{T}|I_2|dt&\leq  C\int_{0}^{T}\Psi_R\|u_{\varepsilon,\delta}\|_{\mathrm{L}_x^2}\|\tilde{\mu}_{\varepsilon,\delta}\|_{L_x^2}\|\mathcal{J}_{\varepsilon} \varphi_{{\varepsilon,\delta}}\|_{H^3}dt\nonumber\\
&\leq
\int_{0}^{T}\frac{3}{16}\|u_{\varepsilon,\delta}\|_{\mathrm{L}_x^2}^2dt+\frac{C}{\varepsilon^4}\int_{0}^{T}\Psi_R^2\|\tilde{\mu}_{\varepsilon, \delta}\|_{L_x^2}^2\|\varphi_{\varepsilon, \delta}\|_{H^1}^2dt\nonumber\\
&\leq \int_{0}^{T}\frac{3}{16}\|u_{\varepsilon,\delta}\|_{\mathrm{L}_x^2}^2dt+\frac{CR^2}{\varepsilon^4}\int_{0}^{T}\|\tilde{\mu}_{\varepsilon,\delta}\|_{L_x^2}^2dt,\\
\int_{0}^{T}|I_3|dt&\leq  C\int_{0}^{T}\Psi_R\|u_{\varepsilon,\delta,\tau}\|_{\mathrm{L}^2(\Gamma)}\|\mathcal{K}(\eta_{\varepsilon,\delta})\|_{L^2(\Gamma)}\|\nabla_\tau(J_{\varepsilon} \eta_{{\varepsilon,\delta}})\|_{H^2(\Gamma)}dt\nonumber\\
&\leq  \int_{0}^{T}\frac{1}{2}\|D(u_{\varepsilon,\delta})\|_{\mathrm{L}_x^2}^2+2^\alpha\|u_{\varepsilon,\delta}\|_{\mathrm{L}_x^2}^2dt\nonumber\\
&\quad+\frac{C}{\varepsilon^4}
\int_{0}^{T}\Psi_R^2\|\mathcal{K}(\eta_{\varepsilon,\delta})\|^2_{L^2(\Gamma)}\|\eta_{{\varepsilon,\delta}}\|_{H^1(\Gamma)}^2dt\nonumber\\
&\leq \int_{0}^{T}\frac{1}{2}\|D(u_{\varepsilon,\delta})\|_{\mathrm{L}_x^2}^2+2^\alpha\|u_{\varepsilon,\delta}\|_{\mathrm{L}_x^2}^2dt
+\frac{CR^2}{\varepsilon^4}\int_{0}^{T}\|\mathcal{K}(\eta_{\varepsilon,\delta})\|^2_{L^2(\Gamma)}dt.
\end{align}
Using  Lemma \ref{lem4.2} with $t=2, r=2, r^*=6$, we have
\begin{align}
\int_{0}^{T}\Psi_R^2\|b_0(v_{{\varepsilon,\delta}},\mathcal{J}_{\varepsilon,\delta} v_{{\varepsilon,\delta}})\|_{(H^1)'}^2dt&\leq \int_{0}^{T}\Psi_R^2\|v_{{\varepsilon,\delta}}\|_{\mathrm{H}^{-1,4}}^2\|\nabla\mathcal{J}_{\varepsilon,\delta} v_{{\varepsilon,\delta}}\|_{\mathrm{H}^1}^2dt\nonumber\\ &\leq \frac{C}{\varepsilon^4}\int_{0}^{T}\Psi_R^2\|v_{\varepsilon, \delta}\|_{\mathrm{L}_x^2}^4dt\leq
\frac{CR^4T}{\varepsilon^4}.
\end{align}
Similarly,
\begin{align}
&\int_{0}^{T}\Psi_R^2\|b_1(\tilde{\mu}_{\varepsilon,\delta}, \mathcal{J}_{{\varepsilon}}\varphi_{\varepsilon,\delta}) \|_{(H^1)'}^2dt\leq \frac{CR^2}{\varepsilon^2}\int_{0}^{T}\|\tilde{\mu}_{\varepsilon,\delta}\|_{L_x^2}^2dt,\\
&\int_{0}^{T}\Psi_R^2\|\mathcal{K}(\eta_{\varepsilon,\delta})\nabla_\tau(J_{\varepsilon} \eta_{{\varepsilon,\delta}})\|_{(H^{1}(\Gamma))'}^2dt\leq \frac{CR^2}{\varepsilon^2}\int_{0}^{T}\|\mathcal{K}(\eta_{\varepsilon,\delta})\|^2_{L^2(\Gamma)}dt.
\end{align}

Furthermore, by the Burkholder-Davis-Gundy inequality, for any $p\ge2$, we get
\begin{align}\label{3.11}
&\mathbb{E}\left(\sup_{ t\in [0,T]}\left|\int_{0}^{t}(h(v_{\varepsilon,\delta},\nabla\varphi_{\varepsilon,\delta}), u_{\varepsilon,\delta})d\mathcal{W}\right|\right)^p\nonumber\\
&\leq C\mathbb{E}\left(\int_{0}^{T}\|u_{\varepsilon,\delta}\|_{\mathrm{L}_x^2}^2(1+\|(v_{\varepsilon,\delta}, \nabla\varphi_{\varepsilon,\delta})\|^2_{L_x^2})dt\right)^{\frac{p}{2}}\nonumber\\
&\leq C\mathbb{E}\left(\int_{0}^{T}\|u_{\varepsilon,\delta}\|_{\mathrm{L}_x^2}^4+1+\|(v_{\varepsilon,\delta}, \nabla\varphi_{\varepsilon,\delta})\|^4_{L_x^2}dt\right)^{\frac{p}{2}}\nonumber\\
&\leq CT^{\frac{p}{2}-1}\mathbb{E}\int_{0}^{T}\| u_{\varepsilon, \delta}\|_{\mathrm{L}_x^2}^{2p}dt+CT^{\frac{p}{2}}(1+\|(v_{\varepsilon,\delta}, \varphi_{\varepsilon,\delta})\|_{X}^{2p}).
\end{align}
Taking power $p$, expectation, we deduce from (\ref{3.7})-(\ref{3.11}) that
\begin{align}\label{3.12}
&\mathbb{E}\|u_{\varepsilon, \delta}\|_{L_t^\infty\mathrm{L}_x^2}^{2p}+\mathbb{E}\left(\int_{0}^{T}\|\nabla u_{\varepsilon, \delta}\|_{\mathrm{L}_x^2}^2+\|\partial_t u_{\varepsilon, \delta}\|_{(H^{1})'}^2dt\right)^p\nonumber\\
&\leq\mathbb{E}\|u_0\|_{\mathrm{L}_x^2}^{2p}+\left(\frac{CR^4T}{\varepsilon^6}+\frac{CR^2K_2}{\varepsilon^4}+C(K_1,K_2)T^{\frac{1}{2}}\right)^p\nonumber\\
&\quad+C(p)\left(\sqrt{2}^\alpha T^{1-\frac{2}{p}}+ T^{\frac{1}{2}-\frac{1}{p}}\right)^p\mathbb{E}\int_{0}^{T}\| u_{\varepsilon, \delta}\|_{\mathrm{L}_x^2}^{2p}dt.
\end{align}
In above, we also have used Korn's inequality $\|\nabla u \|_{\mathrm{L}_x^2}\leq \sqrt{2}\|D(u)\|_{\mathrm{L}_x^2}\leq \sqrt{2}\|\nabla u \|_{\mathrm{L}_x^2}$. It follows from \eqref{3.12} and Gronwall's inequality that
 \begin{align*}\label{3.13}
&\mathbb{E}\|u_{\varepsilon, \delta}\|_{L_t^\infty\mathrm{L}_x^2}^{2p}+\mathbb{E}\left(\int_{0}^{T}\|\nabla u_{\varepsilon, \delta}\|_{\mathrm{L}_x^2}^2+\|\partial_t u_{\varepsilon, \delta}\|_{(H^{1})'}^2dt\right)^p\nonumber\\
&\leq\left(C\mathbb{E}\|u_0\|_{\mathrm{L}_x^2}^{2}+\frac{CR^4T}{\varepsilon^6}+\frac{CR^2K_2}{\varepsilon^4}+C(K_1,K_2)T^{\frac{1}{2}}\right)^p e^{CT\left(\sqrt{2}^\alpha T^{1-\frac{2}{p}}+ T^{\frac{1}{2}-\frac{1}{p}}\right)^p}.
\end{align*}
For fixed $R, \varepsilon, \delta$, we choose $K_1$ large and $T$ small enough such that
$$\left(C\mathbb{E}\|u_0\|_{\mathrm{L}_x^2}^{2}+\frac{CR^2K_2}{\varepsilon^4}\right)^p e^{CT\left(\sqrt{2}^\alpha T^{1-\frac{2}{p}}+ T^{\frac{1}{2}-\frac{1}{p}}\right)^p}\leq \frac{K_1}{2},$$
and
$$\left(\frac{CR^4T}{\varepsilon^6} +C(K_1,K_2)T^{\frac{1}{2}}\right)^p e^{CT\left(\sqrt{2}^\alpha T^{1-\frac{2}{p}}+ T^{\frac{1}{2}-\frac{1}{p}}\right)^p}\leq \frac{K_1}{2}.$$

As (\ref{5.4}) in Appendix, taking power $p$, and then expectation, we obtain
\begin{align}
&\mathbb{E}\left(\|(\phi_{\varepsilon,\delta},\psi_{\varepsilon,\delta})\|_{L_t^\infty V^1}^{2p}
+\|\partial_t(\phi_{\varepsilon,\delta},\psi_{\varepsilon,\delta})\|_{L_t^2((H^1)'\times L^2(\Gamma))}^{2p}+\delta\|\phi_{\varepsilon,\delta}\|_{L^2_tL_x^2}^{2p}+\|(\phi_{\varepsilon,\delta},\psi_{\varepsilon,\delta})\|^{2p}_{L_t^2V^2}\right)\nonumber\\
&\quad+\mathbb{E}\left(\|\bar{\mu}_{\varepsilon,\delta}\|_{L_t^2L_x^2}^{2p}+\|\widetilde{\mathcal{K}}(\psi_{\varepsilon,\delta})\|_{L^2_tL^2_x}^{2p}\right)\nonumber\\
&\leq C_p\mathbb{E}\Big(\|(\phi_0,\psi_0)\|_{V^1}^2+\Psi^2_R\|b_2(v_{\varepsilon, \delta}, \mathcal{J}_{\varepsilon}\varphi_{\varepsilon, \delta})\|_{L^2_tL^2_x}^2+\Psi^2_R\|b_\Gamma(v_{\varepsilon, \delta},J_{\varepsilon}\eta_{\varepsilon,\delta})\|_{L_t^2L^2(\Gamma)}^2\nonumber\\&\quad+\|f_\varepsilon(\varphi_{\varepsilon, \delta})\|_{L_t^2L_x^2}^2
+\|g_\varepsilon(\eta_{\varepsilon, \delta})\|_{L_t^2L^2(\Gamma)}^2\Big)^p=:C_p\mathbb{E}\left(\|(\phi_0,\psi_0)\|_{V^1}^2+\sum_{i=1}^4 J_i\right)^p.
\end{align}
We estimate terms $J_1, J_2$ using H\"{o}lder's inequality, Lemma \ref{lem4.2} and Lemma \ref{lem4.4} to find
\begin{align}
\int_{0}^{T}|J_1|dt&\leq \int_{0}^{T}\Psi^2_R\|v_{\varepsilon, \delta}\|_{\mathrm{L}_x^2}^{2}\|\mathcal{J}_{\varepsilon}\varphi_{\varepsilon, \delta}\|_{H^3}^{2}dt\leq \frac{C}{\varepsilon^4}\int_{0}^{T}\Psi^2_R\|v_{\varepsilon, \delta}\|_{\mathrm{L}_x^2}^{2}\|\mathcal{J}_{\varepsilon}\varphi_{\varepsilon, \delta}\|_{H^1}^{2}dt\nonumber\\
&\leq \frac{CR^4T}{\varepsilon^4},\label{3.15}\\
\int_{0}^{T}|J_2|dt&\leq \int_{0}^{T}\Psi^2_R\|v_{\varepsilon, \delta}\|_{\mathrm{L}^2(\Gamma)}^2\|J_{\varepsilon}\eta_{\varepsilon, \delta}\|_{H^3(\Gamma)}^2dt\leq \frac{C}{\varepsilon^4}\int_{0}^{T}\Psi^2_R\|v_{\varepsilon, \delta}\|_{\mathrm{L}^2(\Gamma)}^2\|J_{\varepsilon}\eta_{\varepsilon, \delta}\|_{H^1(\Gamma)}^2dt\nonumber\\
&\leq \frac{C}{\varepsilon^4}\int_{0}^{T}\Psi^2_R(\bar{\varepsilon}\|D(v_{\varepsilon, \delta})\|_{\mathrm{L}_x^2}^2+\bar{\varepsilon}^{-\nu}\|v_{\varepsilon, \delta}\|_{\mathrm{L}_x^2}^2)\|J_{\varepsilon}\eta_{\varepsilon, \delta}\|_{H^1(\Gamma)}^2dt\\
&\leq \frac{CR^4T}{\varepsilon^4\bar{\varepsilon}^{\nu}}+\frac{CR^2\bar{\varepsilon}}{\varepsilon^4}\int_{0}^{T}\|D(v_{\varepsilon, \delta})\|_{\mathrm{L}_x^2}^2dt.\notag
\end{align}
By \eqref{3.7*} and \eqref{3.8**}, we have
\begin{align}\label{3.17}
\int_{0}^{T}|J_3+J_4|dt&\leq C(K_2)T.
\end{align}

Combining estimates \eqref{3.13}-\eqref{3.17}, we find
\begin{align}\label{3.18}
&\mathbb{E}\bigg(\|(\phi_{\varepsilon,\delta},\psi_{\varepsilon,\delta})\|_{L_t^\infty V^1}^{2p}
+\|\partial_t(\phi_{\varepsilon,\delta},\psi_{\varepsilon,\delta})\|_{L_t^2((H^1)'\times L^2(\Gamma))}^{2p}\nonumber\\ &\qquad+\delta\|\phi_{\varepsilon,\delta}\|_{L^2_tL_x^2}^{2p}+\|(\phi_{\varepsilon,\delta},\psi_{\varepsilon,\delta})\|^{2p}_{L_t^2V^2}\bigg)\nonumber\\
&\leq C_p\mathbb{E}\|(\phi_0,\psi_0)\|_{V^1}^{2p}+C(K_1,K_2,p)\left(\frac{R^4T}{\varepsilon^4}+\frac{R^4T}{\varepsilon^4\bar{\varepsilon}^{\nu}}
+\frac{R^2\bar{\varepsilon}}{\varepsilon^4}+T\right)^p.
\end{align}
Choosing $K_2, \bar{\varepsilon}$ such that
\begin{align*}
3C_p\mathbb{E}\|(\phi_0,\psi_0)\|_{V^1}^{2p}\leq K_2, ~ \bar{\varepsilon}<\frac{\varepsilon^4K_2^\frac{1}{p}}{3^\frac{1}{p}C(K_1,K_2,p)R^2},
\end{align*}
and then a small time $T$ such that
$$C(K_1,K_2,p)\left(\frac{R^4T}{\varepsilon^4}+\frac{R^4T}{\varepsilon^4\bar{\varepsilon}^{\nu}}
+T\right)^p\leq \frac{K_2}{3}.$$
We conclude that $\mathcal{M}$ is a mapping from $B_{K_1,K_2}$ into itself from \eqref{3.12}, \eqref{3.18}.

{\bf Step 2}. Mapping $\mathcal{M}$ is contraction.

Let $(v_{\varepsilon, \delta}^i,\varphi_{\varepsilon, \delta}^i, \eta_{\varepsilon, \delta}^i)\in B_{K_1,K_2}, i=1,2$ and $(u_{\varepsilon, \delta}^i,\phi_{\varepsilon, \delta}^i, \psi_{\varepsilon, \delta}^i)=\mathcal{M}(v_{\varepsilon, \delta}^i,\varphi_{\varepsilon, \delta}^i, \eta_{\varepsilon, \delta}^i)$. Denote by $(\bar{u}_{\varepsilon, \delta},\bar{\phi}_{\varepsilon, \delta}, \bar{\psi}_{\varepsilon, \delta})$, $(\bar{v}_{\varepsilon, \delta},\bar{\varphi}_{\varepsilon, \delta}, \bar{\eta}_{\varepsilon, \delta})$ the differences of $(u_{\varepsilon, \delta}^i,\phi_{\varepsilon, \delta}^i, \psi_{\varepsilon, \delta}^i)$, $(v_{\varepsilon, \delta}^i,\varphi_{\varepsilon, \delta}^i, \eta_{\varepsilon, \delta}^i)$ satisfying
\begin{eqnarray}\label{diff}
\left\{\begin{array}{ll}
\!\!d\bar{u}_{\varepsilon, \delta}-{\rm div}(2 D(\bar{u}_{\varepsilon, \delta}))dt+\Psi_R^1(b_0( v^1_{\varepsilon, \delta},\mathcal{J}_\varepsilon \bar{v}_{\varepsilon, \delta})+b_0(\bar{v}_{\varepsilon, \delta},\mathcal{J}_\varepsilon v^2_{\varepsilon, \delta} ))dt\\ \qquad\quad+\tilde{\Phi}_R b_0(v^2_{\varepsilon, \delta},\mathcal{J}_\varepsilon v^2_{\varepsilon, \delta} )dt-(\Psi_R^1\tilde{\mu}^1_{\varepsilon, \delta}\nabla \mathcal{J}_{\varepsilon}\varphi^1_{\varepsilon, \delta}-\Psi_R^2\tilde{\mu}^2_{\varepsilon, \delta}\nabla \mathcal{J}_{\varepsilon}\varphi^2_{\varepsilon, \delta})dt\\
\qquad\quad=(h(v^1_{\varepsilon, \delta},\nabla\varphi^1_{\varepsilon, \delta})-h(v^2_{\varepsilon, \delta},\nabla\varphi^2_{\varepsilon, \delta}))d\mathcal{W},\\
\!\!(D(\bar{u}_{\varepsilon, \delta})\cdot n)_{\tau}+\bar{u}
_{\varepsilon, \delta,\tau}\\ \qquad\quad=\Psi_R^1(\overline{\mathcal{K}}(\bar{\eta}_\varepsilon)\nabla_\tau(J_\varepsilon \eta^1_{\varepsilon, \delta})+\mathcal{K}(\eta^2_{\varepsilon, \delta}) \nabla_\tau(J_\varepsilon \bar{\eta}
_{\varepsilon, \delta}))+\tilde{\Phi}_R\mathcal{K}(\eta^2_{\varepsilon, \delta}) \nabla_\tau(J_\varepsilon \eta^2
_{\varepsilon, \delta}),\\
\!\!\partial_t\bar{\phi}_{\varepsilon, \delta}+A_1\bar{\mu}_{\varepsilon, \delta}+\Psi_R^1(b_2(\bar{v}_{\varepsilon, \delta}, \mathcal{J}_{\varepsilon}\phi^1_{\varepsilon, \delta})+b_2(v^2_{\varepsilon, \delta}, \mathcal{J}_{\varepsilon}\bar{\phi}_{\varepsilon, \delta}))+\tilde{\Phi}_Rb_2(v^2_{\varepsilon, \delta}, \mathcal{J}_{\varepsilon}\phi^2_{\varepsilon, \delta})=0,\\
\!\!\partial_t\bar{\psi}_{\varepsilon, \delta}+\Psi_R^1(b_\Gamma(\bar{v}_{\varepsilon, \delta},J_{\varepsilon}(\eta^1_{\varepsilon, \delta}))+b_\Gamma(v^2_{\varepsilon, \delta},J_{\varepsilon}(\bar{\eta}_{\varepsilon, \delta})))+\tilde{\Phi}_Rb_\Gamma(v^2_{\varepsilon, \delta},J_{\varepsilon}(\eta^2_{\varepsilon, \delta}))
+\widetilde{\mathcal{K}}(\bar{\psi}_{\varepsilon, \delta})\\ \qquad\quad=-(g(\eta^1_{\varepsilon, \delta})-g(\eta^2_{\varepsilon, \delta})),\\
\!\! (\bar{u}_{\varepsilon, \delta}(0), \bar{\phi}_{\varepsilon, \delta}(0), \bar{\psi}_{\varepsilon, \delta}(0))=(0,0,0),
\end{array}\right.
\end{eqnarray}
with
\begin{eqnarray*}
&&\bar{\mu}_{\varepsilon, \delta}\!=\!-\Delta\bar{\phi}_{\varepsilon, \delta}+\delta\partial_t\bar{\phi}_{\varepsilon, \delta}+f_\varepsilon(\varphi^1_{\varepsilon, \delta})-f_\varepsilon(\varphi^2_{\varepsilon, \delta}),\\
&&\overline{\mathcal{K}}(\bar{\eta}_{\varepsilon, \delta})\!=\!A_{\tau}\bar{\eta}_{\varepsilon, \delta}+\partial_n\bar{\varphi}_{\varepsilon, \delta}+\bar{\eta}_{\varepsilon, \delta}
+g_\varepsilon(\eta^1_{\varepsilon, \delta})-g_\varepsilon(\eta^2_{\varepsilon, \delta}),
\end{eqnarray*}
and
\begin{eqnarray*}
&& \tilde{\Phi}_R:=\Psi_R^1-\Psi_R^2,\\
&&  \Psi_R^1=\overline{\Psi}_{R}(\|v_{\varepsilon, \delta}^1\|_{\mathrm{L}_x^2})\cdot\overline{\Psi}_{R}(\|(\varphi_{\varepsilon, \delta}^1, \eta_{\varepsilon, \delta}^1)\|_{ V^1}),\\  &&\Psi_R^2=\overline{\Psi}_{R}(\|v_{\varepsilon, \delta}^2\|_{\mathrm{L}_x^2})\cdot\overline{\Psi}_{R}(\|(\varphi_{\varepsilon, \delta}^2, \eta_{\varepsilon, \delta}^2)\|_{ V^1}).
\end{eqnarray*}

Using It\^{o}'s formula to the function $\frac{1}{2}\|\bar{u}_{\varepsilon, \delta}\|_{\mathrm{L}_x^2}^2$, we have
\begin{align}\label{3.21}
&\frac{1}{2}d\|\bar{u}_{\varepsilon, \delta}\|_{\mathrm{L}_x^2}^2+a_0(\bar{u}_{\varepsilon, \delta}, \bar{u}_{\varepsilon, \delta})dt\nonumber\\
&=\big(-\Psi_R^1B_0(v^1_{\varepsilon, \delta},\mathcal{J}_\varepsilon \bar{v}_{\varepsilon, \delta},  \bar{u}_{\varepsilon, \delta})-\Psi_R^1B_0( \bar{v}_{\varepsilon, \delta},\mathcal{J}_\varepsilon v^2_{\varepsilon, \delta},  \bar{u}_{\varepsilon, \delta})-\tilde{\Phi}_RB_0(v^2_{\varepsilon, \delta},\mathcal{J}_\varepsilon v^2_{\varepsilon, \delta}, \bar{u}_{\varepsilon, \delta})\nonumber\\
&\quad+\Psi_R^1(\Delta \bar{\varphi}_{\varepsilon, \delta} \nabla(\mathcal{J}_{\varepsilon}\varphi^1_{\varepsilon, \delta}),\bar{u}_{\varepsilon, \delta})
+\Psi_R^1(\Delta \varphi^2_{\varepsilon, \delta} \nabla(\mathcal{J}_{\varepsilon}\bar{\varphi}_{\varepsilon, \delta}),\bar{u}_{\varepsilon, \delta})
-\delta\Psi_R^1(\partial_t\bar{\varphi}_{\varepsilon, \delta}\nabla(\mathcal{J}_{\varepsilon}\varphi^1_{\varepsilon, \delta}),\bar{u}_{\varepsilon, \delta})\nonumber\\
&\quad-\delta\Psi_R^1(\partial_t\varphi^2_{\varepsilon, \delta}\nabla(\mathcal{J}_{\varepsilon}\bar{\varphi}_{\varepsilon, \delta}),\bar{u}_{\varepsilon, \delta})
-\Psi_R^1((f_\varepsilon(\varphi^1_{\varepsilon, \delta})-f_\varepsilon(\varphi^2_{\varepsilon, \delta}))\nabla(\mathcal{J}_{\varepsilon}\varphi^1_{\varepsilon, \delta}),\bar{u}_{\varepsilon, \delta})\nonumber\\
&\quad-\Psi_R^1(f_\varepsilon(\varphi^2_{\varepsilon, \delta})\nabla(\mathcal{J}_{\varepsilon}\bar{\varphi}_{\varepsilon, \delta}),\bar{u}_{\varepsilon, \delta})
-\tilde{\Phi}_R(\mu^2_{\varepsilon, \delta}\nabla \mathcal{J}_{\varepsilon}\varphi^2_{\varepsilon, \delta}, \bar{u}_{\varepsilon, \delta})+\Psi_R^1(\overline{\mathcal{K}}(\bar{\eta}_{\varepsilon, \delta})\nabla_\tau(J_\varepsilon \eta^1_{\varepsilon, \delta}), \bar{u}_{\varepsilon, \delta,\tau})_\Gamma\nonumber\\
&\quad+\Psi_R^1(\mathcal{K}(\eta^2_{\varepsilon, \delta}) \nabla_\tau(J_\varepsilon \bar{\eta}_{\varepsilon, \delta}), \bar{u}_{\varepsilon, \delta,\tau})_\Gamma+\tilde{\Phi}_R(\mathcal{K}(\eta^2_{\varepsilon, \delta}) \nabla_\tau(J_\varepsilon \eta^2
_{\varepsilon, \delta}),  \bar{u}_{\varepsilon, \delta,\tau})\big)dt\nonumber\\
&\quad+\frac{1}{2}\|h(v^1_{\varepsilon, \delta},\nabla\varphi^1_{\varepsilon, \delta})-h(v^2_{\varepsilon, \delta},\nabla\varphi^2_{\varepsilon, \delta})\|_{L_2(\mathcal{H},\mathrm{L}_x^2)}^2dt\nonumber\\
&\quad+(h(v^1_{\varepsilon, \delta},\nabla\varphi^1_{\varepsilon, \delta})-h(v^2_{\varepsilon, \delta},\nabla\varphi^2_{\varepsilon, \delta}), \bar{u}_{\varepsilon, \delta})d\mathcal{W}\nonumber\\
&=:\sum_{k=1}^{14}H_{k}dt+(h(v^1_{\varepsilon, \delta},\nabla\varphi^1_{\varepsilon, \delta})-h(v^2_{\varepsilon, \delta},\nabla\varphi^2_{\varepsilon, \delta}), \bar{u}_{\varepsilon, \delta})d\mathcal{W}.
\end{align}
We proceed to estimate all terms $H_k, k=1,\cdots, 14$. By H\"{o}lder's inequality, the embedding $H^2\hookrightarrow L^\infty$ and Lemma \ref{lem4.4}, we deduce
\begin{align}
|H_1|&\leq C\Psi_R^1\|\bar{u}_{\varepsilon, \delta}\|_{\mathrm{H}^1}\|v^1_{\varepsilon, \delta}\|_{\mathrm{L}_x^2}\|\nabla\mathcal{J}_\varepsilon \bar{v}_{\varepsilon, \delta}\|_{\mathrm{H}^2}\nonumber\\ &\leq \varepsilon_1a_0(\bar{u}_{\varepsilon, \delta}, \bar{u}_{\varepsilon, \delta})+\frac{C}{\varepsilon_1\varepsilon^6}\Psi_R^1\| \bar{v}_{\varepsilon, \delta}\|_{\mathrm{L}_x^2}^2\| v^1_{\varepsilon, \delta}\|_{\mathrm{L}_x^2}^2,\label{3.22}\\
|H_2|&\leq C\Psi_R^1\|\bar{u}_{\varepsilon, \delta}\|_{\mathrm{H}^1}\|\bar{v}_{\varepsilon, \delta}\|_{\mathrm{L}_x^2}\|\nabla\mathcal{J}_\varepsilon v^2_{\varepsilon, \delta}\|_{\mathrm{H}^2}\nonumber\\ &\leq \varepsilon_1a_0(\bar{u}_{\varepsilon, \delta}, \bar{u}_{\varepsilon, \delta})+\frac{C}{\varepsilon_1\varepsilon^6}\Psi_R^1\|\bar{ v}_{\varepsilon, \delta}\|_{\mathrm{L}_x^2}^2\|v^2_{\varepsilon, \delta}\|_{\mathrm{L}_x^2}^2,\\
|H_4|&\leq C\Psi_R^1\|\bar{u}_{\varepsilon, \delta}\|_{\mathrm{L}_x^2}\|\Delta\bar{\varphi}_{{\varepsilon, \delta}}\|_{L_x^2}\|\nabla\mathcal{J}_{{\varepsilon}}\varphi^1_{{\varepsilon, \delta}}\|_{H^2}\nonumber\\ &\leq \frac{C}{{\varepsilon}_1{\varepsilon}^4}\|\bar{u}_{\varepsilon, \delta}\|_{\mathrm{L}_x^2}^2+{\varepsilon}_1\Psi_R^1\|\bar{\varphi}_{\varepsilon, \delta}\|_{V^2}^2\|\varphi_{{\varepsilon, \delta}}^1\|_{H^1}^2,\\
|H_5|&\leq \frac{C}{{\varepsilon}^2}\Psi_R^1\|\bar{u}_{\varepsilon, \delta}\|_{\mathrm{L}_x^2}\|\varphi^2_{\varepsilon, \delta}\|_{V^2}\|\bar{\varphi}_{{\varepsilon, \delta}}\|_{H^1}\nonumber\\ &\leq \frac{C}{{\varepsilon}_1{\varepsilon}^4}\|\bar{u}_{\varepsilon, \delta}\|_{\mathrm{L}_x^2}^2
+{\varepsilon}_1\Psi_R^1\|\varphi^2_{\varepsilon, \delta}\|_{V^2}^2\|\bar{\varphi}_{{\varepsilon, \delta}}\|_{H^1}^2,
\end{align}
and
\begin{align}
&|H_6+H_7|\\
&\leq \frac{C\delta}{{\varepsilon}^2}\Psi_R^1\|\bar{u}_{\varepsilon, \delta}\|_{\mathrm{L}_x^2}\|\partial_t\bar{\varphi}_{\varepsilon, \delta}\|_{L_x^2}\|\varphi_{\varepsilon, \delta}^1 \|_{H^1}+\frac{C\delta}{{\varepsilon}^2}\Psi_R^1\|\bar{u}_{\varepsilon, \delta}\|_{\mathrm{L}_x^2}\|\partial_t\varphi_{\varepsilon, \delta}^2\|_{L_x^2}\|\bar{\varphi}_{\varepsilon, \delta} \|_{H^1}\nonumber\\
&\leq \frac{C}{{\varepsilon}_1{\varepsilon}^4}\|\bar{u}_{\varepsilon, \delta}\|_{\mathrm{L}_x^2}^2
+{\varepsilon}_1(\delta\Psi_R^1)^2\|\partial_t\bar{\varphi}_{\varepsilon, \delta}\|_{L_x^2}^2\|\varphi_{\varepsilon, \delta}^1 \|_{H^1}^2+{\varepsilon}_1(\delta\Psi_R^1)^2\|\partial_t\varphi_{\varepsilon, \delta}^2\|_{L_x^2}^2\|\bar{\varphi}_{\varepsilon, \delta} \|_{H^1}^2,
\end{align}
where $\varepsilon_1$ is a constant selected later.
 H\"{o}lder's inequality, Property (i), the mean value theorem, and Lemma \ref{lem4.4} yield
\begin{align}
|H_8|&\leq C\Psi_R^1\|\bar{u}_{\varepsilon, \delta}\|_{\mathrm{L}_x^2}\|f_\varepsilon(\varphi^1_{\varepsilon, \delta})-f_\varepsilon(\varphi^2_{\varepsilon, \delta})\|_{L_x^2}\|\nabla \mathcal{J}_\varepsilon v^2_{\varepsilon, \delta}\|_{\mathrm{H}^2}\nonumber\\
&\leq C\|\bar{u}_{\varepsilon, \delta}\|_{\mathrm{L}_x^2}\|\bar{\varphi}_{\varepsilon,\delta}\|_{L_x^2}\|\nabla \mathcal{J}_\varepsilon v^2_{\varepsilon, \delta}\|_{\mathrm{H}^2}\nonumber\\
&\leq \frac{C}{\varepsilon^3}\|\bar{u}_{\varepsilon, \delta}\|_{\mathrm{L}_x^2}\|\bar{\varphi}_{\varepsilon,\delta}\|_{L_x^2}
\| v^2_{\varepsilon, \delta}\|_{\mathrm{L}_x^2}\nonumber\\
&\leq \frac{1}{4}\|\bar{u}_{\varepsilon, \delta}\|_{\mathrm{L}_x^2}^2+\frac{C}{\varepsilon^6}\|\bar{\varphi}_{\varepsilon, \delta}\|_{L_x^2}^2\|v^2_{\varepsilon, \delta}\|_{\mathrm{L}_x^2}^2.
\end{align}
The term $H_9$ could be bounded easily. For term $H_{12}$, by H\"{o}lder's inequality, Lemma \ref{lem4.2}, and Lemma \ref{lem4.4}, we deduce
\begin{align}
|H_{12}|&\leq C\Psi_R^1\|\bar{u}_{{\varepsilon, \delta}}\|_{\mathrm{L}^2(\Gamma)}\|\mathcal{K}(\eta^2_{{\varepsilon, \delta}})\|_{L^2(\Gamma)}\|\nabla_\tau(J_{\varepsilon} \bar{\eta}_{{\varepsilon, \delta}})\|_{H^2(\Gamma)}\nonumber\\
&\leq \frac{C}{{\varepsilon}^3}\|\eta
_{\varepsilon, \delta}^2\|_{V^2}\|\bar{u}_{{\varepsilon, \delta}}\|_{\mathrm{L}^2(\Gamma)}\|\bar{\eta}_{\varepsilon, \delta}\|_{L^2(\Gamma)}\nonumber\\
&\leq\frac{C}{{\varepsilon}^3}\|\eta_{\varepsilon, \delta}^2\|_{V^2}\|\bar{\eta}_{\varepsilon, \delta}\|_{L^2(\Gamma)}\left(\tilde{{\varepsilon}}\|D(\bar{u}_{\varepsilon, \delta})\|_{\mathrm{L}_x^2}
+\tilde{{\varepsilon}}^{-\nu}\|\bar{u}_{\varepsilon, \delta}\|_{\mathrm{L}_x^2}\right)\\
&\leq\frac{C}{{\varepsilon}_1{\varepsilon}^6\tilde{{\varepsilon}}^{2\nu}}\|\bar{u}_{\varepsilon, \delta}\|_{\mathrm{L}_x^2}^2+\frac{C\tilde{{\varepsilon}}^2}{{\varepsilon}^6{\varepsilon}_1}a_0(\bar{u}_{\varepsilon, \delta}, \bar{u}_{\varepsilon, \delta})+{\varepsilon}_1\|\eta_{\varepsilon, \delta}^2\|_{V^2}^2\|\bar{\eta}_{{\varepsilon, \delta}}\|_{L^2(\Gamma)}^2,\nonumber
\end{align}
we choose $\tilde{\varepsilon}=\frac{\varepsilon^3\varepsilon_1^2}{\sqrt{C}}$. For term $H_{11}$, using \eqref{3.8**}, Lemma \ref{lem4.2} and H\"{o}lder's inequality, we obtain
\begin{align}
|H_{11}|&\leq C\Psi_R^1\|\bar{u}_{{\varepsilon, \delta}}\|_{\mathrm{L}^2(\Gamma)}(\|\widetilde{\mathcal{K}}(\bar{\eta}_{{\varepsilon, \delta}})\|_{L^2(\Gamma)}
+\|g_\varepsilon(\eta^1_{\varepsilon, \delta})-g_\varepsilon(\eta^2_{\varepsilon, \delta})\|_{L^2(\Gamma)})\|\nabla_\tau(J_{\varepsilon} \eta^1_{{\varepsilon, \delta}})\|_{H^2(\Gamma)}\nonumber\\
&\leq\frac{C}{{\varepsilon}^3}\|\eta^1_{\varepsilon, \delta}\|_{L^2(\Gamma)}\left(\tilde{{\varepsilon}}\|D(\bar{u}_{\varepsilon, \delta})\|_{\mathrm{L}_x^2}
+\tilde{{\varepsilon}}^{-\nu}\|\bar{u}_{\varepsilon, \delta}\|_{\mathrm{L}_x^2}\right)\|(\bar{\varphi}_{\varepsilon, \delta}, \bar{\eta}_{\varepsilon, \delta})\|_{V^2}\nonumber\\
&\quad+\frac{C}{{\varepsilon}^3}\|\eta^1_{\varepsilon, \delta}\|_{L^2(\Gamma)}\left(\frac{1}{2}\|D(\bar{u}_{\varepsilon, \delta})\|_{\mathrm{L}_x^2}
+C_1\|\bar{u}_{\varepsilon, \delta}\|_{\mathrm{L}_x^2}\right)\|\bar{\eta}_{{\varepsilon, \delta}}\|_{L^2(\Gamma)}\nonumber\\
&\leq\frac{C}{{\varepsilon}_1{\varepsilon}^6}\left(\tilde{{\varepsilon}}^2\|D(\bar{u}_{\varepsilon, \delta})\|^2_{\mathrm{L}_x^2}
+\tilde{{\varepsilon}}^{-2\nu}\|\bar{u}_{\varepsilon, \delta}\|^2_{\mathrm{L}_x^2}\right)+{\varepsilon}_1\|(\bar{\varphi}_{\varepsilon, \delta}, \bar{\eta}_{\varepsilon, \delta})\|_{V^2}^2\|\eta^1_{{\varepsilon, \delta}}\|_{L^2(\Gamma)}^2\nonumber\\
&\quad+\frac{1}{4}\|\bar{u}_{\varepsilon, \delta}\|_{\mathrm{L}_x^2}^2+\frac{1}{2}a_0(\bar{u}_{\varepsilon, \delta}, \bar{u}_{\varepsilon, \delta})+\frac{C}{{\varepsilon}^6}\|\eta_{\varepsilon, \delta}^1\|_{L^2(\Gamma)}^2\|\bar{\eta}_{{\varepsilon, \delta}}\|_{L^2(\Gamma)}^2.
\end{align}
For term $H_{14}$, by assumption ${\bf A}_2$, we have
$$|H_{14}|\leq C\|(v_{\varepsilon, \delta}, \nabla\varphi_{\varepsilon, \delta})\|_{L_x^2}^2.$$

To surmount the difficulty caused by cut-off function, we further define stopping time $$\tau_R^i:=\inf\left\{t\geq 0; \|(v^i_{\varepsilon, \delta}(t), \varphi_{\varepsilon, \delta}^i(t), \eta_{\varepsilon, \delta}^i(t))\|_{L_x^2\times V^1}\geq 2R\right\}$$ for $i=1,2$. If the right-hand side set is empty, taking $\tau_R^i=T$. Indeed, the definition of stopping time $\tau_R^i$ is well-defined, since $(v^i_{\varepsilon, \delta}, \varphi_{\varepsilon, \delta}^i, \eta_{\varepsilon, \delta}^i)$ belongs to space $X$, we have the continuity of time in space $L_x^2\times V^1$. Without loss of generality, we could assume that $\tau_R^1\leq \tau_R^2$. Consequently, if $t\geq \tau_R^2$, we have
$$\overline{\Psi}_{R}(\|(v^i_{\varepsilon, \delta}, \varphi_{\varepsilon, \delta}^i, \eta_{\varepsilon, \delta}^i)\|_{L_x^2\times V^1})=0,\; i=1,2.$$
For term $H_3$, by the mean value theorem, Lemma \ref{lem4.4} and the Schwartz inequality, we have
\begin{align}
&\mathbb{E}\left(\int_{0}^{T}|H_3|dt\right)^p\nonumber\\&\leq \mathbb{E}\Bigg(\int_{0}^{T}\!\!\left|\overline{\Psi}_{R}(\|v_{\varepsilon, \delta}^1\|_{\mathrm{L}_x^2})-\overline{\Psi}_{R}(\|v_{\varepsilon, \delta}^2\|_{\mathrm{L}_x^2})\right|\overline{\Psi}_{R}(\|(\varphi_{\varepsilon, \delta}^1, \eta_{\varepsilon, \delta}^1)\|_{ V^1})B_0(v^2_{\varepsilon, \delta},\mathcal{J}_\varepsilon v^2_{\varepsilon, \delta}, \bar{u}_{\varepsilon, \delta})\nonumber\\
&\qquad+\overline{\Psi}_{R}(\|v_{\varepsilon, \delta}^2\|_{\mathrm{L}_x^2})\left|\overline{\Psi}_{R}(\|(\varphi_{\varepsilon, \delta}^1, \eta_{\varepsilon, \delta}^1)\|_{ V^1})-\overline{\Psi}_{R}(\|(\varphi_{\varepsilon, \delta}^1, \eta_{\varepsilon, \delta}^1)\|_{ V^1})\right|\nonumber\\
&\qquad \times B_0(v^2_{\varepsilon, \delta},\mathcal{J}_\varepsilon v^2_{\varepsilon, \delta}, \bar{u}_{\varepsilon, \delta})dt\Bigg)^p\nonumber\\
&\leq C\mathbb{E}\Bigg(\int_{0}^{\tau_R^2}\|\mathcal{J}_\varepsilon v^2_{\varepsilon, \delta}\|_{\mathrm{H}^3}\|\bar{u}_{\varepsilon, \delta}\|_{\mathrm{L}_x^2}\|v^2_{\varepsilon, \delta}\|_{\mathrm{L}_x^2}\nonumber\\ &\qquad\times(\|\bar{v}_{{\varepsilon},\delta}\|_{L_t^\infty\mathrm{L}_x^2}+\overline{\Psi}_{R}(\|v_{\varepsilon, \delta}^2\|_{\mathrm{L}_x^2})\|(\bar{\varphi}_{\varepsilon, \delta}, \bar{\eta}_{\varepsilon, \delta})\|_{L_t^\infty V^1})dt\Bigg)^p\nonumber\\
&\leq C\mathbb{E}\left(\int_{0}^{\tau_R^2}\|\bar{u}_{\varepsilon, \delta}\|_{L_t^\infty\mathrm{L}_x^2}^2+\|(\bar{\varphi}_{\varepsilon, \delta}, \bar{\eta}_{\varepsilon, \delta})\|_{L_t^\infty V^1}^2+\frac{1}{\varepsilon^6}\|\bar{v}_{{\varepsilon},\delta}\|_{\mathrm{L}_x^2}^2\|v^2_{\varepsilon, \delta}\|_{\mathrm{L}_x^2}^4+\frac{R^4}{\varepsilon^6}\|\bar{v}_{{\varepsilon},\delta}\|_{\mathrm{L}_x^2}^2dt\right)^p\nonumber\\
&\leq C(\tau_R^2)^p\mathbb{E}\left(\|\bar{u}_{\varepsilon, \delta}\|_{L^\infty_t\mathrm{L}_x^2}^{2p}+\|(\bar{\varphi}_{\varepsilon, \delta}, \bar{\eta}_{\varepsilon, \delta})\|_{L^\infty_tV^1}^{2p}\right)+\frac{C(\tau_R^2)^pR^{4p}}{\varepsilon^{6p}}\mathbb{E}\|\bar{v}_{\varepsilon, \delta}\|_{L^\infty_t\mathrm{L}_x^2}^{2p}\nonumber\\
&\leq CT^p\mathbb{E}\left(\|\bar{u}_{\varepsilon, \delta}\|_{L^\infty_t\mathrm{L}_x^2}^{2p}+\|(\bar{\varphi}_{\varepsilon, \delta}, \bar{\eta}_{\varepsilon, \delta})\|_{L^\infty_tV^1}^{2p}\right)+\frac{CT^pR^{4p}}{\varepsilon^{6p}}\mathbb{E}\|\bar{v}_{\varepsilon, \delta}\|_{L^\infty_t\mathrm{L}_x^2}^{2p}.
\end{align}
Similarly, we have
\begin{align}
&\mathbb{E}\left(\int_{0}^{T}|H_{10}|dt\right)^p\nonumber\\&\leq \mathbb{E}\Bigg(\int_{0}^{T}\!\!\left|\overline{\Psi}_{R}(\|v_{\varepsilon, \delta}^1\|_{\mathrm{L}_x^2})-\overline{\Psi}_{R}(\|v_{\varepsilon, \delta}^2\|_{\mathrm{L}_x^2})\right|\overline{\Psi}_{R}(\|(\varphi_{\varepsilon, \delta}^1, \eta_{\varepsilon, \delta}^1)\|_{V^1})(\mu^2_{\varepsilon, \delta}\nabla \mathcal{J}_{\varepsilon}\varphi^2_{\varepsilon, \delta}, \bar{u}_{\varepsilon, \delta})\nonumber\\
&\qquad+\overline{\Psi}_{R}(\|v_{\varepsilon, \delta}^2\|_{\mathrm{L}_x^2})\left|\overline{\Psi}_{R}(\|(\varphi_{\varepsilon, \delta}^1, \eta_{\varepsilon, \delta}^1)\|_{ V^1})-\overline{\Psi}_{R}(\|(\varphi_{\varepsilon, \delta}^1, \eta_{\varepsilon, \delta}^1)\|_{ V^1})\right|\nonumber\\
&\qquad \times(\mu^2_{\varepsilon, \delta}\nabla \mathcal{J}_{\varepsilon}\varphi^2_{\varepsilon, \delta}, \bar{u}_{\varepsilon, \delta})dt\Bigg)^p\nonumber\\
&\leq C\mathbb{E}\left(\int_{0}^{\tau_R^2}(\|\bar{v}_{{\varepsilon},\delta}\|_{L_t^\infty\mathrm{L}_x^2}+\|(\bar{\varphi}_{\varepsilon, \delta}, \bar{\eta}_{\varepsilon, \delta})\|_{L_t^\infty V^1})\|\bar{u}_{\varepsilon, \delta}\|_{\mathrm{L}_x^2}\|\mu^2_{\varepsilon, \delta}\|_{L_x^2}\|\mathcal{J}_\varepsilon \varphi^2_{\varepsilon, \delta}\|_{H^3}dt\right)^p\nonumber\\
&\leq \frac{C(\tau_R^2)^p}{\varepsilon_1}\mathbb{E}\left(\|\bar{u}_{\varepsilon, \delta}\|_{L^\infty_t\mathrm{L}_x^2}^{2p}+\|(\bar{\varphi}_{\varepsilon, \delta}, \bar{\eta}_{\varepsilon, \delta})\|_{L^\infty_tV^1}^{2p}\right)+\frac{\varepsilon_1R^{4p}}{\varepsilon^{4p}}\mathbb{E}\|\bar{v}_{\varepsilon, \delta}\|_{L^\infty_t\mathrm{L}_x^2}^{2p}\nonumber\\
&\leq \frac{CT^p}{\varepsilon_1}\mathbb{E}\left(\|\bar{u}_{\varepsilon, \delta}\|_{L^\infty_t\mathrm{L}_x^2}^{2p}+\|(\bar{\varphi}_{\varepsilon, \delta}, \bar{\eta}_{\varepsilon, \delta})\|_{L^\infty_tV^1}^{2p}\right)+\frac{\varepsilon_1R^{4p}}{\varepsilon^{4p}}\mathbb{E}\|\bar{v}_{\varepsilon, \delta}\|_{L^\infty_t\mathrm{L}_x^2}^{2p},
\end{align}
and
\begin{align}
&\mathbb{E}\left(\int_{0}^{T}|H_{13}|dt\right)^p\nonumber \\ &\leq \mathbb{E}\Bigg(\int_{0}^{T}\left|\overline{\Psi}_{R}(\|v_{\varepsilon, \delta}^1\|_{\mathrm{L}_x^2})-\overline{\Psi}_{R}(\|v_{\varepsilon, \delta}^2\|_{\mathrm{L}_x^2})\right|\overline{\Psi}_{R}(\|(\varphi_{\varepsilon, \delta}^1, \eta_{\varepsilon, \delta}^1)\|_{ V^1})\nonumber\\
&\qquad\times(\mathcal{K}(\eta^2_{\varepsilon, \delta}) \nabla_\tau(J_\varepsilon \eta^2
_{\varepsilon, \delta}),  \bar{u}_{\varepsilon, \delta,\tau})\nonumber\\
&\qquad+\overline{\Psi}_{R}(\|v_{\varepsilon, \delta}^2\|_{\mathrm{L}_x^2})\left|\overline{\Psi}_{R}(\|(\varphi_{\varepsilon, \delta}^1, \eta_{\varepsilon, \delta}^1)\|_{ V^1})-\overline{\Psi}_{R}(\|(\varphi_{\varepsilon, \delta}^1, \eta_{\varepsilon, \delta}^1)\|_{V^1})\right|\nonumber\\
&\qquad\times(\mathcal{K}(\eta^2_{\varepsilon, \delta}) \nabla_\tau(J_\varepsilon \eta^2
_{\varepsilon, \delta}),  \bar{u}_{\varepsilon, \delta,\tau})dt\Bigg)^p\nonumber\\
&\leq C\mathbb{E}\Bigg(\int_{0}^{\tau_R^2}\|\bar{u}_{\varepsilon, \delta}\|_{\mathrm{L}^2(\Gamma)}\|\mathcal{K}(\eta^2_{{\varepsilon, \delta}})\|_{L^2(\Gamma)}\|\nabla_\tau(J_{\varepsilon} \eta^2_{{\varepsilon, \delta}})\|_{H^2(\Gamma)}\nonumber\\
&\qquad\times(\|\bar{v}_{{\varepsilon},\delta}\|_{L_t^\infty\mathrm{L}_x^2}+\|(\bar{\varphi}_{\varepsilon, \delta}, \bar{\eta}_{\varepsilon, \delta})\|_{L_t^\infty V^1})dt\Bigg)^p\nonumber\\
&\leq  \mathbb{E}\Bigg(\int_{0}^{\tau_R^2}\frac{C}{{\varepsilon}^3}(\|\bar{v}_{{\varepsilon},\delta}\|_{L_t^\infty\mathrm{L}_x^2}+\|(\bar{\varphi}_{\varepsilon, \delta}, \bar{\eta}_{\varepsilon, \delta})\|_{L_t^\infty V^1})\|\eta
_{\varepsilon, \delta}^2\|_{V^2}\|\bar{u}_{{\varepsilon, \delta}}\|_{\mathrm{L}^2(\Gamma)}\|\eta^2_{\varepsilon, \delta}\|_{L^2(\Gamma)}dt\Bigg)^p\nonumber\\
&\leq  \mathbb{E}\Bigg(\int_{0}^{\tau_R^2}\frac{C}{{\varepsilon}^3}(\|\bar{v}_{{\varepsilon},\delta}\|_{L_t^\infty\mathrm{L}_x^2}+\|(\bar{\varphi}_{\varepsilon, \delta}, \bar{\eta}_{\varepsilon, \delta})\|_{L_t^\infty V^1})\nonumber\\
&\qquad \times(\tilde{{\varepsilon}}\|D(\bar{u}_{\varepsilon, \delta})\|_{\mathrm{L}_x^2}
+\tilde{{\varepsilon}}^{-\nu}\|\bar{u}_{\varepsilon, \delta}\|_{\mathrm{L}_x^2})\|\eta
_{\varepsilon, \delta}^2\|_{V^2}\|\eta^2_{\varepsilon, \delta}\|_{L^2(\Gamma)}dt\Bigg)^p\nonumber\\
&\leq \frac{(\tau_R^2)^p}{{\varepsilon}_1\tilde{{\varepsilon}}^{2p\nu}}\mathbb{E}\|\bar{u}_{\varepsilon, \delta}\|_{L_t^\infty\mathrm{L}_x^2}^{2p}\!+\!\frac{\tilde{\varepsilon}^{2p}}{\varepsilon^{6p}}\mathbb{E}\left(\int_{0}^{\tau_R^2}\!\!a_0(\bar{u}_{\varepsilon, \delta}, \bar{u}_{\varepsilon,\delta})dt\right)^p\nonumber\\
&\quad+\!\frac{C\varepsilon_1R^{4p}}{\varepsilon^{6p}}\mathbb{E}\left(\|\bar{v}_{{\varepsilon},\delta}\|_{L_t^\infty\mathrm{L}_x^2}^{2p}\!+\!\|(\bar{\varphi}_{\varepsilon, \delta}, \bar{\eta}_{\varepsilon, \delta})\|_{L^\infty_tV^1}^{2p}\right).
\end{align}
For the last term, the Burkholder-Davis-Gundy inequality and assumption ${\bf A}_2$ give
\begin{align}\label{3.31}
&\mathbb{E}\left(\sup_{t\in [0,T]}\left|\int_{0}^{t}(h(v^1_{{\varepsilon, \delta}},\nabla\varphi^1_{{\varepsilon, \delta}})-h(v^2_{{\varepsilon, \delta}},\nabla\varphi^2_{{\varepsilon, \delta}}), \bar{u}_{\varepsilon, \delta})d\mathcal{W}\right|\right)^p\nonumber\\
&\leq C\mathbb{E}\left(\int_{0}^{T}\|\bar{u}_{\varepsilon, \delta}\|_{\mathrm{L}_x^2}^2\|(\bar{v}_{\varepsilon, \delta}, \nabla\bar{\varphi}_{\varepsilon, \delta})\|^2_{L_x^2}dt\right)^{\frac{p}{2}}\nonumber\\
&\leq CT^{\frac{p}{2}-1}\mathbb{E}\int_{0}^{T}\| \bar{u}_{{\varepsilon}, \delta}\|_{\mathrm{L}_x^2}^{2p}dt+CT^{\frac{p}{2}}\|(\bar{v}_{{\varepsilon},\delta}, \bar{\varphi}_{{\varepsilon},\delta})\|_{X}^{2p}\nonumber\\
&\leq CT^{\frac{p}{2}}\mathbb{E}\| \bar{u}_{{\varepsilon}, \delta}\|_{L_t^\infty \mathrm{L}_x^2}^{2p}+CT^{\frac{p}{2}}\|(\bar{v}_{{\varepsilon},\delta}, \bar{\varphi}_{{\varepsilon,},\delta})\|_{X}^{2p}.
\end{align}

Also, we could control $\mathbb{E}\left(\int_{0}^{T}\|\partial_t\bar{u}_{\varepsilon, \delta}\|_{(H^1)'}^2dt\right)^p$ by a same way as above. Here we omit the tedious details.

Integrating of $t$, then taking supremum over $[0,T]$, power $p$ and expectation in \eqref{3.21}, combining all estimates \eqref{3.22}-\eqref{3.31}, we have
\begin{align}\label{3.32}
&\frac{1}{2^p}\mathbb{E}\|\bar{u}_{\varepsilon, \delta}\|^{2p}_{L_t^\infty\mathrm{L}_x^2}+\left(\frac{1}{2}-4\varepsilon_1\right)^p\mathbb{E}\left(\int_{0}^{T}a_0(\bar{u}_{\varepsilon, \delta}, \bar{u}_{\varepsilon, \delta})dt\right)^p+\mathbb{E}\|\partial_t\bar{u}_{\varepsilon, \delta}\|_{L_t^2(H^1)'}^{2p}\notag\\
& \leq T^p\left[CT^{-\frac{p}{2}}+\left(\frac{CR^2+R^4+C}{\varepsilon_1\varepsilon^6}\right)^p
+\left(\frac{C}{\varepsilon^6}\right)^p\right]
\mathbb{E}\left(\|\bar{v}_{\varepsilon, \delta}\|_{L_t^\infty \mathrm{L}_x^2}^{2p}+\|(\bar{\varphi}_{\varepsilon, \delta},\bar{\eta}_{\varepsilon, \delta})\|_{L_t^\infty V^1}^{2p}\right)\nonumber\\
&\quad+C(\varepsilon_1 R^2+\varepsilon_1)^p\mathbb{E}\left(\|(\bar{\varphi}_{\varepsilon, \delta},\bar{\eta}_{\varepsilon, \delta})\|_{L_t^2 V^2}^{2p}+\delta^{2p}\|\partial_t\bar{\varphi}_{\varepsilon, \delta}\|^{2p}_{L^2_tL^2_x}\right)+C\varepsilon_1^p\mathbb{E}\|(\bar{\varphi}_{\varepsilon, \delta},\bar{\eta}_{\varepsilon, \delta})\|_{L_t^\infty V^1}^{2p}\nonumber\\
&\quad+T^p\left(\frac{1}{2}+\frac{C}{\varepsilon_1}+\frac{C}{\varepsilon_1\varepsilon^4}
+\frac{C}{\varepsilon_1\tilde{\varepsilon}^{2\nu}}+\frac{C}{\varepsilon_1\varepsilon^6\tilde{\varepsilon}^{2\nu}}\right)^p\mathbb{E}\|\bar{u}_{\varepsilon, \delta}\|^{2p}_{L_t^\infty \mathrm{L}_x^2}
+CT^{\frac{p}{2}}\mathbb{E}\| \bar{u}_{\varepsilon, \delta}\|_{L_t^\infty \mathrm{L}_x^2}^{2p}\nonumber\\&\quad+\frac{\varepsilon_1R^{4p}}{\varepsilon^{4p}}\mathbb{E}\|\bar{v}_{\varepsilon, \delta}\|_{L^\infty_t\mathrm{L}_x^2}^{2p}.
\end{align}
We first choose $\varepsilon_1$ small and then take $T$ small enough such that
\begin{align*}
&\left[\frac{1}{2}+\frac{C}{\varepsilon_1}+\frac{C}{\varepsilon_1\varepsilon^4}
+\frac{C}{\varepsilon_1\tilde{\varepsilon}^{2\nu}}+\frac{C}{\varepsilon_1\varepsilon^6\tilde{\varepsilon}^{2\nu}}\right]^p\cdot(2T)^p+C(4T)^{\frac{p}{2}}\leq \frac{1}{4},\;~ C2^p(\varepsilon_1 R^2+\varepsilon_1)^p\leq \frac{1}{8},\nonumber\\
& \left[CT^{-\frac{p}{2}}+\left(\frac{CR^2+C}{\varepsilon_1\varepsilon^6}\right)^p
+\left(\frac{C}{\varepsilon^6}\right)^p\right]
\cdot (2T)^p+\frac{\varepsilon_1R^{4p}}{\varepsilon^{4p}}\leq \frac{1}{8}.
\end{align*}
From (\ref{3.32}), for the small time $T$, we conclude that
$$\|u_{\varepsilon, \delta}^1-u_{\varepsilon, \delta}^2\|_{X_1}^{2p}\leq \frac{1}{2}\|(\bar{v}_{\varepsilon, \delta}, \bar{\varphi}_{\varepsilon, \delta}, \bar{\eta}_{\varepsilon, \delta})\|_{X}^{2p}.$$

Next, taking inner product with $A_1^{-1}(\partial_t\bar{\phi}_{\varepsilon, \delta})$ and $\partial_t\bar{\psi}_{\varepsilon, \delta}$ in $\eqref{diff}_3, \eqref{diff}_4$, and also $\bar{\mu}_{\varepsilon, \delta}-\langle\bar{\mu}_{\varepsilon, \delta}\rangle$ in $\eqref{diff}_3$, and then taking expectation, we obtain
\begin{align}\label{3.33}
&\frac{1}{2}\mathbb{E}\|(\bar{\phi}_{\varepsilon, \delta}, \bar{\psi}_{\varepsilon, \delta})\|_{L_t^\infty V^1}^{2p}\!+\!\mathbb{E}\left(\delta\|\partial_t\bar{\phi}_{{\varepsilon, \delta}}\|_{L^2_tL^2(\Gamma)}^{2p}\!+\!\|(\bar{\phi}_{\varepsilon, \delta}, \bar{\psi}_{\varepsilon, \delta})\|^{2p}_{W^{1,2}_t((H^1)'\times L^2(\Gamma))}\!+\!\|\nabla \bar{\mu}_{\varepsilon, \delta}\|_{L^2_tL^2_x}^{2p}\right)\notag\\
&\leq\mathbb{E}\left(\int_{0}^{T}-\Psi^1_R(b_2(\bar{v}_{{\varepsilon, \delta}}, \mathcal{J}_{{\varepsilon}}\varphi^1_{{\varepsilon, \delta}})+b_2(v^2_{{\varepsilon, \delta}}, \mathcal{J}_{{\varepsilon}}\bar{\varphi}_{{\varepsilon, \delta}}),A_1^{-1}(\partial_t\bar{\phi}_{\varepsilon, \delta}))dt\right)^p\notag\\
&\quad+\mathbb{E}\left(\int_{0}^{T}-\Psi^1_R(b_2(\bar{v}_{{\varepsilon, \delta}}, \mathcal{J}_{{\varepsilon}}\varphi^1_{{\varepsilon, \delta}})+b_2(v^2_{{\varepsilon,\delta}}, \mathcal{J}_{{\varepsilon}}\bar{\varphi}_{{\varepsilon, \delta}}),\bar{\mu}_{\varepsilon, \delta}-\langle\bar{\mu}_{\varepsilon, \delta}\rangle)dt\right)^p\notag\\
&\quad+\mathbb{E}\left(\int_{0}^{T}-\Psi^1_R(b_\Gamma(\bar{v}_{{\varepsilon, \delta}},J_{{\varepsilon}}(\eta^1_{{\varepsilon, \delta}}))+b_\Gamma(v^2_{{\varepsilon,\delta}},J_{{\varepsilon}}(\bar{\eta}_{{\varepsilon, \delta}})), \partial_t\bar{\psi}_{\varepsilon, \delta})dt\right)^p\notag\\
&\quad+\mathbb{E}\left(\int_{0}^{T}(f_\varepsilon(\varphi^1_{\varepsilon, \delta})-f_\varepsilon(\varphi^2_{\varepsilon, \delta}), \partial_t \bar{\phi}_{\varepsilon, \delta})-(g_\varepsilon(\eta^1_{\varepsilon, \delta})-g_\varepsilon(\eta^2_{\varepsilon, \delta}), \partial_t\bar{\psi}_{\varepsilon, \delta})_\Gamma dt\right)^p\notag\\
&\quad+\mathbb{E}\left(\int_{0}^{T}-\tilde{\Psi}_R(b_2(v^2_{{\varepsilon, \delta}}, \mathcal{J}_{{\varepsilon}}\varphi^2_{{\varepsilon, \delta}}),A_1^{-1}(\partial_t\bar{\phi}_{\varepsilon, \delta}))dt\right)^p\notag\\
&\quad+\mathbb{E}\left(\int_{0}^{T}-\tilde{\Psi}_R(b_2(v^2_{{\varepsilon, \delta}}, \mathcal{J}_{{\varepsilon}}\varphi^2_{{\varepsilon, \delta}}),\bar{\mu}_{\varepsilon, \delta}-\langle\bar{\mu}_{\varepsilon, \delta}\rangle)dt\right)^p\notag\\
&\quad+\mathbb{E}\left(\int_{0}^{T}-\tilde{\Psi}_R(b_\Gamma(v^2_{{\varepsilon, \delta}}, \mathcal{J}_{{\varepsilon}}\eta^2_{{\varepsilon, \delta}}),\partial_t\bar{\psi}_{\varepsilon, \delta})dt\right)^p=:\mathbb{E}\left(\int_{0}^{T}\sum_{i=1}^7L_idt\right)^p.
\end{align}
According to the conservation property, we know $\langle \partial_t\psi^i_{\varepsilon, \delta}\rangle=0,\; \mathbb{P}\mbox{-a.s.}, i=1,2$, so $\langle \partial_t\bar{\psi}_{\varepsilon, \delta}\rangle=0,\; \mathbb{P} \mbox{-a.s.}$. Simple calculation yields
$$\langle\bar{\mu}_{\varepsilon, \delta}\rangle=\langle f_\varepsilon(\varphi^1_{\varepsilon, \delta})-f_\varepsilon(\varphi^2_{\varepsilon, \delta})\rangle-\frac{|\Gamma|}{|\mathcal{D}|}\langle\widetilde{\mathcal{K}}(\bar{\psi}_{\varepsilon, \delta})\rangle_\Gamma
+\frac{|\Gamma|}{|\mathcal{D}|}\langle\bar{\psi}_{\varepsilon, \delta}\rangle_\Gamma, ~\mathbb{P}\mbox{-a.s.}, $$
%also see \cite[5.40]{GMA},
where $\langle\Lambda\rangle_\Gamma= |\Gamma|^{-1}\int_\Gamma \Lambda dS$. From $\eqref{diff}_4$, we get
\begin{align}\label{mu*}
\begin{split}
\langle\bar{\mu}_{\varepsilon, \delta}\rangle^2\leq & ~C\Big(\|\partial_t\bar{\psi}_{\varepsilon, \delta}\|^2_{L^2(\Gamma)}+\|f_\varepsilon(\varphi^1_{\varepsilon, \delta})-f_\varepsilon(\varphi^2_{\varepsilon, \delta})\|^2_{L_x^2}
+\|g_\varepsilon(\eta^1_{\varepsilon, \delta})-g_\varepsilon(\eta^2_{\varepsilon, \delta})\|^2_{L^2(\Gamma)}\\
&+\|b_\Gamma(\bar{v}_{{\varepsilon, \delta}},J_{{\varepsilon}}(\eta^1_{{\varepsilon, \delta}}))+b_\Gamma(v^2_{{\varepsilon, \delta}},J_{{\varepsilon}}(\bar{\eta}_{{\varepsilon, \delta}}))\|^2_{L^2(\Gamma)}
+\|\bar{\psi}_{{\varepsilon, \delta}}\|^2_{L^2(\Gamma)}\Big), ~\mathbb{P}\mbox{-a.s.},
\end{split}
\end{align}
where $C$ depends on $|\Gamma|, |\mathcal{D}|$ but independence of $\delta$. Combining Lemma \ref{4.4*}, (\ref{3.33}), (\ref{mu*}), and the Poincar\'{e} inequality, we arrive at
\begin{align}\label{3.35}
&\frac{1}{2}\mathbb{E}\|(\bar{\phi}_{\varepsilon, \delta}, \bar{\psi}_{\varepsilon, \delta})\|_{L_t^\infty V^1}^{2p}+\mathbb{E}\left(\delta\|\partial_t\bar{\psi}_{\varepsilon, \delta}\|_{L^2_tL^2(\Gamma)}^{2p}+\|(\bar{\phi}_{{\varepsilon, \delta}}, \bar{\psi}_{{\varepsilon, \delta}})\|_{L_t^2V^2}^{2p}+\|\nabla \bar{\mu}_{\varepsilon, \delta}\|_{L^2_tL^2_x}^{2p}\right)\nonumber\\
&\quad+\mathbb{E}\|(\bar{\phi}_{\varepsilon, \delta}, \bar{\psi}_{\varepsilon, \delta})\|^{2p}_{W^{1,2}_t((H^1)'\times L^2(\Gamma))}\nonumber\\
&\leq C\mathbb{E}\left(\int_{0}^{T}\sum_{i=1}^7L_idt\right)^p+C\mathbb{E}\bigg(\int_{0}^{T}\|f_\varepsilon(\varphi^1_{\varepsilon, \delta})-f_\varepsilon(\varphi^2_{\varepsilon, \delta})\|^2_{L_x^2}
+\|g_\varepsilon(\eta^1_{\varepsilon, \delta})-g_\varepsilon(\eta^2_{\varepsilon, \delta})\|^2_{L^2(\Gamma)}\nonumber\\
&\quad+\|b_\Gamma(\bar{v}_{{\varepsilon, \delta}},J_{{\varepsilon}}(\eta^1_{{\varepsilon, \delta}}))+b_\Gamma(v^2_{{\varepsilon, \delta}},J_{{\varepsilon}}(\bar{\eta}_{{\varepsilon, \delta}}))\|^2_{L^2(\Gamma)}dt\bigg)^p,
\end{align}
where $C$ depends on $|\Gamma|, |\mathcal{D}|$, but independence of $\delta$.

We are going to estimate all terms on the right-hand side of \eqref{3.35}. The estimate of $L_i, i=1,\cdots, 4$ are given in \cite{GMA}. For reader's convenience, we still formulate them below. By H\"{o}lder's inequality,  Lemma \ref{lem4.2} and Lemma \ref{lem4.3}, we get
\begin{align}\label{3.36}
|L_1|&\leq \|\partial_t\bar{\phi}_{\varepsilon, \delta}\|_{(H^1)'}\|b_2(\bar{v}_{{\varepsilon, \delta}}, \mathcal{J}_{{\varepsilon}}\varphi^1_{{\varepsilon \delta}})+b_2(v^2_{{\varepsilon,\delta}}, \mathcal{J}_{{\varepsilon}}\bar{\varphi}_{{\varepsilon, \delta}})\|_{(H^1)'}\nonumber\\
&\leq {\varepsilon}_2\|\partial_t\bar{\phi}_{\varepsilon, \delta}\|_{(H^1)'}^2+\frac{C}{{\varepsilon}_2}(\|\bar{v}_{\varepsilon, \delta}\|_{\mathrm{L}_x^2}^2\|\nabla\mathcal{J}_{{\varepsilon}}\varphi^1_{{\varepsilon, \delta}}\|_{H^2}^2
+\|v^2_{\varepsilon, \delta}\|_{\mathrm{L}_x^2}^2\|\nabla\mathcal{J}_{{\varepsilon}}\bar{\varphi}_{{\varepsilon, \delta}}\|_{H^2}^2)\\
&\leq {\varepsilon}_2\|\partial_t\bar{\phi}_{\varepsilon, \delta}\|_{(H^1)'}^2+\frac{C}{{\varepsilon}_2{\varepsilon}^4}(\|\bar{v}_{\varepsilon, \delta}\|_{\mathrm{L}_x^2}^2\|\varphi^1_{{\varepsilon, \delta}}\|_{H^1}^2
+\|v^2_{\varepsilon}\|_{\mathrm{L}_x^2}^2\|\bar{\varphi}_{{\varepsilon, \delta}}\|_{H^1}^2),\nonumber
\end{align}
and
\begin{align}
|L_2|&\leq {\varepsilon}_2\|\nabla \bar{\mu}_{\varepsilon, \delta}\|^2_{L_x^2}+\frac{C}{{\varepsilon}_2 {\varepsilon}^4}(\|v_{\varepsilon, \delta}^2\|^2_{\mathrm{L}_x^2}\|\bar{\varphi}_{\varepsilon, \delta}\|^2_{H^1}+\|\bar{v}_{\varepsilon, \delta}\|^2_{\mathrm{L}_x^2}\|\varphi^1_{\varepsilon, \delta}\|^2_{H^1}),
\end{align}
and
\begin{align}\label{3.38}
\!\!\!\!|L_3|&\leq \|\partial_t\bar{\psi}_{\varepsilon, \delta}\|_{L^2(\Gamma)}(\|\bar{v}_{{\varepsilon, \delta}}\|_{\mathrm{L}^2(\Gamma)}\|J_{{\varepsilon}}(\eta^1_{{\varepsilon, \delta}})\|_{H^3(\Gamma)}
+\|v^2_{{\varepsilon, \delta}}\|_{\mathrm{L}^2(\Gamma)}\|J_{{\varepsilon}}(\bar{\eta}_{{\varepsilon, \delta}})\|_{H^3(\Gamma)})\nonumber\\
&\leq \frac{C}{{\varepsilon}^3}\|\partial_t\bar{\psi}_{\varepsilon, \delta}\|_{L^2(\Gamma)}(\bar{{\varepsilon}}\|D(\bar{v}_{{\varepsilon, \delta}})\|_{\mathrm{L}_x^2}
+\bar{{\varepsilon}}^{-\nu}\|\bar{v}_{{\varepsilon, \delta}}\|_{\mathrm{L}_x^2})\|\eta^1_{{\varepsilon, \delta}}\|_{L^2(\Gamma)}\nonumber\\
&\quad+\frac{C}{{\varepsilon}^3}\|\partial_t\bar{\psi}_{\varepsilon, \delta}\|_{L^2(\Gamma)}(\bar{{\varepsilon}}\|D(v^2_{{\varepsilon, \delta}})\|_{\mathrm{L}_x^2}
+\bar{{\varepsilon}}^{-\nu}\|v^2_{{\varepsilon, \delta}}\|_{\mathrm{L}_x^2})\|\bar{\eta}_{{\varepsilon, \delta}}\|_{L^2(\Gamma)}\nonumber\\
&\leq {\varepsilon}_2\|\partial_t\bar{\psi}_{\varepsilon, \delta}\|_{L^2(\Gamma)}^2+\frac{C\bar{{\varepsilon}}}{{\varepsilon}_2{\varepsilon}^6}\|D(\bar{v}_{{\varepsilon, \delta}})\|^2_{\mathrm{L}_x^2}\|\eta^1_{{\varepsilon, \delta}}\|^2_{L^2(\Gamma)}+
\frac{C}{{\varepsilon}_2{\varepsilon}^6\bar{{\varepsilon}}^{\nu}}\|\bar{v}_{{\varepsilon, \delta}}\|^2_{\mathrm{L}_x^2}\|\eta^1_{{\varepsilon, \delta}}\|^2_{L^2(\Gamma)}\nonumber\\
&\quad+{\varepsilon}_2\|\partial_t\bar{\psi}_{\varepsilon, \delta}\|_{L^2(\Gamma)}^2+\frac{C\bar{{\varepsilon}}}{{\varepsilon}_2{\varepsilon}^6}
\|D(v^2_{{\varepsilon, \delta}})\|^2_{\mathrm{L}_x^2}\|\bar{\eta}_{{\varepsilon, \delta}}\|^2_{L^2(\Gamma)}+
\frac{C}{{\varepsilon}_2{\varepsilon}^6\bar{{\varepsilon}}^{\nu}}\|v^2_{{\varepsilon, \delta}}\|^2_{\mathrm{L}_x^2}\|\bar{\eta}_{{\varepsilon, \delta}}\|^2_{L^2(\Gamma)}.
\end{align}
In addition, by Lemma \ref{lem4.3}, we obtain
\begin{align}\label{3.39*}
&\|b_\Gamma(\bar{v}_{{\varepsilon, \delta}},J_{{\varepsilon}}(\eta^1_{{\varepsilon, \delta}}))+b_\Gamma(v^2_{{\varepsilon, \delta}},J_{{\varepsilon}}(\bar{\eta}_{{\varepsilon, \delta}}))\|^2_{L^2(\Gamma)}\nonumber\\
&\leq\frac{C}{{\varepsilon}^3}(\|\bar{v}_{\varepsilon, \delta}\|_{\mathrm{L}^2(\Gamma)}^2\|\eta_{\varepsilon, \delta}^1\|_{L^2(\Gamma)}^2+\| v_{\varepsilon, \delta}^2\|_{\mathrm{L}^2(\Gamma)}^2\|\bar{\eta}_{\varepsilon, \delta}\|_{L^2(\Gamma)}^2)\nonumber\\
&\leq\frac{C}{{\varepsilon}^3}(\bar{{\varepsilon}}\|D(\bar{v}_{\varepsilon, \delta})\|_{\mathrm{L}_x^2}^2
+\bar{{\varepsilon}}^{-\nu}\|\bar{v}_{\varepsilon, \delta}\|_{\mathrm{L}^2_x}^2)\|\eta_{\varepsilon, \delta}^1\|_{L^2(\Gamma)}^2\nonumber\\
&\quad+\frac{C}{{\varepsilon}^3}(\bar{{\varepsilon}}\|D(v^2_{\varepsilon, \delta})\|_{\mathrm{L}_x^2}^2+\bar{{\varepsilon}}^{-\nu}\| v_{\varepsilon, \delta}^2\|_{\mathrm{L}^2_x}^2)\|\bar{\eta}_{\varepsilon, \delta}\|_{L^2(\Gamma)}^2.
\end{align}
For term $L_4$, using \eqref{3.7*}, \eqref{3.8**}, we have
\begin{align}
|L_4|&\leq  \|\partial_t \bar{\phi}_{\varepsilon, \delta}\|_{L_x^2}\|f_\varepsilon(\varphi^1_{\varepsilon, \delta})-f_\varepsilon(\varphi^2_{\varepsilon, \delta})\|_{L_x^2}+\|\partial_t \bar{\psi}_{\varepsilon, \delta}\|_{L^2(\Gamma)}\|g_\varepsilon(\eta^1_{\varepsilon, \delta})-g_\varepsilon(\eta^2_{\varepsilon, \delta})\|_{L^2(\Gamma)}\nonumber\\
&\leq \frac{\varepsilon_2}{2}\|\partial_t \bar{\phi}_{\varepsilon, \delta}\|^2_{L_x^2}+\frac{C}{\varepsilon_2}\|f_\varepsilon(\varphi^1_{\varepsilon, \delta})\!-\!f_\varepsilon(\varphi^2_{\varepsilon, \delta})\|_{L_x^2}^2\nonumber\\&\quad+\frac{1}{2}\|\partial_t\bar{ \psi}_{\varepsilon, \delta}\|^2_{L^2(\Gamma)}+2\|g_\varepsilon(\eta^1_{\varepsilon, \delta})\!-\!g_\varepsilon(\eta^2_{\varepsilon, \delta})\|_{L^2(\Gamma)}^2\nonumber\\
&\leq \frac{\varepsilon_2}{2}\|\partial_t \bar{\phi}_{\varepsilon, \delta}\|^2_{L_x^2}+\frac{C}{\varepsilon_2}\|\bar{\varphi}_{{\varepsilon},\delta}\|^2_{ L_x^2}
+\frac{1}{2}\|\partial_t \bar{\psi}_{\varepsilon, \delta}\|^2_{L^2(\Gamma)}+C\|\bar{\eta}_{{\varepsilon},\delta}\|^2_{ L^2(\Gamma)}.
\end{align}
Next, we give the crucial estimate of terms $L_i, i=5,6,7$ arising from the cut-off function. For term $L_5$, by H\"{o}lder's inequality,  Lemma \ref{lem4.2} and Lemma \ref{lem4.3}, we get
\begin{align}
&\mathbb{E}\left(\int_{0}^{T}L_5dt\right)^p\nonumber\\ &\leq  C\mathbb{E}\Bigg(\int_{0}^{T}\left|\overline{\Psi}_{R}(\|v_{\varepsilon, \delta}^1\|_{\mathrm{L}_x^2})-\overline{\Psi}_{R}(\|v_{\varepsilon, \delta}^2\|_{\mathrm{L}_x^2})\right|\overline{\Psi}_{R}(\|(\varphi_{\varepsilon, \delta}^1, \eta_{\varepsilon, \delta}^1)\|_{ V^1})\nonumber\\
&\qquad \times(b_2(v^2_{{\varepsilon, \delta}}, \mathcal{J}_{{\varepsilon}}\varphi^2_{{\varepsilon, \delta}}),A_1^{-1}(\partial_t\bar{\phi}_{\varepsilon, \delta}))\nonumber\\
&\quad+\overline{\Psi}_{R}(\|v_{\varepsilon, \delta}^2\|_{\mathrm{L}_x^2})\left|\overline{\Psi}_{R}(\|(\varphi_{\varepsilon, \delta}^1, \eta_{\varepsilon, \delta}^1)\|_{ V^1})-\overline{\Psi}_{R}(\|(\varphi_{\varepsilon, \delta}^2, \eta_{\varepsilon, \delta}^2)\|_{ V^1})\right|\nonumber\\
&\qquad\times(b_2(v^2_{{\varepsilon, \delta}}, \mathcal{J}_{{\varepsilon}}\varphi^2_{{\varepsilon, \delta}}),A_1^{-1}(\partial_t\bar{\phi}_{\varepsilon, \delta}))dt\Bigg)^p\nonumber\\
&\leq C\mathbb{E}\Bigg(\int_{0}^{\tau_R^2}(\|\bar{v}_{\varepsilon, \delta}\|_{L_t^\infty\mathrm{L}_x^2}+\|(\bar{\varphi}_{\varepsilon, \delta}, \bar{\eta}_{\varepsilon, \delta})\|_{L_t^\infty V^1})\|\partial_t\bar{\phi}_{\varepsilon, \delta}\|_{(H^1)'}\|v^2_{\varepsilon, \delta}\|_{\mathrm{L}_x^2}\|\nabla\mathcal{J}_{{\varepsilon}}\varphi^2_{{\varepsilon, \delta}}\|_{H^2}dt\Bigg)^p\nonumber\\
&\leq C\mathbb{E}\Bigg(\int_{0}^{\tau_R^2}\frac{1}{\varepsilon_1}(\|\bar{v}_{\varepsilon, \delta}\|^2_{L_t^\infty\mathrm{L}_x^2}+\|(\bar{\varphi}_{\varepsilon, \delta}, \bar{\eta}_{\varepsilon, \delta})\|^2_{L_t^\infty V^1})
+\frac{\varepsilon_1}{\varepsilon^4}\|\partial_t\bar{\phi}_{\varepsilon, \delta}\|_{(H^1)'}^2\|v^2_{\varepsilon, \delta}\|_{\mathrm{L}_x^2}^2\|\varphi^2_{{\varepsilon, \delta}}\|_{H^1}^2dt\Bigg)^p\nonumber\\
&\leq \frac{C}{\varepsilon_1^p}(\tau_R^2)^p\mathbb{E}\left(\|\bar{v}_{\varepsilon, \delta}\|_{L_t^\infty\mathrm{L}_x^2}^{2p}+\|(\bar{\varphi}_{\varepsilon, \delta}, \bar{\eta}_{\varepsilon, \delta})\|_{L_t^\infty V^1}^{2p}\right)+\frac{C\varepsilon^p_1R^{4p}}{\varepsilon^{4p}}\mathbb{E}\|\partial_t\bar{\phi}_{\varepsilon, \delta}\|_{L_t^2(H^1)'}^{2p}\nonumber\\
&\leq \frac{C}{\varepsilon^p_1}T^p\mathbb{E}\left(\|\bar{v}_{\varepsilon, \delta}\|_{L_t^\infty\mathrm{L}_x^2}^{2p}+\|(\bar{\varphi}_{\varepsilon, \delta}, \bar{\eta}_{\varepsilon, \delta})\|_{L_t^\infty V^1}^{2p}\right)+\frac{C\varepsilon^p_1R^{4p}}{\varepsilon^{4p}}\mathbb{E}\|\partial_t\bar{\phi}_{\varepsilon, \delta}\|_{L_t^2(H^1)'}^{2p}.
\end{align}
Similarly,
\begin{align}
&\mathbb{E}\left(\int_{0}^{T}L_6dt\right)^p\nonumber\\ &\leq C\mathbb{E}\Bigg(\int_{0}^{\tau_R^2}(\|\bar{v}_{\varepsilon, \delta}\|_{L_t^\infty\mathrm{L}_x^2}+\|(\bar{\varphi}_{\varepsilon, \delta}, \bar{\eta}_{\varepsilon, \delta})\|_{L_t^\infty V^1})\|\nabla\bar{\mu}_{\varepsilon, \delta}\|_{L_x^2}\|v^2_{\varepsilon, \delta}\|_{\mathrm{L}_x^2}\|\nabla\mathcal{J}_{{\varepsilon}}\varphi^2_{{\varepsilon, \delta}}\|_{H^2}dt\Bigg)^p\nonumber\\
&\leq C\mathbb{E}\Bigg(\int_{0}^{\tau_R^2}(\|\bar{v}_{\varepsilon, \delta}\|^2_{L_t^\infty\mathrm{L}_x^2}+\|(\bar{\varphi}_{\varepsilon, \delta}, \bar{\eta}_{\varepsilon, \delta})\|^2_{L_t^\infty V^1})
+\frac{1}{\varepsilon^4}\|\nabla\bar{\mu}_{\varepsilon, \delta}\|_{L_x^2}^2\|v^2_{\varepsilon, \delta}\|_{\mathrm{L}_x^2}^2\|\varphi^2_{{\varepsilon, \delta}}\|_{H^1}^2dt\Bigg)^p\nonumber\\
&\leq \frac{C}{\varepsilon^p_1}(\tau_R^2)^p\mathbb{E}\left(\|\bar{v}_{\varepsilon, \delta}\|_{L_t^\infty\mathrm{L}_x^2}^{2p}+\|(\bar{\varphi}_{\varepsilon, \delta}, \bar{\eta}_{\varepsilon, \delta})\|_{L_t^\infty V^1}^{2p}\right)+\frac{C\varepsilon^p_1R^{4p}}{\varepsilon^{4p}}\mathbb{E}\|\nabla\bar{\mu}_{\varepsilon, \delta}\|_{L^2_tL_x^2}^{2p}\nonumber\\
&\leq \frac{C}{\varepsilon^p_1}T^p\mathbb{E}\left(\|\bar{v}_{\varepsilon, \delta}\|_{L_t^\infty\mathrm{L}_x^2}^{2p}+\|(\bar{\varphi}_{\varepsilon, \delta}, \bar{\eta}_{\varepsilon, \delta})\|_{L_t^\infty V^1}^{2p}\right)+\frac{C\varepsilon^p_1R^{4p}}{\varepsilon^{4p}}\mathbb{E}\|\nabla\bar{\mu}_{\varepsilon, \delta}\|_{L^2_tL_x^2}^{2p},
\end{align}
and
\begin{align}\label{3.50*}
&\mathbb{E}\left(\int_{0}^{T}L_7dt\right)^p\nonumber\\
&\leq C\mathbb{E}\Bigg(\int_{0}^{\tau_R^2}(\|\bar{v}_{\varepsilon, \delta}\|_{L_t^\infty\mathrm{L}_x^2}+\|(\bar{\varphi}_{\varepsilon, \delta}, \bar{\eta}_{\varepsilon, \delta})\|_{L_t^\infty V^1})\|\partial_t\bar{\psi}_{\varepsilon, \delta}\|_{L^2(\Gamma)}
\|v^2_{{\varepsilon, \delta}}\|_{\mathrm{L}^2(\Gamma)}\|J_{{\varepsilon}}\eta^2_{{\varepsilon, \delta}}\|_{H^3(\Gamma)}dt\Bigg)^p\nonumber\\
&\leq \mathbb{E}\Bigg(\frac{C}{\varepsilon^2}\int_{0}^{\tau_R^2}(\|\bar{v}_{\varepsilon, \delta}\|_{L_t^\infty\mathrm{L}_x^2}\!+\!\|(\bar{\varphi}_{\varepsilon, \delta}, \bar{\eta}_{\varepsilon, \delta})\|_{L_t^\infty V^1})\nonumber\\
&\qquad\times\|\partial_t\bar{\psi}_{\varepsilon, \delta}\|_{L^2(\Gamma)}(\bar{{\varepsilon}}\|D(v^2_{{\varepsilon, \delta}})\|_{\mathrm{L}_x^2}
\!+\!\bar{{\varepsilon}}^{-\nu}\|v^2_{{\varepsilon, \delta}}\|_{\mathrm{L}_x^2})\|\eta^2_{{\varepsilon, \delta}}\|_{H^1(\Gamma)}dt\Bigg)^p\nonumber\\
&\leq \mathbb{E}\Bigg(\int_{0}^{\tau_R^2}\frac{\bar{{\varepsilon}}}{\varepsilon_1}(\|\bar{v}_{\varepsilon, \delta}\|_{L_t^\infty\mathrm{L}_x^2}^2+\|(\bar{\varphi}_{\varepsilon, \delta}, \bar{\eta}_{\varepsilon, \delta})\|_{L_t^\infty V^1}^2)
\|D(v^2_{{\varepsilon, \delta}})\|_{\mathrm{L}_x^2}^2\nonumber\\&\qquad+\frac{\varepsilon_1}{\varepsilon^4}\|\partial_t\bar{\psi}_{\varepsilon, \delta}\|_{L^2(\Gamma)}^2\|\eta^2_{{\varepsilon, \delta}}\|_{H^1(\Gamma)}^2dt\Bigg)^p\nonumber\\
&\quad+\mathbb{E}\Bigg(\int_{0}^{\tau_R^2}\frac{1}{\bar{\varepsilon}^\nu\varepsilon_1}(\|\bar{v}_{\varepsilon, \delta}\|_{L_t^\infty\mathrm{L}_x^2}^2+\|(\bar{\varphi}_{\varepsilon, \delta}, \bar{\eta}_{\varepsilon, \delta})\|_{L_t^\infty V^1}^2)
\|v^2_{{\varepsilon, \delta}}\|_{\mathrm{L}_x^2}^2\nonumber\\&\qquad+\frac{\varepsilon_1}{\varepsilon^4}\|\partial_t\bar{\psi}_{\varepsilon, \delta}\|_{L^2(\Gamma)}^2\|\eta^2_{{\varepsilon, \delta}}\|_{H^1(\Gamma)}^2dt\Bigg)^p\nonumber\\
&\leq \left(\frac{C}{\bar{\varepsilon}^{p\nu}\varepsilon^p_1}(\tau_R^2)^p+\frac{R^{2p}\bar{{\varepsilon}}^p}{\varepsilon^p_1}\right)\mathbb{E}\left(\|\bar{v}_{\varepsilon, \delta}\|_{L_t^\infty\mathrm{L}_x^2}^{2p}+\|(\bar{\varphi}_{\varepsilon, \delta}, \bar{\eta}_{\varepsilon, \delta})\|_{L_t^\infty V^1}^{2p}\right)\nonumber\\&\quad+\frac{C\varepsilon^p_1R^{2p}}{\varepsilon^{4p}}\mathbb{E}\|\partial_t\bar{\psi}_{\varepsilon, \delta}\|_{L_t^2L^2(\Gamma)}^{2p}\nonumber\\
&\leq \left(\frac{C}{\bar{\varepsilon}^{p\nu}\varepsilon^p_1}T^p+\frac{R^{2p}\bar{\varepsilon}^p}{\varepsilon^p_1}\right)\mathbb{E}\left(\|\bar{v}_{\varepsilon, \delta}\|_{L_t^\infty\mathrm{L}_x^2}^{2p}+\|(\bar{\varphi}_{\varepsilon, \delta}, \bar{\eta}_{\varepsilon, \delta})\|_{L_t^\infty V^1}^{2p}\right)\nonumber\\
&\quad+\frac{C\varepsilon^p_1R^{2p}}{\varepsilon^{4p}}\mathbb{E}\|\partial_t\bar{\psi}_{\varepsilon, \delta}\|_{L_t^2L^2(\Gamma)}^{2p}.
\end{align}

Choosing first $\varepsilon_1, \varepsilon_2$ small enough, such that
\begin{align}
\frac{C\varepsilon^p_1R^{4p}}{\varepsilon^{4p}}+\varepsilon_2^p\leq \frac{1}{4},
\end{align}
then, for fixed $\varepsilon_1, \varepsilon_2$, choosing $\bar{\varepsilon}$ small enough, such that
\begin{align}
\frac{R^{2p}\bar{\varepsilon}^p}{\varepsilon^p_1}+\frac{C\bar{{\varepsilon}}^pR^{2p}}{{\varepsilon}^p_2{\varepsilon}^{6p}}
+\frac{C\bar{\varepsilon}^pR^{2p}}{{\varepsilon}^{3p}}\leq \frac{1}{4},
\end{align}
finally, for the small $\varepsilon_1, \varepsilon_2$ and $\bar{\varepsilon}$, we choose time $T$ small enough such that
\begin{align}\label{3.53}
\left(C+\frac{C(1+R^2)}{\delta}+CR^2+\frac{CR^2}{\varepsilon_2\varepsilon^4}+\frac{CR^2}{\varepsilon_2\varepsilon^6\bar{\varepsilon}^{\nu}}
+\frac{CR^2}{\varepsilon^3\bar{\varepsilon}^{\nu}}+\frac{C}{\bar{\varepsilon}^{\nu}\varepsilon_1}+\frac{C}{\varepsilon_1}\right)^pT^p\leq \frac{1}{4}.
\end{align}
Collecting all estimates \eqref{3.35}-\eqref{3.53}, we conclude that for this small time $T$,
$$\|(\bar{\phi}_{\varepsilon, \delta}, \bar{\psi}_{\varepsilon, \delta})\|_{X_2}^{2p}\leq \frac{1}{2}\|(\bar{v}_{\varepsilon, \delta}, \bar{\varphi}_{\varepsilon, \delta}, \bar{\eta}_{\varepsilon, \delta})\|_{X}^{2p}.$$
 We obtain mapping $\mathcal{M}$ is contraction in $[0,T]$.

Hence, we have mapping $\mathcal{M}$ is contraction from $X$ into itself for the fixed constant $R$ and small time $T$. Thus, for every $\varepsilon,\delta, R$, there exists unique approximate solutions sequence $(u_{\varepsilon, \delta}^R, \phi_{\varepsilon, \delta}^R, \psi_{\varepsilon, \delta}^R)$ to system \eqref{equ3.1*} in time interval $[0,T]$.  The following estimates uniformly in parameters $\varepsilon,\delta, R$ allow us to extend the local existence time to $[0,T]$ for any $T>0$. Thus, we obtain the global existence and uniqueness of solutions $(u_{\varepsilon, \delta}^R, \phi_{\varepsilon, \delta}^R, \psi_{\varepsilon, \delta}^R)$ for any fixed $\varepsilon,\delta, R$.

{\bf Part 2.  The a priori estimates and pass the limit as $R\rightarrow\infty$.}

In this part, we establish the a priori estimates for the approximate solutions uniformly in $\varepsilon,\delta, R$, and then prove the global existence of unique approximate solutions $(u_{\varepsilon, \delta}, \phi_{\varepsilon, \delta}, \psi_{\varepsilon, \delta})$ of system \eqref{equ3.1} for any fixed $\varepsilon,\delta$ by passing the limit as $R\rightarrow\infty$.

We first show uniform estimates. Test $(\ref{equ3.1*})_2$ and $(\ref{equ3.1*})_3$ by $\xi_2=\mu^R_{\varepsilon, \delta}\in H^1$ and $\xi_3=\mathcal{K}(\psi^R_{\varepsilon, \delta})\in L^2(\Gamma)$, and also test (\ref{3.3}) and (\ref{3.4}) by $\partial_t\phi^R_{\varepsilon, \delta}\in L_x^2$ and $\partial_t\psi^R_{\varepsilon, \delta}\in L^2(\Gamma)$. Applying It\^{o}'s formula to the function $\frac{1}{2}\|u^R_{\varepsilon, \delta}\|_{\mathrm{L}_x^2}^2$, by the cancellation properties $$B_1(\mu,\phi, u)=B_2(u,\phi,\mu)$$ and $$B_\Gamma(u,\psi, \mathcal{K}(\psi))=(\mathcal{K}(\psi)\cdot \nabla_{\tau}\psi,u_\tau)_\Gamma,$$ we conclude
\begin{align}\label{3.52}
&\frac{1}{2}d\left(\|u^R_{\varepsilon,\delta}\|_{\mathrm{L}_x^2}^2+\|\nabla \phi^R_{\varepsilon, \delta}\|_{L_x^2}^2+\|\psi^R_{\varepsilon,\delta}\|_{L^2(\Gamma)}^2+\|\nabla_{\tau}\psi^R_{\varepsilon, \delta}\|_{L^2(\Gamma)}^2+2\int_{\mathcal{D}}F(\phi^R_{\varepsilon, \delta})dx\right)\nonumber\\
&\quad+d\int_{\Gamma}G(\psi^R_{\varepsilon, \delta})dS+(\|(\nabla u^R_{\varepsilon,\delta}, \nabla \mu^R_{\varepsilon,\delta})\|_{L_x^2}^2+\|\mathcal{K}(\psi^R_{\varepsilon, \delta})\|_{L^2(\Gamma)}^2+\delta\|\partial_t\phi^R_{\varepsilon,\delta}\|_{L_x^2}^2)dt\\
&=(h(u^R_{\varepsilon, \delta},\nabla\phi^R_{\varepsilon, \delta}), u^R_{\varepsilon,\delta})d\mathcal{W}+\frac{1}{2}\|h(u^R_{\varepsilon, \delta},\nabla\phi^R_{\varepsilon, \delta})\|^2_{L_2(\mathcal{H};\mathrm{L}_x^2)}dt.\nonumber
\end{align}

\begin{remark}\label{rem3.3} Note that, to achieve the a priori estimates we have to choose the functions $\xi_2=\mu$ and $\xi_3=\mathcal{K}(\psi)$ to test equations (\ref{3.3}) and (\ref{3.4}). In process of Galerkin approximation, we can not implement such choice since this procedure requires ${\rm Tr}_{\mathcal{D}}\mu=\mathcal{K}(\psi)$ at the level of finite-dimension. For more explanations, see \cite{GMA}.
\end{remark}

Define by the stopping time
$$\rho_N=\inf\left\{t\geq 0;\sup_{t\geq 0}\|u^R_{\varepsilon,\delta}\|^2_{\mathrm{L}_x^2}\geq N\right\},$$
and choose $\rho_N=T$, if the set is empty. Note that $\rho_N$ is non-decreasing with $\lim_{N\rightarrow\infty}\rho_N=T$.
Assumption ${\bf A}_1$ and the Burkholder-Davis-Gundy inequality give
\begin{align}\label{3.37}
&\mathbb{E}\left(\sup_{s\in [0,t\wedge \rho_N]}\left|\int_{0}^{s}(h(u^R_{\varepsilon, \delta},\nabla\phi^R_{\varepsilon, \delta}), u^R_{\varepsilon,\delta})d\mathcal{W}\right|\right)^p\nonumber\\
&\leq C\mathbb{E}\left(\int_{0}^{t\wedge \rho_N}\|h(u^R_{\varepsilon, \delta},\nabla\phi^R_{\varepsilon, \delta})\|^2_{L_2(\mathcal{H};L_x^2)}\| u^R_{\varepsilon,\delta}\|_{\mathrm{L}_x^2}^2ds\right)^\frac{p}{2}\nonumber\\
&\leq \frac{1}{4}\mathbb{E}\left(\sup_{s\in [0,t\wedge \rho_N]}\| u^R_{\varepsilon,\delta}\|_{\mathrm{L}_x^2}^{2p}\right)+C\mathbb{E}\int_{0}^{t\wedge \rho_N}1+\|(u^R_{\varepsilon,\delta},\nabla\phi^R_{\varepsilon,\delta})\|_{L_x^2}^{2p}ds.
\end{align}
Again, assumption ${\bf A}_1$ gives
\begin{align}\label{3.54}
\frac{1}{2}\|h(u^R_{\varepsilon, \delta},\nabla\phi^R_{\varepsilon, \delta})\|^2_{L_2(\mathcal{H};L_x^2)}\leq C(1+\|(u^R_{\varepsilon, \delta},\nabla\phi^R_{\varepsilon, \delta})\|_{L_x^2}^2).
\end{align}
By assumptions ${\bf A}_3$ and ${\bf A}_4$, the initial data $(\phi_0,\psi_0)\in L_\omega^pV^1$, we have $F(\phi_0)\in L_\omega^pL_x^1$, $G(\psi_0)\in L_\omega^pL^1(\Gamma)$.
Integrating of $t$ in \eqref{3.52}, taking supremum over $[0, T\wedge \rho_N]$ and then power $p$, expectation, \eqref{3.37}, \eqref{3.54} and Gronwall's inequality yield
\begin{align*}
&\mathbb{E}\left(\|u^R_{\varepsilon,\delta}\|_{L_{t\wedge \rho_N}^\infty \mathrm{L}_x^2}^2+\|\nabla \phi^R_{\varepsilon, \delta}\|_{L_{t\wedge \rho_N}^\infty L_x^2}^2+\|\psi^R_{\varepsilon,\delta}\|_{L_{t\wedge \rho_N}^\infty L^2(\Gamma)}^2+\|\nabla_{\tau}\psi^R_{\varepsilon, \delta}\|_{L_{t\wedge \rho_N}^\infty L^2(\Gamma)}^2\right)^p\nonumber\\
&+\mathbb{E}\left(\sup_{t\in [0,T\wedge \rho_N]}\int_{\Gamma}G(\psi^R_{\varepsilon, \delta})dS+\sup_{t\in [0,T\wedge \rho_N]}\int_{\mathcal{D}}F(\phi^R_{\varepsilon, \delta})dx\right)^p\nonumber\\
&+\mathbb{E}\left(\|(\nabla u^R_{\varepsilon,\delta}, \nabla \mu^R_{\varepsilon,\delta})\|_{L_{t\wedge \rho_N}^2L_x^2}^2+\|\mathcal{K}(\psi^R_{\varepsilon, \delta})\|_{L_{t\wedge \rho_N}^2L^2(\Gamma)}^2+\delta\|\partial_t\phi^R_{\varepsilon,\delta}\|_{L_{t\wedge \rho_N}^2L_x^2}^2\right)^p\leq C,
\end{align*}
where $C=C(p, T, \mathcal{D}, \Gamma)$ is a constant independence of $\varepsilon, \delta, R, N$. Finally, the monotone convergence theorem yields as $N\rightarrow\infty$
\begin{align}\label{main1*}
&\mathbb{E}\left(\|u^R_{\varepsilon,\delta}\|_{L_{t}^\infty \mathrm{L}_x^2}^2+\|\nabla \phi^R_{\varepsilon, \delta}\|_{L_{t}^\infty L_x^2}^2+\|\psi^R_{\varepsilon,\delta}\|_{L_{t}^\infty L^2(\Gamma)}^2+\|\nabla_{\tau}\psi^R_{\varepsilon, \delta}\|_{L_{t}^\infty L^2(\Gamma)}^2\right)^p\nonumber\\
&+\mathbb{E}\left(\sup_{t\in [0,T]}\int_{\Gamma}G(\psi^R_{\varepsilon, \delta})dS+\sup_{t\in [0,T]}\int_{\mathcal{D}}F(\phi^R_{\varepsilon, \delta})dx\right)^p\nonumber\\
&+\mathbb{E}\left(\|(\nabla u^R_{\varepsilon,\delta}, \nabla \mu^R_{\varepsilon,\delta})\|_{L_{t}^2L_x^2}^2+\|\mathcal{K}(\psi^R_{\varepsilon, \delta})\|_{L_{t}^2L^2(\Gamma)}^2+\delta\|\partial_t\phi^R_{\varepsilon,\delta}\|_{L_{t}^2L_x^2}^2\right)^p\leq C.
\end{align}
Using assumptions ${\bf A}_3$, ${\bf A}_4$  and \eqref{main1*}, we get
\begin{align}\label{main1}
&\mathbb{E}\left(\|u^R_{\varepsilon,\delta}\|_{L_{t}^\infty \mathrm{L}_x^2}^2+\|\nabla \phi^R_{\varepsilon, \delta}\|_{L_{t}^\infty L_x^2}^2+\|\psi^R_{\varepsilon,\delta}\|_{L_{t}^\infty L^2(\Gamma)}^2+\|\nabla_{\tau}\psi^R_{\varepsilon, \delta}\|_{L_{t}^\infty L^2(\Gamma)}^2\right)^p\nonumber\\
&+\mathbb{E}\left(\|(\nabla u^R_{\varepsilon,\delta}, \nabla \mu^R_{\varepsilon,\delta})\|_{L_{t}^2L_x^2}^2+\|\mathcal{K}(\psi^R_{\varepsilon, \delta})\|_{L_{t}^2L^2(\Gamma)}^2+\delta\|\partial_t\phi^R_{\varepsilon,\delta}\|_{L_{t}^2L_x^2}^2\right)^p\leq C.
\end{align}

Since $\langle \partial_t\phi^R_{\varepsilon,\delta}\rangle=0$, a.e. on $[0,T]\times \Omega$, we also have
\begin{align*}
\langle\mu^R_{\varepsilon,\delta}\rangle&=\langle-\Delta \phi^R_{\varepsilon, \delta}\rangle+\langle f_\varepsilon(\phi^R_{\varepsilon, \delta})\rangle\nonumber\\
&=-|\mathcal{D}|^{-1}|\Gamma|(\langle\mathcal{K}(\psi^R_{\varepsilon,\delta})\rangle_\Gamma-\langle\psi^R_{\varepsilon,\delta}\rangle+\langle g_\varepsilon(\psi^R_{\varepsilon,\delta})\rangle)+\langle f_\varepsilon(\phi^R_{\varepsilon, \delta})\rangle,
\end{align*}
as well as \eqref{main1}, assumption ${\bf A}_3$, and \eqref{3.7*}, \eqref{3.8**}, we infer that $\langle \mu^R_{\varepsilon,\delta}\rangle\in L^2_t$, $\mathbb{P}$-a.s. By (\ref{main1}) again, the Poincar\'{e} inequality yields
$$\mu^R_{\varepsilon,\delta}\in L_\omega^pL_t^2H^1,~ {\rm uniformly ~in} ~\varepsilon, \delta, R.$$
Moreover, $(\phi^R_{\varepsilon, \delta}, \psi^R_{\varepsilon, \delta})$ is the solution of the following elliptic boundary value problem:
\begin{align*}
-\Delta \phi^R_{\varepsilon,\delta}=\Psi^R_{\varepsilon,\delta}, ~\mathbb{P}\mbox{-a.s.} ~{\rm in}~\mathcal{D}\times (0,T),\\
A_\tau  \psi^R_{\varepsilon, \delta}+\partial_n \phi^R_{\varepsilon,\delta}+\psi^R_{\varepsilon,\delta}=\widetilde{\Psi}^R_{\varepsilon,\delta}, ~\mathbb{P}\mbox{-a.s.} ~{\rm on}~\Gamma\times (0,T),
\end{align*}
where
$$\Psi^R_{\varepsilon,\delta}= \mu^R_{\varepsilon,\delta}-\delta\partial_t\phi^R_{\varepsilon,\delta}-f_\varepsilon(\phi^R_{\varepsilon,\delta})\in L_\omega^p L^2_tL^2_x,$$
$$\widetilde{\Psi}^R_{\varepsilon,\delta}=\mathcal{K}(\psi^R_{\varepsilon,\delta})-g_\varepsilon(\psi^R_{\varepsilon,\delta})\in L^p_\omega L^2_tL^2(\Gamma).$$
By the elliptic regularity theory and Lemma \ref{4.4*}, we have
\begin{align*}
\mathbb{E}\|(\phi^R_{\varepsilon, \delta}, \psi^R_{\varepsilon, \delta})\|_{L_t^2V^2}^p\leq C,
\end{align*}
where constant $C$ is independent of $\varepsilon, \delta, R$.

In conclusion, we have
\begin{eqnarray}\label{3.39}
\left\{\begin{array}{ll}
\!\!u^R_{\varepsilon,\delta}\in L_\omega^p(L_t^\infty \mathrm{L}^2_x\cap L^2_t\mathrm{H}^1),\\  \mu^R_{\varepsilon,\delta}\in L_\omega^pL_t^2H^1,\\
\!\!(\phi^R_{\varepsilon,\delta}, \psi^R_{\varepsilon,\delta})\in L_w^p(L^\infty_tV^1\cap L^2_xV^2),\\ \mathcal{K}(\psi^R_{\varepsilon,\delta})\in L_w^pL^2_tL^2(\Gamma),\\
\!\!\delta^\frac{1}{2}\partial_t\phi^R_{\varepsilon,\delta}\in L_w^pL^2_tL_x^2,
\end{array}\right.
\end{eqnarray}
for all $p\geq 2$, uniformly in $\varepsilon,\delta, R$.

With these estimates in hand, we could build the existence and uniqueness of approximate solutions $(u_{\varepsilon, \delta}, \phi_{\varepsilon, \delta}, \psi_{\varepsilon, \delta})$ for any fixed $\varepsilon,\delta$ by passing $R\rightarrow\infty$ with the spirit of \cite{bre}. To this end, we define
\begin{align*}
\mathcal{T}_R=&\inf\left\{t\in [0,T];\|u^R_{\varepsilon, \delta}(t)\|^2_{\mathrm{L}_x^2}\geq R\right\}\\
&\qquad\qquad\wedge\inf\left\{t\in [0,T];\|(\phi^R_{\varepsilon,\delta}(t),\psi^R_{\varepsilon,\delta}(t))\|^2_{ V^1}\geq R\right\},
\end{align*}
if the set is empty, we take $\mathcal{T}_R=T$.  We infer from above argument that $(u_{\varepsilon, \delta}^R, \phi_{\varepsilon, \delta}^R, \psi_{\varepsilon, \delta}^R)$ is a unique solution on $[0,\mathcal{T}_R)$. Note that $\mathcal{T}_R$ is an increasing sequence, thus, $\mathcal{T}_{R'}\geq \mathcal{T}_R$ when $R'>R$. Then, the uniqueness implies $(u_{\varepsilon, \delta}^{R'}, \phi_{\varepsilon, \delta}^{R'}, \psi_{\varepsilon, \delta}^{R'})=(u_{\varepsilon, \delta}^R, \phi_{\varepsilon, \delta}^R, \psi_{\varepsilon, \delta}^R)$ on $[0, \mathcal{T}_R)$. As a result, we can define $(u_{\varepsilon, \delta}, \phi_{\varepsilon, \delta}, \psi_{\varepsilon, \delta})=(u_{\varepsilon, \delta}^R, \phi_{\varepsilon, \delta}^R, \psi_{\varepsilon, \delta}^R)$ in $[0, \mathcal{T}_R)$. Finally, we need to show that $$\mathbb{P}\left\{\sup_{R\in \mathbb{N}}\mathcal{T}_R=T\right\}=1,$$
that is, the blow-up does not occur within any finite time.

Indeed, using the Chebyshev inequality and \eqref{3.39}
\begin{align*}
\mathbb{P}\left\{\sup_{R\in \mathbb{N}}\mathcal{T}_{R}=T\right\}&=1-\mathbb{P}\left\{\left(\sup_{R\in \mathbb{N}}\mathcal{T}_{R}=T\right)^c\right\}=1-\mathbb{P}\left\{\sup_{R\in \mathbb{N}}\mathcal{T}_{R}<T\right\}\\
&\geq 1-\mathbb{P}\left\{\mathcal{T}_{R}<T\right\}\\
&\geq 1-\frac{C}{R}\mathbb{E}\left(\|u^R_{\varepsilon, \delta}\|^2_{L_t^\infty \mathrm{L}_x^2\cap L_t^2\mathrm{H}^1}+\|(\phi^R_{\varepsilon,\delta},\psi^R_{\varepsilon,\delta})\|^2_{L_t^\infty V^1\cap L_t^2V^2 }\right).
\end{align*}
Passing $R\rightarrow\infty$ on both sides, the result follows, making sure that we could define $(u_{\varepsilon, \delta}, \phi_{\varepsilon, \delta}, \psi_{\varepsilon, \delta})$ in the whole time interval $[0,T]$ for any $T>0$.

In order to show the energy inequality 5 in Definition \ref{def2.1}, let
\begin{align*}
&E_{\varepsilon, \delta}(t):=\frac{1}{2}\left(\|u^R_{\varepsilon,\delta}\|_{\mathrm{L}_x^2}^2+\|\nabla \phi^R_{\varepsilon, \delta}\|_{L_x^2}^2+\|\psi^R_{\varepsilon,\delta}\|_{L^2(\Gamma)}^2+\|\nabla_{\tau}\psi^R_{\varepsilon, \delta}\|_{L^2(\Gamma)}^2+2\int_{\mathcal{D}}F(\phi^R_{\varepsilon, \delta})dx\right)\nonumber\\
&\qquad\qquad+\int_{\Gamma}G(\psi^R_{\varepsilon, \delta})dS.
\end{align*}
Then after an applications of It\^{o}'s formula of $E^2_{\varepsilon, \delta}(t)$, from \eqref{3.52} we have
\begin{align}\label{3.59}
&\frac{1}{2}E^2_{\varepsilon, \delta}(t)+\int_{0}^{t}E_{\varepsilon, \delta}(\|(\nabla u^R_{\varepsilon,\delta}, \nabla \mu^R_{\varepsilon,\delta})\|_{L_x^2}^2+\|\mathcal{K}(\psi^R_{\varepsilon, \delta})\|_{L^2(\Gamma)}^2+\delta\|\partial_t\phi^R_{\varepsilon,\delta}\|_{L_x^2}^2)ds\nonumber\\
&=\frac{1}{2}E^2_{\varepsilon, \delta}(0)+\int^t_0E_{\varepsilon, \delta}(h(u^R_{\varepsilon, \delta},\nabla\phi^R_{\varepsilon, \delta}), u^R_{\varepsilon,\delta})d\mathcal{W}+\frac{1}{2}\int^t_0E_{\varepsilon, \delta}\|h(u^R_{\varepsilon, \delta},\nabla\phi^R_{\varepsilon, \delta})\|^2_{L_2(\mathcal{H};\mathrm{L}_x^2)}ds\nonumber\\
&\quad+\frac{1}{2}\int^t_0(h(u^R_{\varepsilon, \delta},\nabla\phi^R_{\varepsilon, \delta}), u^R_{\varepsilon,\delta})^2ds.
\end{align}
By the elliptic regularity theory and Lemma \ref{4.4*} again, from \eqref{3.59} we have
\begin{align}\label{3.60}
&\frac{1}{2}E^2_{\varepsilon, \delta}(t)\nonumber\\&+\int_{0}^{t}E_{\varepsilon, \delta}(\|(\nabla u^R_{\varepsilon,\delta}, \nabla \mu^R_{\varepsilon,\delta})\|_{L_x^2}^2+\|\mathcal{K}(\psi^R_{\varepsilon, \delta})\|_{L^2(\Gamma)}^2+\delta\|\partial_t\phi^R_{\varepsilon,\delta}\|_{L_x^2}^2+\|(\phi^R_{\varepsilon, \delta}, \psi^R_{\varepsilon, \delta})\|^2_{V^2})ds\nonumber\\
&\leq E^2_{\varepsilon, \delta}(0)+2\int^t_0E_{\varepsilon, \delta}(h(u^R_{\varepsilon, \delta},\nabla\phi^R_{\varepsilon, \delta}), u^R_{\varepsilon,\delta})d\mathcal{W}+\int^t_0E_{\varepsilon, \delta}\|h(u^R_{\varepsilon, \delta},\nabla\phi^R_{\varepsilon, \delta})\|^2_{L_2(\mathcal{H};L_x^2)}ds\nonumber\\
&\quad+\int^t_0(h(u^R_{\varepsilon, \delta},\nabla\phi^R_{\varepsilon, \delta}), u^R_{\varepsilon,\delta})^2ds.
\end{align}
Inequality \eqref{3.60} will be used later.

{\bf Part 3. Conduct the compactness argument}

In the stochastic setting, embedding $L^{2}(\Omega;X)\hookrightarrow L^{2}(\Omega;Y)$ might not be compact, even if the embedding $X$ into $Y$ is compact. As a result, the usual compactness criteria, such as the Aubin or Arzel\`{a}-Ascoli type theorems, can not be used directly. We turn to invoke the Skorokhod representative theorem to obtain the compactness, which require us to show the tightness of the probability set induced by the law of sequence $\{u_{\varepsilon,\delta}, \phi_{\varepsilon,\delta}, \psi_{\varepsilon,\delta}\}$. Toward the goal, define
$$\mathcal{P}_{\varepsilon, \delta}(B)=\mathcal{P}^1_{\varepsilon, \delta}\times \mathcal{P}^2_{\varepsilon, \delta}\times \mathcal{P}_\mathcal{W}(B)$$
for any $B\in \mathcal{B}(Y)$, where the Banach space $Y=Y_1\times Y_2\times Y_3$ is defined as
$$Y_1=(L_t^\infty \mathrm{L}_x^2)_{w^*}\cap L_t^2\mathrm{L}_x^2\cap (L_t^2\mathrm{H}^1)_{w}, ~Y_2=(L_t^\infty V^1)_{w^*}\cap L_t^2V^{2-\sigma}\cap (L^2_tV^2)_w,~ Y_3=C_t\mathcal{H}_0,$$
with $\sigma\in (0,\frac{1}{2})$.
It is enough to show each set $\mathcal{P}_{\varepsilon,\delta}^i$ is tight on the corresponding space $Y_i$.

{\bf Claim} 1. The probability set $\mathcal{P}^1_{\varepsilon, \delta}$ is tight on $Y_1$.
\begin{proof}
We first show that there exists a constant $\beta>0$ such that,
\begin{align}\label{3.40}
 \mathbb{E}\|u_{\varepsilon,\delta}\|^p_{C_t^\beta(\mathrm{H}^{1,3})'}\leq C,
\end{align}
where constant $C$ is independent of $\varepsilon, \delta$.

For any $v\in \mathrm{H}^{1,3}$, we have
\begin{align*}
(u_{\varepsilon,\delta},v)&+\int_{0}^{t}a_0(u_{\varepsilon,\delta}, v)ds+\int_{0}^{t}B_0( u_{\varepsilon, \delta},\mathcal{J}_\varepsilon u_{\varepsilon, \delta}, v)ds-\int_{0}^{t}B_1(\mu_{\varepsilon, \delta}, \mathcal{J}_{\varepsilon}\phi_{\varepsilon, \delta}, v)ds\nonumber\\
&-\int_{0}^{t}(\mathcal{K}(\psi_{\varepsilon,\delta})\nabla_{\tau}(J_\varepsilon \psi_{\varepsilon,\delta}), v_\tau)_{\Gamma}ds=(u_0,v)+\int_{0}^{t}(h(u_{\varepsilon, \delta},\nabla\phi_{\varepsilon, \delta}), v)d\mathcal{W}.
\end{align*}
For a.s. $\omega$, and for any $\epsilon>0$, there exist $t_1,t_2\in [0,T]$ such that

$$\sup_{t_0\neq t'_0}\frac{\left|\int_{t_0}^{t'_0}(h(u_{\varepsilon, \delta},\nabla\phi_{\varepsilon, \delta}), v)d\mathcal{W}\right|}{|t_0-t'_0|^\beta}\leq \frac{\left|\int_{t_1}^{t_2}(h(u_{\varepsilon, \delta},\nabla\phi_{\varepsilon, \delta}), v)d\mathcal{W}\right|}{|t_1-t_2|^\beta}+\epsilon,$$
together with assumption ${\bf A}_1$ and the Burkholder-Davis-Gundy inequality, we see
\begin{align}\label{3.41}
&\mathbb{E}\left(\sup_{t_0\neq t'_0}\frac{\left|\int_{t_0}^{t'_0}(h(u_{\varepsilon, \delta},\nabla\phi_{\varepsilon, \delta}), v)d\mathcal{W}\right|}{|t_0-t'_0|^\beta}\right)^p\nonumber\\
&\leq\frac{\mathbb{E}\left(\int_{t_1}^{t_2}\|h(u_{\varepsilon, \delta},\nabla\phi_{\varepsilon, \delta})\|_{L_2(\mathcal{H};L_x^2)}^2 \|v\|_{\mathrm{L}_x^2}^2dt\right)^\frac{p}{2}}{|t_1-t_2|^{\beta p}}+\epsilon^p\nonumber\\
&\leq\frac{C(t_2-t_1)^\frac{p}{2}\mathbb{E}(1+\|(u_{\varepsilon, \delta},\nabla\phi_{\varepsilon, \delta})\|_{L_t^\infty L_x^2})^p}{|t_1-t_2|^{\beta p}}+\epsilon^p\nonumber\\
&\leq C(t_2-t_1)^{\left(\frac{1}{2}-\beta\right)p}+\epsilon^p\leq C,
\end{align}
where $\beta< \frac{1}{2}$.
Applying Lemma \ref{lem4.3}, Lemma \ref{lem4.4}  and H\"{o}lder's inequality, we obtain
\begin{align}
\|B_0( u_{\varepsilon, \delta},\mathcal{J}_\varepsilon u_{\varepsilon, \delta}, v)\|_{L_t^2}&\leq \big\|\|u_{\varepsilon,\delta}\|_{\mathrm{H}^1}\|\nabla \mathcal{J}_\varepsilon u_{\varepsilon, \delta}\|_{\mathrm{H}^{-1}}\big\|_{L_t^2}\|v\|_{\mathrm{H}^{1,3}}\nonumber\\ &\leq C\|u_{\varepsilon, \delta}\|_{L_t^2\mathrm{H}^1}\|u_{\varepsilon, \delta}\|_{L_t^\infty \mathrm{L}_x^2},\label{5.42}\\
\|B_1(\mu_{\varepsilon, \delta}, \mathcal{J}_{\varepsilon}\phi_{\varepsilon, \delta}, v)\|_{L_t^2} &\leq\left\|\|\mu_{\varepsilon,\delta}\|_{L_x^6}\|\nabla \mathcal{J}_\varepsilon \phi_{\varepsilon, \delta}\|_{L_x^{2}}\right\|_{L_t^2}\|v\|_{\mathrm{L}_x^{3}}\nonumber\\&\leq C\|\mu_{\varepsilon, \delta}\|_{L_t^2H^1}\|\phi_{\varepsilon, \delta}\|_{L_t^\infty H^1}.
\end{align}
Choosing $s=\frac{1}{2}$ in Lemma \ref{lem4.1}, and using (\ref{3.39}), we infer that
\begin{align}
\mathbb{E}\|\nabla_{\tau}(J_\varepsilon \psi_{\varepsilon,\delta})\|_{L_t^3L^4(\Gamma)}^p\leq C,
\end{align}
where $C$ is independent of $\varepsilon,\delta$. By H\"{o}lder's inequality, obtaining
\begin{align}\label{5.45}
\begin{split}
\|(\mathcal{K}(\psi_{\varepsilon,\delta})\nabla_{\tau}(J_\varepsilon \psi_{\varepsilon,\delta}), v_\tau)_{\Gamma}\|_{L_t^\frac{6}{5}}&\leq \|v_\tau\|_{\mathrm{L}^4(\Gamma)}\big\|\|\mathcal{K}(\psi_{\varepsilon,\delta})\|_{L^2(\Gamma)}\|\nabla_{\tau}(J_\varepsilon \psi_{\varepsilon,\delta})\|_{L^4(\Gamma)}\big\|_{L_t^\frac{6}{5}}\\
&\leq C\|\mathcal{K}(\psi_{\varepsilon,\delta})\|_{L^2_tL^2(\Gamma)}\|\nabla_{\tau}(J_\varepsilon \psi_{\varepsilon,\delta})\|_{L_t^3L^4(\Gamma)}.
\end{split}
\end{align}
Combining (\ref{5.42})-(\ref{5.45}), we have
\begin{align}\label{3.46}
\mathbb{E}\left\|u_{\varepsilon,\delta}-\int_{0}^{t}h(u_{\varepsilon, \delta},\nabla\phi_{\varepsilon, \delta})d\mathcal{W}\right\|^p_{W^{1,\frac{6}{5}}_t(\mathrm{H}^{1,3})'}\leq C.
\end{align}
The Sobolev embedding
\begin{align*}
W^{\alpha,p}_tX\hookrightarrow C^\beta_tX,~{\rm for}~ \beta<\alpha-\frac{1}{p},
\end{align*}
and (\ref{3.46}) give
\begin{align}\label{3.66}
\mathbb{E}\left\|u_{\varepsilon,\delta}-\int_{0}^{t}h(u_{\varepsilon, \delta},\nabla\phi_{\varepsilon, \delta})d\mathcal{W}\right\|^p_{C^\beta_t(\mathrm{H}^{1,3})'}\leq C, ~{\rm for}~ \beta<\frac{1}{6}.
\end{align}
Estimates (\ref{3.41}) and \eqref{3.66} yield the desired result \eqref{3.40}.

For positive constant $K$, define set
\begin{align*}
B_{1,K}=\left\{u_{\varepsilon,\delta}\in L_t^2\mathrm{H}^1\cap C_t^\alpha(\mathrm{H}^{1,3})':\|u_{\varepsilon,\delta}\|_{ L_t^2\mathrm{H}^1}+\|u_{\varepsilon,\delta}\|_{C_t^\alpha(\mathrm{H}^{1,3})'}\leq K\right\}.
\end{align*}
The Aubin-Lions compact embedding (see Lemma \ref{lem4.5})
$$L_t^2H^1\cap C_t^\alpha(H^{1,3})'\hookrightarrow L^2_tL^2_x,$$
implies that the set $B_{1,K}$ is relatively compact in $L^2_tL^2_x$. Considering (\ref{3.39}), (\ref{3.40}), using the Chebyshev inequality, we have
\begin{align*}
\begin{split}
\mathcal{P}^1_{\varepsilon, \delta}(B_{1,K}^c)&\leq \mathbb{P}\left(\|u_{\varepsilon,\delta}\|_{ L_t^2\mathrm{H}^1}>\frac{K}{2}\right)+\mathbb{P}\left(\|u_{\varepsilon,\delta}\|_{ {C_t^\alpha(\mathrm{H}^{1,3})'}}>\frac{K}{2}\right)\\
&\leq \frac{2}{K}\mathbb{E}(\|u_{\varepsilon,\delta}\|_{ L_t^2\mathrm{H}^1}+\|u_{\varepsilon,\delta}\|_{C_t^\alpha(\mathrm{H}^{1,3})'})\leq \frac{C}{K}.
\end{split}
\end{align*}
Owing to the Banach-Alaoglu theorem, two sets
\begin{eqnarray*}
&&B_{2,K}=\left\{u_{\varepsilon,\delta}\in L_t^{2} \mathrm{H}^1: \|u_{\varepsilon,\delta}\|_{L_t^{2} \mathrm{H}^1}\leq K\right\},\\
&&B_{3,K}=\left\{u_{\varepsilon,\delta}\in L_t^\infty \mathrm{L}_x^2: \|u_{\varepsilon,\delta}\|_{L_t^\infty \mathrm{L}_x^2}\leq K\right\},
\end{eqnarray*}
are relatively compact on path spaces $(L_t^{2} \mathrm{H}^1)_w, (L_t^\infty \mathrm{L}_x^2)_{w^*}$, respectively. $(\ref{3.39})_1$ and the Chebyshev inequality yield that
$$\mathcal{P}^1_{\varepsilon, \delta}((B_{1,K}\cap B_{2,K})^c)\leq \mathcal{P}^1_{\varepsilon, \delta}(B_{1,K}^c)+\mathcal{P}^1_{\varepsilon, \delta}(B_{2,K}^c)\leq \frac{C}{K}.$$
Then, we obtain the tightness of set $\mathcal{P}^1_{\varepsilon, \delta}$.
\end{proof}

{\bf Claim} 2. The probability set $\mathcal{P}^2_{\varepsilon, \delta}$ is tight on $Y_2$.
\begin{proof}
To this end, we first give the time regularity estimates of $(\phi_{\varepsilon, \delta}, \psi_{\varepsilon,\delta})$. For any $v\in H^1$, H\"{o}lder's inequality, Lemmas \ref{lem4.3} and \ref{lem4.4} imply that
\begin{align}\label{3.49}
\begin{split}
\|B_2(u_{\varepsilon, \delta}, \mathcal{J}_{\varepsilon}\phi_{\varepsilon, \delta}, v)\|_{L^2_t}&\leq \left\|\|v\|_{H^1}\|u_{\varepsilon, \delta}\|_{\mathrm{L}_x^2}\| \mathcal{J}_{\varepsilon}\phi_{\varepsilon, \delta}\|_{H^2}\right\|_{L^2_t}\\
&\leq \|v\|_{H^1}\|u_{\varepsilon, \delta}\|_{L_t^\infty \mathrm{L}_x^2}\| \phi_{\varepsilon, \delta}\|_{L^2_tH^2}.
\end{split}
\end{align}
For any $v\in L^2(\Gamma)$, using the trace embedding $H^1(\mathcal{D})\hookrightarrow H^\frac{1}{2}(\Gamma)$ and the embedding $ H^\frac{1}{2}(\Gamma)\hookrightarrow L^4(\Gamma)$, Lemma \ref{lem4.4}, and the Gagliardo-Nirenberg inequality $$\|u\|^2_{L^4(\Gamma)}\leq \|u\|_{L^2(\Gamma)}\|\nabla u\|_{L^2(\Gamma)},$$ we have
\begin{align}\label{3.50}
\|B_\Gamma(u_{\varepsilon, \delta},J_{\varepsilon}(\psi_{\varepsilon,\delta}), v)\|_{L_t^\frac{4}{3}}&\leq C\left\|\|v\|_{L^2(\Gamma)}\|u_{\varepsilon, \delta}\|_{\mathrm{L}^4(\Gamma)}\|\nabla J_{\varepsilon}(\psi_{\varepsilon,\delta})\|_{L^4(\Gamma)}\right\|_{L_t^\frac{4}{3}}\nonumber\\
&\leq C\|v\|_{L^2(\Gamma)}\|u_{\varepsilon, \delta}\|_{L_t^2\mathrm{H}^{\frac{1}{2}}(\Gamma)}\|\nabla J_{\varepsilon}(\psi_{\varepsilon,\delta})\|_{L_t^4L^4(\Gamma)}\nonumber\\
&\leq C\|v\|_{L^2(\Gamma)}\|u_{\varepsilon, \delta}\|_{L_t^2\mathrm{H}^1}\left\|\|\nabla J_{\varepsilon}(\psi_{\varepsilon,\delta})\|^\frac{1}{2}_{L^2(\Gamma)}\|\nabla J_{\varepsilon}(\psi_{\varepsilon,\delta})\|^\frac{1}{2}_{H^1}\right\|_{L_t^4}\nonumber\\ &\leq C\|u_{\varepsilon, \delta}\|_{L_t^2\mathrm{H}^1}\|\nabla J_{\varepsilon}(\psi_{\varepsilon,\delta})\|^\frac{1}{2}_{L^\infty_tL^2(\Gamma)}\|\nabla J_{\varepsilon}(\psi_{\varepsilon,\delta})\|^\frac{1}{2}_{L_t^2H^1}\nonumber\\
&\leq C\|u_{\varepsilon, \delta}\|_{L_t^2\mathrm{H}^1}\|\nabla \psi_{\varepsilon,\delta}\|_{L^\infty_tL^2(\Gamma)}^\frac{1}{2}\|\nabla \psi_{\varepsilon,\delta}\|^\frac{1}{2}_{L_t^2H^1}.
\end{align}
By (\ref{3.49}), (\ref{3.50}), (\ref{3.39}), we infer from equations $\eqref{equ3.1}_2, \eqref{equ3.1}_3$ that
$$\mathbb{E}\|\partial_t(\phi_{\varepsilon, \delta}, \psi_{\varepsilon,\delta})\|_{L^\frac{4}{3}_t((H^1)'\times L^2(\Gamma))}^p\leq C,$$
where $C$ is independent of $\varepsilon, \delta$.
Then, using the Aubin-Lions compact embedding,
$$L_t^2V^2\cap W_t^{1,\frac{4}{3}}(V^1)'\hookrightarrow L_t^2V^{2-\sigma},$$
and  $(\ref{3.39})_3$, we can deduce that the probability set $\mathcal{P}^2_{\varepsilon, \delta}$ is tight by the same argument as that of in Claim 1.
\end{proof}
Since the sequence $\mathcal{W}$ is only one element, the set $\{\mathcal{P}_\mathcal{W}\}$ is weakly compact.

Consequently, we have the following result by the Skorokhod representative theorem (see Theorem \ref{thm4.2}).
\begin{proposition}\label{pro3.1} There exist a new sequence $(\tilde{u}_{\varepsilon,\delta}, \tilde{\phi}_{\varepsilon,\delta}, \tilde{\psi}_{\varepsilon,\delta}, \widetilde{\mathcal{W}}_{\varepsilon,\delta})$, a new probability space $(\widetilde{\Omega}, \widetilde{\mathcal{F}}, \widetilde{\mathbb{P}})$ and $Y$-valued processes $(\tilde{u}, \tilde{\phi}, \tilde{\psi}, \widetilde{\mathcal{W}})$ such that\\
{\rm 1}. the joint law of $(\tilde{u}_{\varepsilon,\delta}, \tilde{\phi}_{\varepsilon,\delta}, \tilde{\psi}_{\varepsilon,\delta}, \widetilde{\mathcal{W}}_{\varepsilon,\delta})$ is $\mathcal{P}_{\varepsilon, \delta}$,\\
{\rm 2}. the sequence $(\tilde{u}_{\varepsilon,\delta}, \tilde{\phi}_{\varepsilon,\delta}, \tilde{\psi}_{\varepsilon,\delta}, \widetilde{\mathcal{W}}_{\varepsilon,\delta} )$ converges to $(\tilde{u}, \tilde{\phi}, \tilde{\psi}, \widetilde{\mathcal{W}})$ in the topology of $Y$,~ $\widetilde{\mathbb{P}}${\mbox-a.s.},\\
{\rm 3}. the measure set $\mathcal{P}_{\varepsilon, \delta}$ converges to $\mathcal{P}$ where $\mathcal{P}$ is a Radon measure with $\widetilde{\mathbb{P}}(\tilde{u}, \tilde{\phi}, \tilde{\psi}\in \cdot)=\mathcal{P}(\cdot)$.
\end{proposition}

By Proposition \ref{pro3.1} and the same argument with \cite{WW,ww}, we have the following result.
\begin{proposition}\label{pro3.2} The sequence $(\tilde{u}_{\varepsilon,\delta}, \tilde{\phi}_{\varepsilon,\delta}, \tilde{\psi}_{\varepsilon,\delta}, \widetilde{\mathcal{W}}_{\varepsilon,\delta})$ still satisfies system (\ref{equ3.1}) relative to the stochastic basis $\widetilde{S}^{\varepsilon,\delta}:=(\widetilde{\Omega},\widetilde{\mathcal{F}},\widetilde{\mathbb{P}}, \{\widetilde{\mathcal{F}}_t^{\varepsilon,\delta}\}_{t\geq 0}, \widetilde{\mathcal{W}}_{\varepsilon,\delta} )$, where $\widetilde{\mathcal{F}}_t^{\varepsilon,\delta}$ is a canonical filtration defined by $$\sigma\left(\sigma\left((\tilde{u}_{\varepsilon,\delta}(s), \tilde{\phi}_{\varepsilon,\delta}(s), \tilde{\psi}_{\varepsilon,\delta}(s), \widetilde{\mathcal{W}}_{\varepsilon,\delta}(s)):s\leq t\right)\cup \left\{\Sigma\in \widetilde{\mathcal{F}}; \widetilde{\mathbb{P}}(\Sigma)=0\right\}\right),$$ and also the sequence $(\tilde{u}_{\varepsilon,\delta}, \tilde{\phi}_{\varepsilon,\delta}, \tilde{\psi}_{\varepsilon,\delta}, \widetilde{\mathcal{W}}_{\varepsilon,\delta})$ shares the uniform estimates (\ref{3.39}).
\end{proposition}

{\bf Part 4. Identify the limit.}

To simplify the notation, we still use $(u_{\varepsilon, \delta}, \phi_{\varepsilon, \delta}, \psi_{\varepsilon, \delta}, \mathcal{W}_{\varepsilon,\delta})$ standing for the new sequence. Using $(\ref{3.39})_5$, it holds
$$\|\delta\partial_t\phi_{\varepsilon, \delta}\|_{L_\omega^{p}L_t^2L_x^2}=\delta^\frac{1}{2}\left(\widetilde{\mathbb{E}}\|\delta^\frac{1}{2}\partial_t\phi_{\varepsilon, \delta}\|_{L_t^2L_x^2}^{p}\right)^\frac{1}{p}\rightarrow 0,$$
as $\varepsilon,\delta\rightarrow 0$. By thinning the sequence (still denoted by $\varepsilon, \delta$), we have $\widetilde{\mathbb{P}}$-a.s.
\begin{align}\label{3.51}
\delta\partial_t\phi_{\varepsilon, \delta}\rightarrow 0,~{\rm in}~ L^2_tL^2_x.
\end{align}
Using Proposition \ref{pro3.1}, Property (ii) and $(\ref{3.39})_3$ again, we get $\widetilde{\mathbb{P}}$-a.s.
$$f_\varepsilon(\phi_{\varepsilon,\delta})\rightarrow f(\phi),~{\rm in}~ L^2_tL_x^2.$$
From (\ref{3.51}), $(\ref{3.39})_2$ and Proposition \ref{pro3.1}, we know that $\widetilde{\mathbb{P}}$-a.s.
$$\mu_{\varepsilon, \delta}\rightharpoonup \mu, ~{\rm in}~ L^2_tH^1.$$

For any $v_0\in D(A_0)$, decompose
$$B_0(u_{\varepsilon, \delta}, \mathcal{J}_\varepsilon u_{\varepsilon, \delta}, v_0)-B_0(u, u, v_0)=B_0( u,\mathcal{J}_\varepsilon u_{\varepsilon, \delta}-u, v_0)+B_0( u_{\varepsilon,\delta}-u,\mathcal{J}_\varepsilon u_{\varepsilon, \delta}, v_0).$$
By $u_{\varepsilon,\delta}\rightarrow u$ in $L_t^2\mathrm{L}_x^2, ~\widetilde{\mathbb{P}}\mbox{-a.s.}$, Lemma \ref{lem4.4}, and $(\ref{3.39})_1$, we have $\widetilde{\mathbb{P}}\mbox{-a.s.}$
$$\int_{0}^{t}|B_0( u_{\varepsilon,\delta}-u, \mathcal{J}_\varepsilon u_{\varepsilon, \delta}, v_0)|ds\leq \|v_0\|_{H^2}\|u_{\varepsilon, \delta}-u\|_{L_t^2\mathrm{L}_x^2}\|\mathcal{J}_\varepsilon u_{\varepsilon, \delta}\|_{L_t^2\mathrm{H}^1}\rightarrow 0.$$
Furthermore, $\mathcal{J}_\varepsilon u_{\varepsilon,\delta}\rightharpoonup u$ in $L_t^2\mathrm{H}^1, ~\widetilde{\mathbb{P}}\mbox{-a.s.}$ and (\ref{3.39}) imply
$$\int_{0}^{t}B_0( u, \mathcal{J}_\varepsilon u_{\varepsilon, \delta}-u, v_0)ds\rightarrow 0,  ~\widetilde{\mathbb{P}}\mbox{-a.s.}$$
And
$$B_1(\mu_{\varepsilon, \delta}, \mathcal{J}_{\varepsilon}\phi_{\varepsilon, \delta}, v_0)-B_1(\mu, \phi, v_0)=B_1(\mu_{\varepsilon, \delta}-\mu, \phi, v_0)+B_1(\mu_{\varepsilon,\delta}, \mathcal{J}_{\varepsilon}\phi_{\varepsilon, \delta}-\phi, v_0).$$
The convergence $\mathcal{J}_{\varepsilon,\delta}\phi_{\varepsilon,\delta}\rightarrow \phi$ in $L_t^2H^{2-\sigma}$,  $\mu_{\varepsilon, \delta}\rightharpoonup \mu$ in $L_t^2H^{1}$,~$\widetilde{\mathbb{P}}\mbox{-a.s.}$ and (\ref{3.39}) yield $\widetilde{\mathbb{P}}$-a.s.
\begin{align*}
&\int_{0}^{t}B_1(\mu_{\varepsilon, \delta}-\mu, \phi, v_0)ds\rightarrow 0, \\
&\int_{0}^{t}|B_1(\mu_{\varepsilon,\delta}, \mathcal{J}_{\varepsilon}\phi_{\varepsilon, \delta}-\phi, v_0)|ds\leq \|v_0\|_{\mathrm{H}^1}\|\mu_{\varepsilon,\delta}\|_{L_t^2H^1}\|\nabla (\mathcal{J}_{\varepsilon}\phi_{\varepsilon, \delta}-\phi)\|_{L_t^2L_x^2}\rightarrow 0.
\end{align*}
Using a same argument, we also have $\widetilde{\mathbb{P}}$-a.s.
$$\int_{0}^{t}B_2(u_{\varepsilon, \delta}, \mathcal{J}_{\varepsilon}\phi_{\varepsilon, \delta},v_1)-B_2(u,\phi, v_1)ds\rightarrow 0,$$
for any $v_1\in L^2$.

Next, we focus on the nonlinear boundary integral terms. For any $v_2\in L^2(\Gamma)$, decompose
$$B_\Gamma(u_{\varepsilon, \delta},J_{\varepsilon}\psi_{\varepsilon,\delta}, v_2)-B_\Gamma(u,\psi, v_2)=B_\Gamma(u_{\varepsilon, \delta}-u,\psi, v_2)+B_\Gamma(u_{\varepsilon, \delta},J_{\varepsilon}\psi_{\varepsilon,\delta}-\psi, v_2).$$
The fact that $u_{\varepsilon,\delta}\rightharpoonup u$ in $L_t^2\mathrm{H}^\frac{1}{2}(\Gamma)$, $\widetilde{\mathbb{P}}\mbox{-a.s.}$ together with $\psi\in L_t^2H^2(\Gamma)$, $\widetilde{\mathbb{P}}$-a.s. yields
$$\int_{0}^{t}B_\Gamma(u_{\varepsilon, \delta}-u,\psi, v_2)ds\rightarrow 0.$$
Moreover, by H\"{o}lder's inequality, the Sobolev embedding $H^\frac{1}{2}(\Gamma)\hookrightarrow L^4(\Gamma)$ and $J_{\varepsilon}\psi_{\varepsilon,\delta}\rightarrow \psi$ in $L^2_tH^{2-\sigma},\, \widetilde{\mathbb{P}}\mbox{-a.s.}$, we find $\widetilde{\mathbb{P}}\mbox{-a.s.}$
\begin{align*}
&\int_{0}^{t}|B_\Gamma(u_{\varepsilon, \delta},J_{\varepsilon}\psi_{\varepsilon,\delta}-\psi, v_2)|ds\nonumber\\
& \leq \|v_2\|_{L^2(\Gamma)}\|u_{\varepsilon, \delta}\|_{L_t^2\mathrm{L}^4(\Gamma)}\|J_{\varepsilon}\psi_{\varepsilon,\delta}-\psi\|_{L^2_tH^{1,4}(\Gamma)}\\
& \leq \|v_2\|_{L^2(\Gamma)}\|u_{\varepsilon, \delta}\|_{L_t^2\mathrm{H}^\frac{1}{2}(\Gamma)}\|J_{\varepsilon}\psi_{\varepsilon,\delta}-\psi\|_{L^2_tH^{2-\sigma}}\rightarrow 0.\nonumber
\end{align*}
The specific form of boundary term $\mathcal{K}(\psi_{\varepsilon,\delta})\nabla_{\tau}(J_\varepsilon \psi_{\varepsilon,\delta})$ is
$$\mathcal{K}(\psi_{\varepsilon,\delta})\nabla_{\tau}(J_\varepsilon \psi_{\varepsilon,\delta})=(A_{\tau}\psi_{\varepsilon,\delta}+\partial_n\phi_{\varepsilon, \delta}+\psi_{\varepsilon,\delta}+g_\varepsilon(\psi_{\varepsilon,\delta}))\nabla_{\tau}(J_\varepsilon \psi_{\varepsilon,\delta}).
$$
Using $\psi_{\varepsilon,\delta}\rightarrow \psi$ in $L^2_tH^{2-\sigma}, \widetilde{\mathbb{P}}\mbox{-a.s.}$ and $g_\varepsilon(\psi_{\varepsilon,\delta})\rightarrow g(\psi)~{\rm in}~ L^2_tL^2(\Gamma), \widetilde{\mathbb{P}}\mbox{-a.s.}$, we easily obtain $\widetilde{\mathbb{P}}$-a.s.
$$\int_{0}^{t}((\psi_{\varepsilon,\delta}+g_\varepsilon(\psi_{\varepsilon,\delta}))\nabla_{\tau}(J_\varepsilon \psi_{\varepsilon,\delta}),v_{0,\tau})_{\Gamma}ds\rightarrow \int_{0}^{t}((\psi+g(\psi))\nabla_{\tau}( \psi),v_{0,\tau})_{\Gamma}ds,$$
for any $v_0\in \mathrm{H}^1$.

By the trace theorem and Proposition \ref{pro3.1}, we infer that $\partial_n\phi_{\varepsilon,\delta}\rightarrow \partial_n \phi$ in $L_t^2H^{\frac{1}{2}-\sigma}(\Gamma)$ for any $\sigma\in (0,\frac{1}{2})$,  ~$\widetilde{\mathbb{P}}\mbox{-a.s.}$ as well as  $J_\varepsilon\psi_{\varepsilon,\delta}\rightarrow \psi$ in $L_t^2H^{2-\sigma}(\Gamma)$, $(\ref{3.39})_3$, the embedding $H^1(\mathcal{D})\hookrightarrow H^\frac{1}{2}(\Gamma)\hookrightarrow L^4(\Gamma)$ give $\widetilde{\mathbb{P}}$-a.s.
$$\int_{0}^{t}(\partial_n\phi_{\varepsilon, \delta}\nabla_{\tau}(J_\varepsilon \psi_{\varepsilon,\delta}),v_{0,\tau})_{\Gamma}ds\rightarrow \int_{0}^{t}(\partial_n\phi\nabla_{\tau} \psi,v_{0,\tau})_{\Gamma}ds.$$
Indeed,
\begin{align}
&\int_{0}^{t}(\partial_n\phi_{\varepsilon, \delta}\nabla_{\tau}(J_\varepsilon \psi_{\varepsilon,\delta}),v_{0,\tau})_{\Gamma}-(\partial_n\phi\nabla_{\tau}\psi,v_{0,\tau})_{\Gamma}ds\nonumber\\
&=\int_{0}^{t}((\partial_n\phi_{\varepsilon, \delta}-\partial_n\phi)\nabla_{\tau}(J_\varepsilon \psi_{\varepsilon,\delta}),v_{0,\tau} )_{\Gamma}+(\partial_n\phi(\nabla_{\tau}(J_\varepsilon \psi_{\varepsilon,\delta})-\nabla_{\tau}\psi),v_{0,\tau})_{\Gamma}ds\nonumber\\
&\leq \|\partial_n\phi_{\varepsilon, \delta}-\partial_n\phi\|_{L_t^2L^2(\Gamma)}\|\nabla_{\tau}(J_\varepsilon \psi_{\varepsilon,\delta})\|_{L^2_tL^4(\Gamma)}\|v_{0,\tau}\|_{\mathrm{L}^4(\Gamma)}\nonumber\\
&\quad+(\partial_n\phi(\nabla_{\tau}(J_\varepsilon \psi_{\varepsilon,\delta})-\nabla_{\tau}\psi),v_{0,\tau})_{\Gamma}\nonumber\\
&\leq \|\partial_n\phi_{\varepsilon, \delta}-\partial_n\phi\|_{L_t^2H^{\frac{1}{2}-\sigma}(\Gamma)}\|\nabla_{\tau}(J_\varepsilon \psi_{\varepsilon,\delta})\|_{L^2_tH^{2-\sigma}(\Gamma)}\|v_{0,\tau}\|_{\mathrm{L}^4(\Gamma)}\nonumber\\
&\quad+(\partial_n\phi(\nabla_{\tau}(J_\varepsilon \psi_{\varepsilon,\delta})-\nabla_{\tau}\psi),v_{0,\tau})_{\Gamma}
\rightarrow 0.\nonumber
\end{align}
Furthermore, by $(\ref{3.39})_3$ and $\psi_{\varepsilon,\delta}\rightarrow \psi$ in $(L_t^2H^2(\Gamma))_w$, we have $\widetilde{\mathbb{P}}$-a.s.
$$\int_{0}^{t}((A_{\tau}\psi_{\varepsilon,\delta}-A_{\tau}\psi)\nabla_\tau\psi, v_{0,\tau})_{\Gamma}ds\rightarrow 0.$$
Also, $(\ref{3.39})_3$, $J_\varepsilon\psi_{\varepsilon,\delta}\rightarrow \psi$ in $L_t^2H^{2-\sigma}(\Gamma)$ and H\"{o}lder's inequality give $\widetilde{\mathbb{P}}$-a.s.
\begin{align*}
&\int_{0}^{t}(A_{\tau}\psi_{\varepsilon,\delta}(\nabla_{\tau}(J_\varepsilon \psi_{\varepsilon,\delta})-\nabla_{\tau}\psi), v_{0,\tau})_{\Gamma}ds\nonumber\\
&\leq \|A_{\tau}\psi_{\varepsilon,\delta}\|_{L_t^2L^2(\Gamma)}\|J_\varepsilon \psi_{\varepsilon,\delta}-\psi\|_{L_t^2H^{2-\sigma}(\Gamma)}\|v_{0,\tau}\|_{\mathrm{L}^4(\Gamma)}\rightarrow 0.
\end{align*}
Combining all the convergence results, we conclude that $\widetilde{\mathbb{P}}$-a.s.
\begin{align*}
\int_{0}^{t}(\mathcal{K}(\psi_{\varepsilon,\delta})\nabla_{\tau}(J_\varepsilon \psi_{\varepsilon,\delta}),v_{0,\tau})_\Gamma-(\mathcal{K}(\psi)\nabla_{\tau}\psi, v_{0,\tau})_\Gamma ds\rightarrow 0,
\end{align*}
for any $v_0\in \mathrm{H}^1$.

Finally, we pass the limit in stochastic integral term. Using assumption ${\bf A}_2$, we obtain
$$\|h(u_{\varepsilon,\delta}, \nabla\phi_{\varepsilon,\delta})-h(u,\nabla\phi)\|_{L_2(\mathcal{H};L_x^2)}\leq C\|(u_{\varepsilon,\delta}-u, \nabla(\phi_{\varepsilon,\delta}-\phi))\|_{L^2_x},$$
together with Proposition \ref{pro3.1} yields that
$$\|h(u_{\varepsilon,\delta}, \nabla\phi_{\varepsilon,\delta})-h(u,\nabla\phi)\|_{L_2(\mathcal{H};L_x^2)}\rightarrow 0,~ \mbox{a.e.}~ (t, \omega)\in [0,T]\times \Omega.$$
Collecting this convergence, (\ref{3.39}) and the Vitali convergence theorem (see Theorem \ref{thm4.1}), we have
$$h(u_{\varepsilon,\delta}, \nabla\phi_{\varepsilon,\delta})\rightarrow h(u,\nabla\phi),~{\rm in}~ L^p_\omega L_t^2L_2(\mathcal{H};L_x^2),$$
which implies that the convergence holds in probability in $L_t^2L_2(\mathcal{H};L_x^2)$.
Applying Lemma \ref{lem4.6}, we infer that
$$\int_{0}^{t}h(u_{\varepsilon,\delta}, \nabla\phi_{\varepsilon,\delta})d \mathcal{W}_{\varepsilon,\delta}\rightarrow \int_{0}^{t}h(u, \nabla\phi)d \mathcal{W},$$
in probability in $L_t^2L_2(\mathcal{H};L_x^2)$.

Using \eqref{3.39} and H\"{o}lder's inequality, we could claim that three sets
\begin{align*}
\left\{\int_{0}^{t}-B_0(u_{\varepsilon, \delta},\mathcal{J}_\varepsilon u_{\varepsilon, \delta}, v_0)+B_1(\mu_{\varepsilon, \delta}, \mathcal{J}_{\varepsilon}\phi_{\varepsilon, \delta}, v_0)+(\mathcal{K}(\psi_{\varepsilon,\delta})\nabla_{\tau}(J_\varepsilon \psi_{\varepsilon,\delta}), v_{0,\tau})_{\Gamma}ds\right\} _{\varepsilon, \delta},
\end{align*}
and
\begin{align*}
\left\{\int_{0}^{t}-B_2(u_{\varepsilon, \delta}, \mathcal{J}_{\varepsilon}\phi_{\varepsilon, \delta}, v_1)ds\right\}_{\varepsilon,\delta},\left\{\int_{0}^{t} -B_\Gamma(u_{\varepsilon, \delta},J_{\varepsilon}(\psi_{\varepsilon,\delta}), v_2)-(\mathcal{K}(\psi_{\varepsilon,\delta}), v_2)_{\Gamma}ds\right\}_{\varepsilon,\delta},
\end{align*}
are uniformly integrable in $L^p(\widetilde{\Omega}; L^2_t)$.

We now have everything in hand to identify the limit. Define the functionals as
\begin{align*}
\mathbf{F}_1(u_{\varepsilon, \delta},\phi_{\varepsilon, \delta}, \psi_{\varepsilon,\delta}, \mathcal{W}_{\varepsilon,\delta},v_0)&=(u_{0}, v_0)-\int_{0}^{t}a(u_{\varepsilon, \delta},v_0) ds-\int_{0}^{t}B_0(u_{\varepsilon, \delta},\mathcal{J}_\varepsilon u_{\varepsilon, \delta}, v_0)ds\\
&\quad+\int_{0}^{t}B_1(\mu_{\varepsilon, \delta}, \mathcal{J}_{\varepsilon}\phi_{\varepsilon, \delta}, v_0)ds+\int_{0}^{t}(\mathcal{K}(\psi_{\varepsilon,\delta})\nabla_{\tau}(J_\varepsilon \psi_{\varepsilon,\delta}), v_{0,\tau})_{\Gamma}ds\\
&\quad+\left(\int_{0}^{t}h(u_{\varepsilon, \delta},\nabla\phi_{\varepsilon,\delta})d \mathcal{W}_{\varepsilon,\delta}, v_0\right),\\
\mathbf{F}_2(u_{\varepsilon, \delta},\phi_{\varepsilon, \delta}, v_1)&=(\phi_0, v_1)-\int_{0}^{t} (\nabla \mu_{\varepsilon,\delta}, \nabla v_1)+ B_2(u_{\varepsilon, \delta},\phi_{\varepsilon,\delta},v_1)ds,\\
\mathbf{F}_3(u_{\varepsilon, \delta},\psi_{\varepsilon, \delta}, v_2)&=(\psi_0, v_2)-\int_{0}^{t} B_\Gamma(u_{\varepsilon, \delta},J_{\varepsilon}(\psi_{\varepsilon,\delta}), v_2)+(\mathcal{K}(\psi_{\varepsilon,\delta}), v_2)_{\Gamma}ds,
\end{align*}
and
\begin{align*}
\mathbf{F}_1(u,\phi, \psi, \mathcal{W},v_0)&=(u_{0}, v_0)-\int_{0}^{t}a(u,v_0) ds-\int_{0}^{t}B_0(u, u, v_0)ds
+\int_{0}^{t}B_1(\mu, \phi, v_0)ds\\&\quad+\int_{0}^{t}(\mathcal{K}(\psi)\nabla_{\tau}\psi, v_{0,\tau})_{\Gamma}ds+\left(\int_{0}^{t}h(u,\nabla\phi)d \mathcal{W}, v_0\right),\\
\mathbf{F}_2(u,\phi, v_1)&=(\phi_0, v_1)-\int_{0}^{t} (\nabla \mu, \nabla v_1)+ B_2(u,\phi,v_1)ds,\\
\mathbf{F}_3(u,\psi, v_2)&=(\psi_0, v_2)-\int_{0}^{t} B_\Gamma(u,\psi, v_2)+(\mathcal{K}(\psi), v_2)_{\Gamma}ds.
\end{align*}

Combining all the convergence results and the uniformly integrability of sets, we could summarize from the Vitali convergence theorem that
\begin{align}
&\widetilde{\mathbb{E}}\int_{0}^{T}1_{A}(\mathbf{F}_1(u_{\varepsilon, \delta},\phi_{\varepsilon, \delta}, \psi_{\varepsilon,\delta}, \mathcal{W}_{\varepsilon,\delta},v_0)-\mathbf{F}_1(u,\phi, \psi, \mathcal{W},v_0))dt\rightarrow 0,\label{e.51}\\
&\widetilde{\mathbb{E}}\int_{0}^{T}1_{A}(\mathbf{F}_2(u_{\varepsilon, \delta},\phi_{\varepsilon, \delta}, v_1)-\mathbf{F}_2(u,\phi, v_1))dt\rightarrow 0,\label{e.52}\\
&\widetilde{\mathbb{E}}\int_{0}^{T}1_{A}(\mathbf{F}_3(u_{\varepsilon, \delta},\psi_{\varepsilon, \delta}, v_2)-\mathbf{F}_3(u,\psi, v_2))dt\rightarrow 0,\label{3.77}
\end{align}
for any set $A\subset [0,T]$, $v_0\in D(A_0), v_1\in H^1, v_2\in L^2(\Gamma)$.

Considering \eqref{e.51}-\eqref{3.77} and Proposition \ref{pro3.2}, we conclude that $(u,\phi,\psi)$ satisfy \eqref{2.1} of Definition \ref{def2.1},
$\widetilde{\mathbb{P}}\mbox{-a.s.}$ for any $v_0\in D(A_0), v_1\in H^1, v_2\in L^2(\Gamma), t\in [0,T]$. Furthermore, using a density argument, we deduce that (\ref{2.1}) holds for any $v_0\in \mathrm{H}^1, v_1\in H^1, v_2\in L^2(\Gamma)$.

Finally, we show that the energy inequality holds. By Proposition \ref{pro3.2}, the energy inequality \eqref{3.60} still holds for the new approximate sequence on the new probability space. In the limit procedures $\varepsilon, \delta\rightarrow 0$, using the lower semi-continuity on the left hand-side for any time and using the strong convergence of the approximate sequence for a.e. time, we could have for any $t>0$, there exists a nullset $\mathbb{T}_t$ such that
\begin{align}\label{3.81}\mathbb{E}\left[1_\mathcal{A}\mathcal{U}(u, \phi, \psi)(t)\right]\leq \mathbb{E}\left[1_\mathcal{A}\mathcal{U}(u, \phi, \psi)(s)\right]\end{align}
for all $s\in \mathbb{T}_t$ and $\mathcal{A}\in \mathcal{F}_s$, where
\begin{align*}
\mathcal{U}(u, \phi, \psi)(t):&=\frac{1}{2}E^2(t)+\int_{0}^{t}E\left(\|(\nabla u, \nabla \mu)\|_{L_x^2}^2+\|(\phi, \psi)\|^2_{V^2}+\|\mathcal{K}(\psi)\|_{L^2(\Gamma)}^2\right)dr\\
 &-\int_{0}^{t}E\|h(u,\nabla\phi)\|^2_{L_2(\mathcal{H};L_x^2)}dr
-\int_{0}^{t}(h(u,\nabla\phi),u)^2dr.
\end{align*}
Set $\mathrm{T}=\bigcup_{t\in\mathcal{C}}$, where the set $\mathcal{C}$ is countable and dense in $[0,\infty)$. Then for $s\notin \mathrm{T}$ and $s\leq t$, there exists a sequence $t_m\subset \mathcal{C}$ with $t_m\rightarrow t$, then letting $m\rightarrow\infty$ in \eqref{3.81}, the lower semi-continuity of mapping
$$t\mapsto E(t)$$
gives the energy inequality, for further details see \cite[Lemma A.3]{F2} and \cite{brei}. We complete the proof.

\section{Markov Selection}
\subsection{Preliminary}
Denote by $\Omega=C([0,\infty); \mathbb{X})$ where $\mathbb{X}$ is a Polish space and by $\widetilde{\mathcal{B}}$ the Borel $\sigma$-field of
$\Omega$ with Skorokhod topology and by ${\rm Pr}(\Omega)$ the set of all probability measures on $(\Omega, \widetilde{\mathcal{B}})$. For any $t>0$, let $\Omega_t=C([0,t]; \mathbb{X})$ and $\Omega^t=C([t,\infty); \mathbb{X})$, then denote by $\widetilde{\mathcal{B}}_t, \widetilde{\mathcal{B}}^t$ the Borel $\sigma$-field of $\Omega_t, \Omega^t$ with Skorokhod topology respectively. Define the canonical process $\varrho:\Omega\rightarrow \mathbb{X}$ as $\varrho_t(\omega)=\omega(t)$. Define by the shift map $\theta_t:\Omega\rightarrow \Omega^t$ as $$\theta_t(\omega)(s)=\omega(s-t), ~s\geq t,$$
establishes a  measurable isomorphism between $(\Omega, \mathcal{F}, \{\mathcal{F}_s\}_{s\geq 0})$ and $(\Omega^t, \mathcal{F}^t, \{\mathcal{F}^t_s\}_{s\geq t})$ where the filtration $\mathcal{F}^t_s:=\sigma\{\omega(r); t\leq r\leq s\}$ and $\mathcal{F}^t:=\vee_s \mathcal{F}_s^t$, we have the fact that $\mathcal{F}^t=\widetilde{\mathcal{B}}^t$, also establishes an isomorphism between ${\rm Pr}(\Omega)$ and ${\rm Pr}(\Omega^t)$, that is, if $\mathbb{P}\in {\rm Pr}(\Omega)$, then $\mathbb{P}\circ \theta_t^{-1}\in {\rm Pr}(\Omega^t)$, if $\mathbb{P}\in \Omega^t$, then $\mathbb{P}\circ\theta_t\in  {\rm Pr}(\Omega)$.

Given $\mathbb{P}\in {\rm Pr}(\Omega)$ and $t>0$, define by $\mathbb{P}(\cdot|\mathcal{F}_t)$ a regular conditional probability distribution of $\mathbb{P}$ with respect to $\mathcal{F}_t$. Also $\mathbb{P}(\cdot|\mathcal{F}_t)$ is a probability measure on $(\Omega, \mathcal{F})$ and for any bounded $\mathcal{F}$-measurable function $f$
$$\mathbb{E}^\mathbb{P}(f|\mathcal{F}_t)=\int_{\Omega}f(y)\mathbb{P}(dy|\mathcal{F}_t), \mathbb{P}\!-\!{\rm a.s.}$$
and there exists a $\mathbb{P}$-null set $\mathcal{N}\in \mathcal{F}$ such that for any $\omega\notin \mathcal{N}$
\begin{align}\label{4.1}\mathbb{P}(\theta_t=\omega(t),s\in[0,t]|\mathcal{F}_t)=1.\end{align}
Particularly, by \eqref{4.1}, we could consider $\mathbb{P}(\cdot|\mathcal{F}_t)$ as a measure on $(\Omega^t, \mathcal{F}^t)$, that is,
$$\mathbb{P}|^w_{\widetilde{\mathcal{B}}_t}:=\mathbb{P}(\cdot|\mathcal{F}_t)\in {\rm Pr}(\Omega^t).$$
And for any $A\in \widetilde{\mathcal{B}}_t$ and $B\in \widetilde{\mathcal{B}}^t$, it holds
$$\mathbb{P}(A\cap B)=\int_{A}\mathbb{P}|^w_{\widetilde{\mathcal{B}}_t}(B)\mathbb{P}(d\omega).$$

We recall the following reconstruction result which is an inverse procedure of above disintegration.

\begin{proposition}\label{pro4.1}{\rm \!\!\cite[Lemma 6.1.1]{stro}} For any $\mathbb{P}\in {\rm Pr}(\Omega)$ and $t>0$, $Q$ is a map from $\Omega\rightarrow {\rm Pr}(\Omega^t)$ such that for any $B\in \mathcal{F}^t$, $\omega\mapsto Q_\omega(B)$ is $\mathcal{F}_t$-measurable and
$$Q_\omega(\theta_t=\omega(t))=1,~ for~ all~ \omega\in \Omega.$$
Then there exists a unique probability measure $\mathbb{P}\otimes_t Q$ on $\Omega$ such that \\
{\rm (i)} $\mathbb{P}\otimes_t Q_\omega$ and $\mathbb{P}$ agree on $\mathcal{F}_t$;\\
{\rm (ii)} $Q_\omega$ is a regular conditional probability measure distribution of  $\mathbb{P}\otimes_t Q$ on $\mathcal{F}_t$.
 \end{proposition}

The following definition is a generalization of supermartingale, see \cite[Definition 3.2]{F2}.
 \begin{definition}
 Let $\mathcal{G}$ be $\widetilde{\mathcal{B}}_t$-adapted real-valued stochastic process, we call $\mathcal{G}$ is an almost surely $(\widetilde{\mathcal{B}}_t, \mathbb{P})$-supermartingale if
 \begin{align}\label{4.1*}\mathbb{E}^\mathbb{P}[1_\mathcal{A}\mathcal{G}_t]\leq \mathbb{E}^\mathbb{P}[1_\mathcal{A}\mathcal{G}_s],\end{align}
 for a.e. $s\geq 0$ and $s\leq t$, all $\mathcal{A}\in \widetilde{\mathcal{B}}_t$.
 \end{definition}
 We call time point $s$ by regular time of $\mathcal{G}$ if \eqref{4.1*} holds, and by exceptional time of $\mathcal{G}$ if \eqref{4.1*} fails.

 \begin{definition} Let $\widetilde{\mathcal{H}}$ be a Polish space, if there exists a set $B\in \mathcal{F}$ with $\mathbb{P}(B)=1$ such that $B\subset \{\omega\in \Omega: \omega(t)\in \widetilde{\mathcal{H}}, \forall t\geq 0\}$, then we say $\mathbb{P}\in {\rm Pr}_{\widetilde{\mathcal{H}}}(\Omega)\subset {\rm Pr}(\Omega)$ is concentrated on the paths with values in $\widetilde{\mathcal{H}}$.
 \end{definition}
\begin{definition}Let $\{\mathbb{P}_x\}_{x\in \widetilde{\mathcal{H}}}$ be a family of probability measures in  ${\rm Pr}_{\widetilde{\mathcal{H}}}(\Omega)$, which is $\mathcal{B}/\mathcal{B}[0,1]$-measurable, and if for each $x\in \widetilde{\mathcal{H}}$, there exists a Lebesgue null measure set $\mathcal{N}_x$ such that for all $t\notin \mathcal{N}_x$
$$\mathbb{P}_x(\cdot|\mathcal{F}_t)=\mathbb{P}_{\omega(t)}\circ \theta_t^{-1}, ~{\rm for} ~\mathbb{P}_x\!-\!{\rm a.s.}$$
We say the family of $\{\mathbb{P}_x\}_{x\in \widetilde{\mathcal{H}}}$ is an almost surely Markov family.
\end{definition}

 Denote by $Comp({\rm Pr}(\Omega))$ the space of all compact subsets of ${\rm Pr}(\Omega)$, which is a metric space after endowed with the Hausdorff metric.

\begin{definition}\label{def4.4} Consider a measurable map $\mathcal{C}: \widetilde{\mathcal{H}}\rightarrow Comp({\rm Pr}(\Omega))\cap {\rm Pr}_{\widetilde{\mathcal{H}}}(\Omega)$, the family $\{\mathcal{C}(x)\}_{x\in \widetilde{\mathcal{H}}(\Omega)}$ is almost surely pre-Markov if for each $x\in \widetilde{\mathcal{H}}$ and $\mathbb{P}\in \mathcal{C}(x)$, there exists a Lebesgue null measure set $\mathcal{N}_x$ such that for all $t\notin \mathcal{N}_x$, it holds\\
(i) (Disintegration) there exists $\mathbb{P}$-null set $\widetilde{\mathcal{N}}$ such that for all $\omega \notin\widetilde{\mathcal{N}}$
$$\omega(t)\in \widetilde{\mathcal{N}},~~\mathbb{P}(\theta_t(\cdot)|\mathcal{F}_t)\in \mathcal{C}(\omega(t));$$
(ii) (Reconstruction) for each measurable map $\omega\rightarrow Q_\omega:\Omega\rightarrow {\rm Pr}(\Omega^t)$ such that there exists a $\mathbb{P}$-null set $\widetilde{\mathcal{N}}\in \mathcal{F}_t$ such that all $\omega \notin\widetilde{\mathcal{N}}$
$$\omega(t)\in \widetilde{\mathcal{N}},~~ Q_\omega\circ \theta_t\in \mathcal{C}(\omega(t)),$$
then $\mathbb{P}\otimes_tQ\in \mathcal{C}(x)$.
\end{definition}

 The following abstract Markov selection result was given in \cite[Theorem 2.8]{F2}.
 \begin{proposition}\label{pro4.2}
 Let $\{\mathcal{C}(x)\}_{x\in \widetilde{\mathcal{H}}}$ be an almost surely pre-Markov family with non-empty convex values. Then, there exists a measurable map $x\mapsto \mathbb{P}_x\in {\rm Pr}_{\widetilde{\mathcal{H}}}(\Omega)$ such that $\mathbb{P}_x\in \mathcal{C}(x)$ for all $x\in \widetilde{\mathcal{H}}$ and $\{\mathbb{P}_x\}_{x\in \widetilde{\mathcal{H}}}$ has almost surely Markov property.
  \end{proposition}

\subsection{The solution to martingale problem of system \eqref{Equ1.1}-\eqref{1.6}}

In view of the result established in section 3, choosing $\mathbb{X}_1=\mathrm{H}^1\times V^2$ and $\mathbb{Y}=\mathbb{Y}_1\times \mathbb{Y}_2$, where $\mathbb{Y}_1=\mathrm{L}^2\times V^1, \mathbb{Y}_2= H^1\times L^2(\Gamma)$. Introducing by $\mathcal{V}_1$ and $\mathcal{V}_2$ two new variables
$$\mathcal{V}_1=\mathcal{V}_1(0)+\int_{0}^{t}\mu ds, ~~\mathcal{V}_2=\mathcal{V}_2(0)+\int_{0}^{t}\mathcal{K}ds,$$
for arbitrary initial data $\mathcal{V}_1(0)\in H^1, \mathcal{V}_2(0)\in L^2(\Gamma)$.
From the regularity of $\mu, \mathcal{K}$, we easily see $\mathcal{V}_1$ and $\mathcal{V}_2$ are two continuous stochastic processes with trajectories in $W^{1,2}_tH^1$ and $W^{1,2}_tL^2(\Gamma)$, $\mathbb{P}$-a.s. respectively.

Denote by $\widetilde{\Omega}= C([0,\infty); \mathbb{X})$ with $\mathbb{X}=(D(A_0)\times V^2\times H^1\times L^2(\Gamma))'$. Let $$\varrho=(\varrho^1, \varrho^2, \varrho^3, \varrho^4, \varrho^5)$$ be the canonical process with
$$\varrho_t(\omega)=(\varrho^1, \varrho^2, \varrho^3, \varrho^4, \varrho^5)(\omega)=\omega(t)\in \mathbb{X}.$$
For Markov selection, we re-define the solution to martingale problem \eqref{Equ1.1}-\eqref{1.6}.
\begin{definition}\label{def4.5}
Given $y\in \mathbb{Y}$, a probability measure $\mathcal{P}$ on $(\widetilde{\Omega}, \widetilde{\mathcal{B}})$ is a solution to martingale problem \eqref{Equ1.1}-\eqref{1.6} if \\
(M.1) $\mathcal{P}\left(L_{loc}^\infty([0,\infty);\mathbb{Y}_1)\cap L_{loc}^2([0,\infty);\mathbb{X}_1)\right)=1$, $\mathcal{P}\left(W_{loc}^{1,2}([0,\infty);\mathbb{Y}_2)\right)=1;$\\
(M.2) for every $v_0\in \mathrm{H}^1$, the process
\begin{align*}{\bf M}_{v_0}^t&=(\varrho_t^1-\varrho^1_0,v_0)+\int_{0}^{t}a_0(\varrho^1,v_0)-B_0(\varrho^1,v_0,\varrho^1)-B_1(\partial_t\varrho^4, \varrho^2, v_0)ds\nonumber\\&\quad-\int_{0}^{t}(\partial_t\varrho^5\nabla_{\tau}\varrho^3, v_{0,\tau})_\Gamma ds,\end{align*}
is a continuous square integrable martingale with respect to $\mathcal{P}$, its quadratic variation process is given by
$$\langle{\bf M}_{v_0}^t\rangle:=\frac{1}{2}\int_{0}^{t}(h(\varrho^1, \nabla\varrho^2), v_0)^2ds;$$
(M.3) for every $v_1\in H^1$ and $t\geq 0$, it holds $\mathcal{P}$-a.s.
\begin{align*}
(\varrho^2_t, v_1)+\int_{0}^{t}(\nabla\partial_t\varrho^4, \nabla v_1)-B_2(\varrho^1,v_1,\varrho^2)ds=(\varrho^2_0, v_1);
\end{align*}
(M.4) for every $v_2\in H^1$ and $t\geq 0$, it holds $\mathcal{P}$-a.s.
\begin{align*}
(\varrho^3_t, v_2)+\int_{0}^{t}(\partial_t\varrho^5, v_2)-B_\Gamma(\varrho^1,v_2,\varrho^3)ds=(\varrho^3_0, v_1);
\end{align*}
(M.5) the process
\begin{align}
\mathcal{G}: [\omega, t]\longmapsto &\frac{1}{2}\mathcal{E}_t^2+\int_{0}^{t}\mathcal{E}_t\left(\|\nabla \varrho^1, \nabla \partial_t\varrho^4\|_{L_x^2}^2+\|(\varrho^2, \varrho^3)\|^2_{V^2}+\|\partial_t\varrho^5\|^2_{L^2(\Gamma)}\right)ds-\frac{1}{2}\mathcal{E}_0^2\nonumber\\&-\int_{0}^{t}\mathcal{E}_t\|h(\varrho^1,\nabla\varrho^2)\|^2_{L_2(\mathcal{H};L_x^2)}ds\nonumber\\
&-\int_{0}^{t}(h(\varrho^1,\nabla\varrho^2),\varrho^2)^2ds,
\end{align}
is an almost surely $(\widetilde{\mathcal{B}}_t, \mathcal{P})$-supermartingale and $0$ is a regular time, where
$$\mathcal{E}_t:=\frac{1}{2}\left(\|\varrho^1\|_{\mathrm{L}_x^2}^2+\|\nabla \varrho^2\|_{L_x^2}^2+\|\varrho^3\|_{L^2(\Gamma)}^2+\|\nabla_{\tau}\varrho^3\|_{L^2(\Gamma)}^2\right)+\int_{\mathcal{D}}F(\varrho^2)dx+\int_{\Gamma}G(\varrho^3)dS;$$
(M.6) it holds $\mathcal{P}$-a.s.
$$\partial_t\varrho^4=-\Delta\varrho^2+f(\varrho^2), ~\partial_t\varrho^5=A_{\tau}\varrho^3+\partial_n\varrho^2+\varrho^3+g(\varrho^3).$$

\end{definition}

\begin{proposition}\label{pro4.3} Definition \ref{def2.1} and Definition \ref{def4.5} are equivalent in the following sense:\\
{\rm (i)} if $(\mathcal{S}, u,\phi, \psi)$ is a global martingale weak solution to system (\ref{Equ1.1})-(\ref{1.6}), then for every $H^1, L^2(\Gamma)$-valued $\mathcal{F}_0$ measurable random variables $\mathcal{V}_1(0), \mathcal{V}_2(0)$, we could have a probability measure $\mathcal{P}=\mathcal{L}( u,\phi, \psi, \mathcal{V}_1, \mathcal{V}_2)$ is a solution to the martingale problem associated to \eqref{Equ1.1}-\eqref{1.6};\\
{\rm (ii)} if a probability measure $\mathcal{P}$ is a solution to the martingale problem associated to \eqref{Equ1.1}-\eqref{1.6}, then there exists a global martingale weak solution $(\mathcal{S}, u,\phi, \psi)$ to system (\ref{Equ1.1})-(\ref{1.6}) and $H^1, L^2(\Gamma)$-valued $\mathcal{F}_0$ measurable random variables $\mathcal{V}_1(0)$ and $\mathcal{V}_2(0)$ such that $\mathcal{L}( u,\phi, \psi, \mathcal{V}_1, \mathcal{V}_2)=\mathcal{P}$.
\end{proposition}
\begin{proof} The proof is totally same with that of \cite[Proposition 3.8]{brei}, we omit the details.
\end{proof}

\subsection{Markov selection}

In the following, we are going to establish the existence of almost surely Markov selection to martingale problem \eqref{Equ1.1}-\eqref{1.6}. Denote by $\mathcal{P}_y$ be the solution to the martingale problem \eqref{Equ1.1}-\eqref{1.6} starting from initial data $\delta_y$.
\begin{theorem}\label{thm4.1*} Suppose that assumptions ${\bf A_1}-{\bf A_4}$ hold, then there exists an almost surely Markov family $\mathcal{P}_y$ associated to \eqref{Equ1.1}-\eqref{1.6}.
\end{theorem}
Denote by $\{\mathcal{C}(y)\}_{y\in \mathbb{Y}}$ be the set of probability measure $\mathcal{P}_y$. From Proposition \ref{pro4.2}, the proof of above theorem boils down to show the family $\{\mathcal{C}(y)\}_{y\in \mathbb{Y}}\subset {\rm Pr}(\widetilde{\Omega})$ is an almost surely pre-Markov family.

\begin{lemma}
For each $y\in \mathbb{Y}$, the set $\mathcal{C}(y)$ is non-empty and convex.
\end{lemma}
\begin{proof} Due to Theorem \ref{thm2.1} and Proposition \ref{pro4.3}, for each $y\in \mathbb{Y}$, a solution to martingale problem exists making sure the set is non-empty. The convexity follows immediately all properties in Definition \ref{def4.5} involving integration with respect to measure $\mathcal{P}$.
\end{proof}
\begin{lemma} For each $y\in \mathbb{Y}$, the set $\mathcal{C}(y)$ is compact and the map $\mathcal{C}:\mathbb{Y}\rightarrow Comp({\rm Pr}(\widetilde{\Omega}))$ is Borel measurable.
\end{lemma}
\begin{proof} We claim that both properties are direct consequence of the following argument: for the initial data sequence $y_n\in \mathbb{Y}$ which converges to $y$ and $\mathcal{P}_n\in \mathcal{C}(y_n)$, then there exists a subsequence of $\mathcal{P}_n$ which converges to $\mathcal{P}\in \mathcal{C}(y)$ weakly in ${\rm Pr}(\widetilde{\Omega})$.

From Theorem \ref{thm2.1}, there exists a sequence martingale weak solutions $(\mathcal{S}^n, u^n,\phi^n, \psi^n)$ in the sense of Definition \ref{def2.1}. By a same argument as Parts 2,3 in section 3, we also have that the sequence $(u^n,\phi^n, \psi^n, \mu^n, \mathcal{K}(\psi^n))$ has uniform priori estimates and the probability measures $\mathcal{P}^n$ induced by the law of sequence $(u^n,\phi^n, \psi^n, \mathcal{V}_1^n, \mathcal{V}_2^n)$ is tight on path space $Y\times (W^{1,2}_tH^1)_w\times(W^{1,2}_tL^2(\Gamma))_w$, where $Y$ is the space defined in Part 3. The Prokhorov theorem tells us there exist a subsequence $\mathcal{P}^{n_k}$ and a probability measure $\mathcal{P}$ such that $\mathcal{P}^{n_k}$ converges weakly to $\mathcal{P}\in {\rm Pr}(\widetilde{\Omega})$. Next we show that $\mathcal{P}\in \mathcal{C}(y)$.

Again, the Skorokhod representative theorem \ref{thm4.2} gives that there exist a new probability space $(\overline{\Omega}, \overline{\mathcal{F}}, \overline{\mathcal{P}})$,  a new subsequence still denoted by $(\bar{u}^n,\bar{\phi}^n, \bar{\psi}^n, \overline{\mathcal{V}}^n_1, \overline{\mathcal{V}}^n_2, \overline{\mathcal{W}}^n)$ converges to $(\bar{u},\bar{\phi}, \bar{\psi},  \overline{\mathcal{V}}_1, \overline{\mathcal{V}}_2, \overline{\mathcal{W}})$ in $Y\times (W^{1,2}_tH^1)_w\times(W^{1,2}_tL^2(\Gamma))_w$, $\overline{\mathcal{P}}$-a.s. and $\mathcal{L}(\bar{u}^n,\bar{\phi}^n, \bar{\psi}^n, \overline{\mathcal{V}}^n_1, \overline{\mathcal{V}}^n_2)=\mathcal{P}^n$.  Following the same line of Parts 3,4 of section 3, we have that $(\mathcal{S}, \bar{u}, \bar{\phi}, \bar{\psi})$ is a global martingale weak solution to system \eqref{Equ1.1}-\eqref{1.6} in the sense of Definition \ref{def2.1}. Furthermore, we have $\mathcal{P}=\mathcal{L}(\bar{u},\bar{\phi}, \bar{\psi}, \overline{\mathcal{V}}_1, \overline{\mathcal{V}}_2 )$. Finally, Proposition \ref{pro4.3} yields that $\mathcal{P}$ is a solution to the martingale problem associated to system \eqref{Equ1.1}-\eqref{1.6} with initial law $\delta_y$. The fact $\mathcal{P}\in \mathcal{C}(y)$ verified.
\end{proof}

 The following two results are important for next proof of disintegration and reconstruction properties, which were given in \cite[Proposition B.1, Proposition B.4]{F2}.
\begin{proposition}\label{pro4.4} Let $\mathcal{Y}$ and $\mathcal{Q}$ be two real-valued continuous $\mathcal{B}_t$-adapted processes and $s\geq 0$, the followings are equivalent:\\
{\rm (i)} $(\mathcal{Y}_t, \mathcal{B}_t, \mathcal{P})$ is a square integrable martingale with quadratic variation $\mathcal{Q}_t$ for $t\geq s$;\\
{\rm (ii)} there is a $\mathcal{P}$-nullset $\mathcal{N}\in \mathcal{B}_s$ such that for all $\omega\notin\mathcal{N}$, process $(\mathcal{Y}_t, \mathcal{B}_t, \mathcal{P}|^\omega_{\mathcal{B}_t})$ is a square integrable martingale with quadratic variation $\mathcal{Q}_t$ for $t\geq s$ and $\mathbb{E}^\mathcal{P}\left(\mathbb{E}^{\mathcal{P}|^{\cdot}_{\mathcal{B}_t}}[\mathcal{Q}_t]\right)<\infty$.
\end{proposition}
\begin{proposition}\label{pro4.5} Let $\alpha, \beta$ be two real-valued adapted processes such that $\beta$ is non-decreasing and $\mathcal{Y}=\alpha-\beta$ is left lower semi-continuous, the followings are equivalent:\\
{\rm (i)} for $s\geq 0$, $(\mathcal{Y}_t, \mathcal{B}_t, \mathcal{P})$ is an almost surely super-martingale and for all $t\geq s$
$$\mathbb{E}^{\mathcal{P}}(\alpha_t+\beta_t)<\infty;$$
{\rm(ii)} there is a $\mathcal{P}$-nullset $\mathcal{N}\in \mathcal{B}_s$ such that for all $\omega\notin\mathcal{N}$, process $(\mathcal{Y}_t, \mathcal{B}_t, \mathcal{P}|^\omega_{\mathcal{B}_s})$ is an almost surely super-martingale and for all $t\geq s$
$$\mathbb{E}^{\mathcal{P}|^\omega_{\mathcal{B}_s}}(\alpha_t+\beta_t)<\infty,$$
and
$$\mathbb{E}^\mathcal{P}\left(\mathbb{E}^{\mathcal{P}|^{\cdot}_{\mathcal{B}_s}}[\alpha_t+\beta_t]\right)<\infty.$$
\end{proposition}

Next, we show the family $\{\mathcal{C}(y)\}_{y\in \mathbb{Y}}$ satisfies the disintegration and reconstruction properties. Once both properties hold, Theorem \ref{thm4.1*} follows from Proposition \ref{pro4.2}.

\begin{lemma}\label{lem4.3*} The disintegration property of Definition \ref{def4.4} holds for the family $\{\mathcal{C}(y)\}_{y\in \mathbb{Y}}$.
\end{lemma}
\begin{proof} The proof is similar to \cite[Lemma 4.4]{F2}, we still give the details for reader's convenience. For $y\in \mathbb{Y}$ and $\mathcal{P}\in \mathcal{C}(y)$. Let $\mathbb{T}_{\mathcal{P}}$ be the set of exceptional times of $\mathcal{P}$ and for $t\notin \mathbb{T}_{\mathcal{P}}$, denote by $\mathcal{P}|^\omega_{\widetilde{\mathcal{B}}_t}$ the regular conditional probability distribution of $\mathcal{P}$ on $\widetilde{\mathcal{B}}_t$. Our goal is to show that there exists a $\mathcal{P}$-nullset $\mathcal{N}\in \widetilde{\mathcal{B}}_t$ such that for all $\omega\notin \mathcal{N}$, $\mathcal{P}|^\omega_{\widetilde{\mathcal{B}}_t}\circ \theta_t\in \mathcal{C}(y)$. That is, we need to verify ${\rm (M.1)-(M.5)}$ hold for $\mathcal{P}|^\omega_{\widetilde{\mathcal{B}}_t}\circ \theta_t$.

Define by the sets
\begin{align*}
&\overline{\sum}_t=\left\{\omega\in \widetilde{\Omega}: \omega|_{[0,t]}\in L^\infty([0,t);\mathbb{Y}_1)\cap L^2([0,t);\mathbb{X}_1)\times W^{1,2}([0,t);\mathbb{Y}_2)\right\},\\
&\overline{\sum}^t=\left\{\omega\in \widetilde{\Omega}: \omega|_{[t, \infty]}\in L_{loc}^\infty([t,\infty);\mathbb{Y}_1)\cap L_{loc}^2([t,\infty);\mathbb{X}_1)\times W_{loc}^{1,2}([t,\infty);\mathbb{Y}_2)\right\}.
\end{align*}
We have by ${\rm (M.1)}$ for $\mathcal{P}$ that
$$1=\mathcal{P}\left(\overline{\sum}_t\cap \overline{\sum}^t\right)=\int_{\overline{\Sigma}_t}\mathcal{P}|^\omega_{\widetilde{\mathcal{B}}_t}\left(\overline{\sum}^t\right)\mathcal{P}(d\omega),$$
following there exists a $\mathcal{P}$-nullset $\mathcal{N}_1\in \widetilde{\mathcal{B}}_t$ such that for all $\omega\notin \mathcal{N}_1 $, $\mathcal{P}|^\omega_{\widetilde{\mathcal{B}}_t}\left(\overline{\sum}^t\right)=1$.

Let $\{v_n\}_{n\geq 1}$ be the dense subset of $D(A_0)$, then we have by ${\rm (M.2)}$ for $\mathcal{P}$ that $({\bf M}_{v_n}^t, \widetilde{\mathcal{B}}_t, \mathcal{P})$ is a square integrable martingale with quadratic variation
$$\langle{\bf M}_{v_n}^t\rangle=\frac{1}{2}\int_{0}^{t}(h(\varrho^1, \nabla\varrho^2), v_n)^2ds.$$
Proposition \ref{pro4.4} gives that there exists a $\mathcal{P}$-nullset $\mathcal{N}_2^n$ such that $\left({\bf M}_{v_n}^{t'}, \widetilde{\mathcal{B}}_{t'}, \mathcal{P}|^\omega_{\widetilde{\mathcal{B}}_{t'}}\right)_{t'\geq t}$ is a square integrable martingale with quadratic variation $\langle{\bf M}_{v_n}^{t'}\rangle_{t'\geq t}$ for all $\omega\notin \mathcal{N}_2^n$. Choosing $\mathcal{N}_2=\bigcup\mathcal{N}_2^n$.

Denote by $\widetilde{\sum}^t, \widetilde{\sum}_t$
$$\widetilde{\sum}_t:=\left\{\omega\in \widetilde{\Omega}; (\varrho^2_t, v_1)\big|_{r=0}^{r=s}+\int_{0}^{s}(\nabla\partial_t\varrho^4, \nabla v_1)- B_2(\varrho^1,v_1,\varrho^2)dr=0, s\leq t\right\},$$
and
$$\widetilde{\sum}^t:=\left\{\omega\in \widetilde{\Omega}; (\varrho^2_t, v_1)\big|_{r=t}^{r=s}+\int_{t}^{s}(\nabla\partial_t\varrho^4, \nabla v_1)- B_2(\varrho^1,v_1,\varrho^2)dr=0, s\geq t\right\}.$$
Then ${\rm (M.3)}$ and ${\rm (M.4)}$ hold for $\mathcal{P}|^\omega_{\widetilde{\mathcal{B}}_t}\circ \theta_t$ a.s. $\omega\notin \mathcal{N}_3\cup\mathcal{N}_4$ by a same argument as above.

Regarding ${\rm (M.5)}$, we set
$$\alpha_t=\frac{1}{2}\mathcal{E}_t^2+\int_{0}^{t}\mathcal{E}_t(\|\nabla \varrho^1, \nabla \partial_t\varrho^4\|_{L_x^2}^2+\|(\varrho^2, \varrho^3)\|^2_{V^2}+\|\partial_t\varrho^5\|^2_{L^2(\Gamma)})ds$$
and
$$\beta_t=\frac{1}{2}\mathcal{E}_0^2+\int_{0}^{t}\mathcal{E}_t\|h(\varrho^1,\nabla\varrho^2)\|^2_{L_2(\mathcal{H};L_x^2)}ds
+\int_{0}^{t}(h(\varrho^1,\nabla\varrho^2),\varrho^2)^2ds.$$
Note that $\alpha_t$ is left lower semi-continuous and $\beta_t$ is non-decreasing, so that $\mathcal{G}_t$ is also left lower semi-continuous. Since we have ${\rm (M.5)}$ for $\mathcal{P}$ and by Proposition \ref{pro4.5}, there exists a $\mathcal{P}$-nullset $\mathcal{N}_5$, for all $\omega\notin \mathcal{N}_5$ the result holds.

By same argument as ${\rm (M.1)}$, there exists a $\mathcal{P}$-nullset $\mathcal{N}_6\in \widetilde{\mathcal{B}}_t$ such that for all $\omega\notin \mathcal{N}_6$, $({\rm M.6})$ holds for $\mathcal{P}|^\omega_{\widetilde{\mathcal{B}}_t}$, a.s.

Finally, letting $\mathcal{N}=\mathcal{N}_1\cup\mathcal{N}_2\cup\mathcal{N}_3\cup\mathcal{N}_4\cup\mathcal{N}_5\cup\mathcal{N}_6$, we complete the proof.
\end{proof}
\begin{lemma} The reconstruction property of Definition \ref{def4.4} holds for the family $\{\mathcal{C}(y)\}_{y\in \mathbb{Y}}$.
\end{lemma}
\begin{proof} Let $y\in \mathbb{Y}$ and $\mathcal{P}\in \mathcal{C}(y)$. Let $\mathbb{T}_{\mathcal{P}}$ be the set of exceptional times of $\mathcal{P}$. We need to prove that for a $\widetilde{\mathcal{B}}_t$-measurable map $Q_\omega:\omega\longmapsto {\rm Pr}(\Omega^t)$, there exists a $\mathcal{P}$-nullset $\mathcal{N}_Q\in \widetilde{\mathcal{B}}_t$ such that for all $\omega\in \mathcal{N}_Q$, $Q_\omega\circ \theta_t\in \mathcal{C}(\omega(t))$, then we have the probability measure $\mathcal{P}\otimes_tQ$ given by Proposition \ref{pro4.1}, belongs to $\mathcal{C}(y)$. That is, ${\rm (M.1)-(M.5)}$ need to be verified for $\mathcal{P}\otimes_tQ$.

First, we have ${\rm (M.1)}$ holds for $Q_\omega$, therefore $Q_\omega\left(\overline{\sum}^t\right)=1$, which yields
$$\mathcal{P}\otimes_tQ\left(\overline{\sum}_t\cap\overline{\sum}^t\right)=\int_{\overline{\sum}_t}Q_\omega\left(\overline{\sum}^t\right)\mathcal{P}(d\omega)
=\mathcal{P}\left(\overline{\sum}_t\right)=1,$$
where sets $\overline{\sum}_t, \overline{\sum}^t$ are same with that of in Lemma \ref{lem4.3*}.

Since ${\rm (M.2)}$ holds for $Q_\omega$, so $\langle{\bf M}_{v_0}^t\rangle_{t\geq s}$ is a $((\widetilde{\mathcal{B}}_t)_{t\geq s}, Q_\omega)$ square integrable martingale, together with the fact that $Q_\omega$ is a regular conditional probability distribution of $\mathcal{P}\otimes_sQ$ on $\widetilde{\mathcal{B}}_s$, we have $\langle{\bf M}_{v_0}^t\rangle_{t\geq s}$ is a $((\widetilde{\mathcal{B}}_t)_{t\geq s}, \mathcal{P}\otimes_sQ)$ square integrable martingale by Proposition \ref{pro4.4}. Since $\mathcal{P}$ and $\mathcal{P}\otimes_tQ$ coincide on $\widetilde{\mathcal{B}}_t$ and $({\bf M}_{v_0}^t, \widetilde{\mathcal{B}}_t, \mathcal{P})_{t\leq s}$ is a square integrable martingale, we conclude that $({\bf M}_{v_0}^s, \widetilde{\mathcal{B}}_s,  \mathcal{P}\otimes_tQ)_{s\geq 0}$ is a square integrable martingale.

Since $Q_\omega\left(\widetilde{\sum}^t\right)=1$ by $({\rm M.3})$ for $Q_\omega$, we have
$$\mathcal{P}\otimes_tQ\left(\widetilde{\sum}_t\cap \widetilde{\sum}^t\right)=\int_{\widetilde{\sum}_t}Q_\omega\left(\widetilde{\sum}^t\right)d\mathcal{P}(\omega)=\mathcal{P}\left(\widetilde{\sum}_t\right)=1.$$
Also, $({\rm M.4})$ holds for $\mathcal{P}\otimes_tQ$ by a same argument.

Since $({\rm M.5})$ holds for $Q_\omega$, $(\mathcal{G}_t, \widetilde{\mathcal{B}}_t, Q_\omega)_{t\geq s}$ is an almost surely supermartingale. Proposition \ref{pro4.5} gives that $(\mathcal{G}_t, \widetilde{\mathcal{B}}_t, \mathcal{P}\otimes_sQ)_{t\geq s}$ is an almost surely supermartingale. Since $\mathcal{P}$ and $\mathcal{P}\otimes_tQ$ coincide on $\widetilde{\mathcal{B}}_t$ and $(\mathcal{G}_t, \widetilde{\mathcal{B}}_t, \mathcal{P})_{t\leq s}$ is an almost surely supermartingale, we obtain that $(\mathcal{G}_t, \widetilde{\mathcal{B}}_t, \mathcal{P}\otimes_sQ)_{t\geq 0}$ is an almost surely supermartingale.

As the proof of $({\rm M.1})$, we have
\begin{align*}
&\mathcal{P}\otimes_tQ(\partial_r\varrho^4=-\Delta\varrho^2+f(\varrho^2))\nonumber\\
&=\mathcal{P}\otimes_tQ(\partial_r\varrho^4|_{0}^t=-\Delta\varrho^2+f(\varrho^2)|_{0}^t, \partial_r\varrho^4|_{t}^\infty=-\Delta\varrho^2+f(\varrho^2)|_{t}^\infty)\nonumber\\
&=\int_{\left\{\partial_r\varrho^4|_{0}^t=-\Delta\varrho^2+f(\varrho^2)|_{0}^t\right\}}Q_\omega(\partial_r\varrho^4|_{t}^\infty=-\Delta\varrho^2+f(\varrho^2)|_{t}^\infty)d\mathcal{P}(\omega)\nonumber\\
&=\int_{\left\{\partial_r\varrho^4|_{0}^t=-\Delta\varrho^2+f(\varrho^2)|_{0}^t\right\}}d\mathcal{P}(\omega)=1.
\end{align*}
Similarly, we also have $\mathcal{P}\otimes_tQ(\partial_t\varrho^5=A_{\tau}\varrho^3+\partial_n\varrho^2+\varrho^3+g(\varrho^3))=1$.

This completes the proof.
\end{proof}

\section{Appendix}
At the beginning, we introduce some lemmas used frequently for improving the regularity.
\begin{lemma}{\rm\!\!\cite[Theorem 1.1.1]{bergh}}\label{lem4.1} Let $\mathcal{L}: L^{p_1}(0,T)\rightarrow L^{p_2}(\mathcal{D})$ and $L^{q_1}(0,T)\rightarrow L^{q_2}(\mathcal{D})$ be a linear operator with $q_1>p_1$ and $q_2<p_2$. Then, for any $s\in (0,1)$, the operator $\mathcal{L}: L^{r_1}(0,T)\rightarrow L^{r_2}(\mathcal{D})$, where $r_1=\frac{1}{s/p_1+(1-s)/q_1}$, $r_2=\frac{1}{s/p_2+(1-s)/q_2}$.
\end{lemma}

\begin{lemma}{\rm \!\!\cite[Lemmas 4.2, 4.3]{GMA}}\label{lem4.2} For all $\kappa\in (0,1)$, there exists $\alpha>0$ such that
$$\|u\|_{L^2(\Gamma)}^2\leq \kappa\|\nabla u\|_{L_x^2}^2+\kappa^{-\alpha}\|u\|_{L_x^2}^2,$$
as a consequence, we have
$$\|u_\tau\|_{L^2(\Gamma)}^2\leq \kappa\|\nabla u\|_{L_x^2}^2+\kappa^{-\alpha}\|u\|_{L_x^2}^2.$$
\end{lemma}

\begin{lemma}{\rm \!\!\cite[Lemma 3]{simon}}\label{lem4.3} Let $r^*=\left\{\begin{array}{ll}
\!\!\frac{dr}{d-r}, {\rm if } ~r<d,\\
\!\!  {\rm any ~finite~nonnegetive~ real~ number,~if}~r=d,\\
\!\! \infty, {\rm if } ~r>d,
\end{array}\right.$ where $d=2,3$ is the dimension.  For $1\leq r\leq \infty, 1\leq s\leq \infty$, if $\frac{1}{r}+\frac{1}{s}\leq 1$, and $\frac{1}{r^*}+\frac{1}{s}=\frac{1}{t}$, $f\in H^{1,r}$ and $g\in H^{-1,s}$, then $fg\in H^{-1,t}$, that is,
$$\|fg\|_{H^{-1,t}}\leq \|f\|_{H^{1,r}}\|g\|_{H^{-1,s}}.$$
\end{lemma}

\begin{lemma}{\rm \!\!\cite[Lemma A.1]{MS}} \label{4.4*} Let $(\phi,\psi)$ be the solution to the following elliptic boundary problem
\begin{align*}
\left\{\begin{array}{ll}
\!\!-\Delta \phi=l_1, ~{\rm in}~\mathcal{D},\\
\!\!-\Delta_\tau \psi+\partial_n\phi +\psi=l_2, ~{\rm on}~\Gamma.
\end{array}\right.
\end{align*}
If $(l_1,l_2)\in L_x^2\times L^2(\Gamma)$, then, it holds for some constant $C$
$$\|(\phi,\psi)\|_{H^2\times H^2(\Gamma)}\leq C\|(l_1,l_2)\|_{L_x^2\times L^2(\Gamma)}.$$
\end{lemma}

\begin{lemma}{\rm \!\!\cite{Tay}}\label{lem4.4}
Let $\mathcal{D}$ be a sufficient smooth bounded domain and $s\geq 0$. For every $\epsilon>0$, the mollifier operator $\mathcal{J}_{\epsilon}$ maps $H^{s}(\mathcal{D})$ into $H^{s'}(\mathcal{D})$ where $s'>s$ and has the following properties,\\
$(1)$ The collection $\{\mathcal{J}_{\epsilon}\}_{\epsilon >0}$ is uniformly bounded in $H^{s}(\mathcal{D})$ independence of $\epsilon$, i.e. there exists a positive constant C=C(s) such that,
$$\|\mathcal{J}_{\epsilon}f\|_{H^s}\leq C\|f\|_{H^s},~~~f\in H^{s}(\mathcal{D}).$$
$(2)$ For every $\epsilon>0$, if $s\geq 1$  then for $f\in H^{s}(\mathcal{D})$,
\begin{equation*}
\|\mathcal{J}_{\epsilon}f\|_{H^{s^\prime}}\leq \frac{C}{\epsilon^{s'-s}}\|f\|_{H^{s}}.
\end{equation*}
$(3)$ Sequence $\mathcal{J}_{\epsilon}f$ converges to $f$, for $f\in H^{s}(\mathcal{D})$, that is,
\begin{equation*}
\lim_{\epsilon\rightarrow 0}\|\mathcal{J}_{\epsilon}f-f\|_{H^s}=0.
\end{equation*}
\end{lemma}
\begin{remark}\label{rem4.1} The mollifier operator $\mathcal{J}_{\epsilon}$ could be defined as $\mathcal{J}_{\epsilon}u=\mathcal{R}\widetilde{\mathcal{J}_{\epsilon}}\widetilde{E}u$, where $\widetilde{\mathcal{J}_{\epsilon}}$ is a standard Friedrich's mollifier on $\widetilde{\mathcal{D}}$. $\widetilde{\mathcal{D}}$ is the un-boundary extend of smooth domain $\mathcal{D}$, and here $\widetilde{E}$ is an extension operator from $H^s(\mathcal{D})$ into $H^s(\widetilde{\mathcal{D}})$, $\mathcal{R}$ is a restriction operator from $H^s(\widetilde{\mathcal{D}})$ into $H^s(\mathcal{D})$, see \cite[Chapter 5]{Evans}.
\end{remark}

In order to establish the tightness of a family of solutions, the following Aubin-Lions lemma and Skorokhod representative theorem are commonly used. For more backgrounds we recommend Corollary 5 of \cite{Simon} and Theorem 2.4 of \cite{Zabczyk}. The Vitali convergence theorem is applied to identify the limit, see for example, Chapter 3 of \cite{Kallenberg}.

\begin{lemma}{\rm \!\!(The Aubin-Lions lemma)}\label{lem4.5} Suppose that $X_{1}\subset X_{0}\subset X_{2}$ are Banach spaces and $X_{1}$ and $X_{2}$ are reflexive and the embedding of $X_{1}$ into $X_{0}$ is compact.
Then for any $1\leq p<\infty,~ 0<\alpha<1$, the embedding
\begin{equation*}
L^{p}_tX_{1}\cap C^{\alpha}_tX_{2}\hookrightarrow L^{p}_tX_{0},
\end{equation*}
is compact.
\end{lemma}

\begin{theorem}{\rm \!\!(The Vitali convergence theorem)}\label{thm4.1} Let $p\geq 1$, $\{X_n\}_{n\geq 1}\in L^p$ and $X_n\rightarrow X$ in probability. Then, the following are equivalent:\\
{\rm (1)} $X_n\rightarrow X$ in $L^p$;\\
{\rm (2)} the sequence $|X_n|^p$ is uniformly integrable;\\
{\rm (3)} $\mathbb{E}|X_n|^p\rightarrow \mathbb{E}|X|^p$.
\end{theorem}

\begin{theorem}{\rm \!\!(The Skorokhod representative theorem)}\label{thm4.2} Let $X$ be a topological space. If the set of probability measures $\{\nu_n\}_{n\geq 1}$ on $\mathcal{B}(X)$ is tight, then there exist a probability space $(\Omega, \mathcal{F}, \mathbb{P})$ and a sequence of random variables $u_n, u$ such that theirs laws are $\nu_n$, $\nu$ and $u_n\rightarrow u$, $\mathbb{P}$-a.s. as $n\rightarrow \infty$.
\end{theorem}

We need the following result taking from \cite[Lemma 2.1]{ANR} for passing the limit of the sequence of stochastic integrals,
\begin{lemma}\label{lem4.6} Assume that $G_\varepsilon$ is a sequence of $X$-valued $\mathcal{F}_{t}^\varepsilon$ predictable process such that
$$G_\varepsilon\rightarrow G,~ {\rm in~ probability~ in}~L_t^2L_2(\mathcal{H};X),$$
and the cylindrical Wiener process sequence $\mathcal{W}_\varepsilon$ satisfies
$$\mathcal{W}_\varepsilon\rightarrow \mathcal{W},~ {\rm in ~probability~ in}~C_t\mathcal{H}_0,$$
then,
$$\int_{0}^{t}G_\varepsilon d\mathcal{W}_\varepsilon\rightarrow \int_{0}^{t}G d\mathcal{W}, ~ {\rm in ~probability~ in}~ L_t^2X.$$
\end{lemma}

\section{ Auxiliary results}
In this section, we provide two well-posedness results of the stochastic Stokes equation with the Navier boundary condition, and the Cahn-Hilliard equation with the dynamic boundary condition.

The specific form of the stochastic Stokes equation with additive noise reads as
\begin{eqnarray}\label{Equ1.1*}
\left\{\begin{array}{ll}
\!\!d u-{\rm div}(2D(u))dt+\nabla p dt=h_0 dt+hd\mathcal{W},\\
\!\! \nabla\cdot u=0,
\end{array}\right.
\end{eqnarray}
with the boundary conditions
\begin{align}\label{5.2}
u\cdot n=0,~ 2(D(u)\cdot n)_{\tau}+ u_{\tau}=\ell,
\end{align}
and the initial data $u|_{t=0}=u_0$.
\begin{proposition}\label{pro5.1} Assume that $u_0\in L_\omega^p\mathrm{L}^2_x$ is $\mathcal{F}_0$ measurable random variable, operator $h\in L_\omega^pL_t^2L_2(\mathcal{H};L_x^2)$, processes $h_0\in L_\omega^pL_t^2L_x^2$ and $\ell\in L_\omega^pL^2_tL^\frac{4}{3}(\Gamma)$, then system (\ref{Equ1.1*})-(\ref{5.2}) has a unique global weak solution $u$ such that
$$u\in L^p_\omega( C_t\mathrm{L}^2_x\cap L^2_t\mathrm{H}^1\cap W_t^{1,2}(\mathrm{H}^1)'),$$
satisfying
$$(u,v)+\int_{0}^{t}a_0(u,v)ds=(u_0, v)+\int_{0}^{t}(h,v)d\mathcal{W}+\int_{0}^{t}(\ell,v)_\Gamma ds+\int_{0}^{t}(h_0,v)ds,$$
for all $t\geq 0$, $p\geq 2$ and $v\in \mathrm{H}^1$.
\end{proposition}
\begin{proof} The proof could be implemented by the classical Galerkin approximate scheme, stochastic compactness argument and identifying the limit. Define by $\mathbb{H}_n={\rm span}\{e_1, \cdots, e_n\}$ the finite-dimensional space where the sequence $\{e_k\}_{k\geq 1}$ is the orthonormal basis for $\mathrm{L}_x^2$ of the Stokes operator $A_0$. Define
 $$u_n=P_nu=\sum\limits_{k=1}^{n}(u,e_{k})e_{k},~~~\mbox{for all} ~u\in \mathrm{L}_{x}^2,$$
where the operator $P_n: \mathrm{L}_x^2\rightarrow \mathbb{H}_n$.
Then, there exists a unique approximate solutions sequence to the following finite-dimensional equation
$$(u_n,v)+\int_{0}^{t}a_0(u_n,v)ds=(u_0, v)+\int_{0}^{t}(h,v)d\mathcal{W}+\int_{0}^{t}(\ell,v)_\Gamma ds+\int_{0}^{t}(h_0,v)ds,$$
for any $v\in \mathbb{H}_n$. Using It\^{o}'s formula to the function $\frac{1}{2}\|u_n\|_{\mathrm{L}_x^2}^2$, we have
\begin{align}\label{5.3*}
\begin{split}
\frac{1}{2}\|u_n\|_{\mathrm{L}_x^2}^2+\int_{0}^{t}a_0(u_n,u_n)ds&=\frac{1}{2}\|u_0^n\|_{\mathrm{L}_x^2}^2+\int_{0}^{t}(h,u_n)d\mathcal{W}+\int_{0}^{t}(\ell,u_n)_\Gamma ds\\
&\quad+\frac{1}{2}\int_{0}^{t}\|Ph\|^2_{L_2(\mathcal{H},\mathrm{L}_x^2)}ds+\int_{0}^{t}(h_0,u_n)ds.
\end{split}
\end{align}
The Burkholder-Davis-Gundy inequality yields
\begin{align*}
\begin{split}
\mathbb{E}\left(\sup_{t\in [0,T]}\left|\int_{0}^{t}(h,u_n)d\mathcal{W}\right|\right)^p&\leq \mathbb{E}\left(\int_{0}^{T}(h,u_n)^2dt\right)^\frac{p}{2}\\
&\leq \mathbb{E}\left(\sup_{t\in [0,T]}\|u_n\|_{\mathrm{L}_x^2}^{2p}\right)+\mathbb{E}\left(\int_{0}^{T}\|h\|^2_{L_2(\mathcal{H},L_x^2)}dt\right)^\frac{p}{2}.
\end{split}
\end{align*}
Using the embedding $H^1(\mathcal{D})\hookrightarrow H^\frac{1}{2}(\Gamma)\hookrightarrow L^4_x(\Gamma)$ and H\"{o}lder's inequality, we get
$$\int_{0}^{t}(\ell,u_n)_\Gamma ds\leq \int_{0}^{t}\|\ell\|_{L^\frac{4}{3}(\Gamma)}\|u_n\|_{\mathrm{L}_x^4(\Gamma)}ds\leq \|\ell\|_{L_t^2L^\frac{4}{3}(\Gamma)}\|u_n\|_{L_t^2\mathrm{H}^1}.$$
The last term in \eqref{5.3*} could be easily controlled by using H\"{o}lder's inequality. Taking supremum about $t$ in \eqref{5.3*}, power $p$ and then expectation, using above estimates, and conditions on $l,h$, we have
$$ \mathbb{E}\left(\sup_{t\in [0,T]}\|u_n\|_{\mathrm{L}_x^2}^{2p}\right)+\mathbb{E}\left(\int_{0}^{T}a_0(u_n,u_n)dt\right)^p\leq C,$$
where constant $C$ is independent of $n$. Using above estimate, we could obtain that $u_n\in L^p_\omega W_t^{1,2}(H^1)'$.
Then, we could pass the limit using a stochastic compactness argument. Note that system (\ref{Equ1.1*}) is a linear stochastic evolution equation driven by the additive noise. Therefore, the procedure of identifying the limit could be achieved directly by passing the limit in approximate equation. The specific procedure is even simpler than the argument of  Part 4 in section 3. Here, we omit the details.

 The regularity of solution $u\in L^p_\omega( L^\infty_t\mathrm{L}^2_x\cap L^2_t\mathrm{H}^1)$ follows from the lower-semicontinuous of norm. Then, from equation itself and the regularity, we could further obtain that $u\in L^p_\omega W_t^{1,2}(H^1)'$. Applying \cite[Lemma 1.2, Chapter 3]{tem}, we have the time regularity of $u$ is continuous.

 The uniqueness could be achieved easily due to the linear structure of the system. We complete the proof.
\end{proof}

We proceed to introduce the well-posedness result of the Cahn-Hilliard equation with the dynamic boundary condition which reads as
\begin{align}\label{5.3}
\left\{\begin{array}{ll}
\!\!\partial_t\phi -\Delta\bar{\mu}=\ell_1,\\
\!\!\partial_n \bar{\mu}=0,
\end{array}\right.
\end{align}
with
$$\partial_t\psi+\overline{\mathcal{K}}(\psi)=\ell_3, ~{\rm on} ~\Gamma\times (0,\infty),$$
where $\bar{\mu}:=-\Delta \phi+\delta \partial_t\phi+\ell_2$ for any $\delta\in (0,1)$ and $\overline{\mathcal{K}}(\psi):=A_\tau\psi+\partial_n\phi+\psi$. We focus on the initial data $(\phi,\psi)|_{t=0}=(\phi_0,\psi_0)$.
\begin{proposition}{\rm\!\!\cite[Theorem 4.2]{GMA}} \label{pro5.2} Suppose that $\ell_1\in L_t^2 (H^1)'$ with $\langle \ell_1\rangle=0$, and $\ell_2\in L_t^2L_x^2$, $\ell_3\in L_t^2L^2(\Gamma)$. Moreover, the initial data $(\phi_0,\psi_0)\in V^1$, then system (\ref{5.3}) has a unique global solution $(\phi,\psi)$ such that
$$(\phi,\psi)\in L^\infty_t V^1\cap L^2_tV^2\cap W_t^{1,2}(L_x^2\times L^2(\Gamma)).$$
Moreover, the mass is conserved, i.e., $\langle\phi(t)\rangle=\langle\phi_0\rangle$, and the following estimate holds:
\begin{align}\label{5.4}
&\|(\phi_{\varepsilon,\delta},\psi_{\varepsilon,\delta})\|_{L_t^\infty V^1}^2
+\|\partial_t(\phi_{\varepsilon,\delta},\psi_{\varepsilon,\delta})\|_{L_t^2((H^1)'\times L^2(\Gamma))}^2+\delta\|\partial_t\phi_{\varepsilon,\delta}\|_{L^2_tL_x^2}^2+\|(\phi_{\varepsilon,\delta},\psi_{\varepsilon,\delta})\|_{L_t^2V^2}\nonumber\\
&\quad+\|\bar{\mu}\|_{L_t^2L_x^2}^2+\|\overline{\mathcal{K}}(\psi)\|_{L^2_tL^2_x}^2\nonumber\\
&\leq C\left(\|(\phi_0,\psi_0)\|_{V^1}^2+\|\ell_1\|_{L^2_t(H^1)'}^2+\|\ell_2\|_{L_t^2L_x^2}^2+\|\ell_3\|_{L_t^2L^2(\Gamma)}^2\right).
\end{align}
\end{proposition}

\smallskip
\section*{Acknowledgments}
H. Wang is supported by the National Natural Science Foundation of China (Grant No.11901066), the Natural Science Foundation of Chongqing (Grant No.cstc2019jcyj-msxmX0167) and projects No.2019CDXYST0015 and No.2020CDJQY-A040 supported by the Fundamental Research Funds for the Central Universities.

\end{document}